\documentclass[12pt]{amsart}
\usepackage[utf8]{inputenc}
\usepackage[T1]{fontenc}
\usepackage{lmodern}
\usepackage{amssymb,amsmath,amsthm,fullpage,dsfont}
\usepackage{pdfsync}
\usepackage{enumerate}
\usepackage{pdfsync}
\usepackage{mathrsfs}
\usepackage[bookmarksnumbered,colorlinks,plainpages,backref]{hyperref}
\usepackage[normalem]{ulem}
\usepackage[dvipsnames]{xcolor}
\usepackage{esint}
\usepackage{transparent}

\numberwithin{equation}{section}

\newcommand{\abs}[1]{\left\vert#1\right\vert}

\newcommand{\ovl}[1]{\overline{#1}}

\newcommand{\R}{\mathbb R}

\newcommand{\N}{\mathbb N}
\newcommand{\C}{\mathbb C}

\newcommand{\F}{\mathcal F}
\newcommand{\M}{\mathcal M}

\newcommand{\1}{\mathds{1}}

\DeclareMathOperator{\supp}{supp}

\newcommand{\vp}{\varphi}

\newcommand{\nn}{\nonumber}

\newtheorem{theorem}{Theorem}[section]

\newtheorem{proposition}[theorem]{Proposition}

\newtheorem{lemma}[theorem]{Lemma}

\newtheorem{corollary}[theorem]{Corollary}

\theoremstyle{definition}

\newtheorem{definition}[theorem]{Definition}

\theoremstyle{remark}

\newtheorem{remark}[theorem]{\textbf{Remark}}

\newcommand{\teal}[1]{\textcolor{black}{#1}}

\newcommand{\olive}[1]{\textcolor{blue}{#1}}

\newcommand{\lla}[1]{\left\{ #1\right\} }

\newcommand{\pare}[1]{\left( #1\right)}
\newcommand{\corc}[1]{\left[ #1\right]}
\newcommand{\ip}[1]{ \left\langle #1 \right\rangle }

\newcommand{\norm}[1]{{\| {#1} \|}}

\newcommand{\Norm}[1]{{\left\| {#1} \right\|}}

\newcommand{\ve}{\varepsilon}

\allowdisplaybreaks

\begin{document}

\title{Cancellation conditions and boundedness of Inhomogeneous Calder\'on-Zygmund operators on local Hardy spaces associate with spaces of homogeneous type}%

\author{Joel Coacalle}
\address{Departamento de Computa\c{c}\~ao e Matem\'atica, Universidade S\~ao Paulo, Ribeir\~ao Preto, SP, 14040-901, Brasil}
\email{jcoacalle@usp.br}

\author {Tiago Picon}
\address{Departamento de Computa\c{c}\~ao e Matem\'atica, Universidade S\~ao Paulo, Ribeir\~ao Preto, SP, 14040-901, Brasil}
\email{picon@ffclrp.usp.br}

\author {Claudio Vasconcelos}
\address{Laboratoire de Math\'ematiques d'Orsay, CNRS UMR 8628, Universit\'e Paris-Saclay, B\^atiment 307, 91405 Orsay Cedex, France}
\email{claudio.vasconcelos@alumni.usp.br}



\thanks{The first and the second authors were supported by Funda\c{c}\~ao de Amparo \`a  Pesquisa do Estado de S\~ao Paulo (FAPESP - grant 18/15484-7 and grant 22/02211-8) 
and the second by Conselho Nacional de Desenvolvimento Cient\'ifico e Tecnol\'ogico (CNPq - grant 315478/2021-7)}

\begin{abstract} In this work, we present sufficient cancellation conditions for the boundedness of inhomogeneous Calder\'on-Zygmund type operators on local Hardy spaces defined over spaces of homogeneous type in the sense of Coifman \& Weiss for $ 0<p\leq 1 $. A new approach to atoms and molecules for local Hardy spaces in this setting are introduced with special moment conditions. 
\end{abstract}

\maketitle

\section{Introduction}

The theory of Hardy spaces associated with spaces of homogeneous type was introduced by Coifman \& Weiss in \cite{CoWe77} and it has been extensively studied in several settings and applications. Certainly, there is a vast literature on the subject and any tentative to mention 
it will not be complete.

The spaces of homogeneous type $(X,d,\mu)$ in the sense of Coifman \& Weiss are given by a quasi-metric space $(X,d)$ equipped with a non negative measure $\mu$ satisfying the doubling property:  there exists $ A^\prime>0 $ such that for all $ x\in X $ and $ r>0 $, the control 
\begin{equation}\label{thIIa} 
	\mu(B_d(x,2r))\leq A^\prime \mu(B_d(x,r)).
\end{equation}
holds. 
Several examples of spaces of homogeneous type can be found at \cite{AM,CoWe77}. The authors in \cite{CoWe77} introduced an atomic Hardy space defined on $(X,d,\mu)$ for $0<p\leq 1$, denoted in this work by  $H^{p}_{cw}(X)$, consisting of linear functionals on the dual of Lipschitz space $\mathcal{L}_{1/p-1}(X)$ admitting an atomic decomposition given by
\begin{equation}\label{Hp}
 f=\sum_{j=1}^{\infty} \lambda_{j}a_{j}, \quad \quad \text{in} \,\, \mathcal{L}^{\ast}_{1/p-1}(X)
 \end{equation}
where $\sum_{j=1}^{\infty}|\lambda_{j}|^{p}<\infty$ and $a_{j}$ are measurable functions whose support is contained in a ball $B_j:=B(x_j,r_j)$ and in addition they satisfy the size control $\|a_j\|_{L^{q}}\leq \mu(B_j)^{\frac{1}{q}-\frac{1}{p}}$ for some $1\leq q \leq \infty$ with $p<q$ and the vanishing moment condition $\int_{X}a_{j} \,d\mu=0$. Such functions will be called $(p,q)$-atoms.
The functional $\|f\|_{H^{p}_{cw}}:=\inf\left\{ \left( \sum_{j=1}^{\infty}|\lambda_{j}|^{p} \right)^{1/p} \right\}$, where the infimum is taken over all decompositions satisfying \eqref{Hp}, defines a quasi-norm in $H^{p}_{cw}(X)$ and consequently makes the space complete. We remark that the definition of $H^{p}_{cw}(X)$ is independent of $1\leq q <\infty$ (see \cite[Theorem A]{CoWe77}).
Based on  the classical theory of Hardy spaces on Euclidean spaces, they showed that some singular integral operators naturally bounded on $L^{2}(X)$ ({e.g. the} Riesz transform) have bounded extension from $H^{p}_{cw}(X)$ to $L^{p}(X)$. Several questions have been left open in this paper, in particular regarding a possible maximal characterization of $H^{p}_{cw}(X)$ without any additional geometric assumptions on the tern $(X,d,\mu)$, such as the reverse doubling property on $\mu$ or requiring that $d$ is a metric.
     
With the advent of the orthonormal wavelet basis on $(X,d,\mu)$ due to Auscher and Hyt\"onen in \cite{AusHyto13}, other characterizations of Hardy spaces defined on spaces of homogeneous type were developed along with several applications. In \cite{HeHanLiLYY19}, He et al. presented a complete answer for the mentioned question establishing some characterizations in terms of radial, \textit{grand} and non-tangential maximal functions, wavelet and Littlewood-Paley functions. As an application, they extended to this setting the classical criteria, valid in the Euclidean framework, of atoms and molecules for showing the boundedness of sublinear operators on $H^{p}(X)$, which represents the Hardy space refereed at \cite[pp. 2209]{HeHanLiLYY19} for  $\frac{\gamma}{\gamma+\eta}<p\leq 1$ or any of its equivalent maximal representations presented at \cite[Theorem 3.5]{HeHanLiLYY19}. Here $\gamma$ is the upper dimension on $(X,d,\mu)$ defined at \eqref{upper} below and $\eta$ is the regularity of the splines defined in \cite{AusHyto13}. 
In the same paper, the authors proved that the atomic space $H^{p}_{cw}(X)$ coincides with $H^{p}(X)$ with equivalence between quasi-norms for  $\frac{\gamma}{\gamma+\eta}<p\leq 1$ and coincides with $L^{p}(X)$  when $1<p<\infty$. For $p=1$, the space is strictly contained in $L^{1}(X)$.   

Both atomic and maximal representations are fundamental to extend the boundedness of Calder\'on-Zygmund operators on $H^{p}(X)$, that we will describe in sequel. Let $V_s(X):= C_b^{s}(X)$ the space of $s$-H\"older regular functions with bounded support  equipped with the usual topology, where $s \in (0,\eta)$.  
Following the \cite[Definition 12.1]{AusHyto13}, we say a linear and continuous operator $R: V_{s}(X) \rightarrow V^{\ast}_{s}(X)$ is associated to a Calder\'on-Zygmund kernel
of order $s$ if there exists a distributional kernel $K$ satisfying:
	\begin{itemize}
		\item[(i)] for every $ x,y\in X $ with $ x\neq y $, there exists $C_{1}>0$ such that
\begin{equation}\label{stonga}
		\abs{K(x,y)}\leq C_{1} \, \frac{1}{V(x,y)}, \quad \text{with } V(x,y):= \mu(B(x,d(x,y)));
\end{equation}
		\item[(ii)] for every $ x,y,z \in X $ with $ (2A_0)d(y,z)\leq d(x,z)$ and $ x\neq y $  ($A_{0}$ defined in \eqref{disntance}), there exists $C_{2}>0$ such that
\begin{equation}\label{czstd1}
		\abs{K(x,y)-K(x,z)}+ \abs{K(y,x)-K(z,x)}\leq C_{2}\corc{\frac{d(y,z)}{d(x,z)} }^s \frac{1}{V(x,z)};
\end{equation}
	\item[(iii)] for any $f \in V_{s}(X)$, the operator $R$ has the representation 
	\begin{equation}\label{czs}
		Rf(x)=\int K(x,y)f(y)d\mu(y), \quad \quad \text{for } x \notin \supp\,(f). 
	\end{equation}
	\end{itemize} 
An operator $ R $ is called a Calder\'on-Zygmund operator of type $s$ if $R$ is associated to a Calder\'on-Zygmund kernel of order $s$ 
and it is bounded on $L^{2}(X)$. Operators of this type were characterized by Auscher and Hyt\"onen  in \cite[Theorem 12.2]{AusHyto13}, so called $T(1)$-theorem for spaces  of homogeneous type in the sense of Coifman \& Weiss. Estimates for Calder\'on-Zygmund operators defined on spaces of homogeneous type have been extensively studied in the literature. For instance, it has been known that such operators are bounded in $L^p(X)$ when $1<p<\infty$ and satisfy a weak $L^1(X)$-estimate (see for instance \cite[Theorem 1.10]{DengHan08}).

The following extension for Hardy spaces was stated by Han et al. in \cite[Theorem 1.3]{HanHanLi16}: 

\begin{theorem}[\cite{HanHanLi16}]\label{teo1a} Let $s\in (0,\eta]$ 
and $R$ be a Calder\'on-Zygmund operator of type $s$. Then R extends to a
bounded operator on $H^{p}(X)$ for $\frac{\gamma}{\gamma+s}<p\leq 1$ if and only if $R^{\ast}(1) = 0$.
\end{theorem}

\noindent The condition $R^{\ast}(1)$ is understood in the following distributional sense as
\begin{equation*}
\langle R^{\ast}(1),f  \rangle:=\int_{X}Rf(x)d\mu(x), \quad \quad \forall f \in H^{p}(X) \cap L^{2}(X).
\end{equation*}
The method to prove the boundedness in Theorem \ref{teo1a} is given by the classical property that $R$ maps atoms into molecules. The cancellation condition $R^{*}(1)=0$ is necessary since from \cite[Proposition 3.3]{HanHanLi16} if $Rf \in L^2(X)\cap H^{p}(X)$ for  $\frac{\gamma}{\gamma+\eta}<p\leq 1$  then
$Rf \in L^1(X)$ and $\int_{X}Rfd\mu=0$. In contrast to convolution operators, we point out that, in general, there is no reason to think that non-convolution linear operators as \eqref{czs} preserve vanishing moment conditions.

It is well known in the Euclidean setting that if $f \in (L^1 \cap H^{p})(\R^n)$ then $\int_{\R^n}f(x)dx=0$ and then this fact shows that $H^{p}(\R^n)$ is not closed by multiplication of test functions. Motivated by this, Goldberg in \cite{Goldberg1979} introduced a localizable or non-homogeneous version of Hardy spaces in $\R^n$, which is called {\em local Hardy spaces} and denoted by $h^p(\R^n)$. Moreover, $H^p(\R^n)$ is continuously embedded in $h^{p}(\R^n)$, when $p>1$ we have the equivalence $h^p(\R^n)=L^{p}(\R^n)$ with comparable norms, $h^1(\R^n) \subset L^{1}(\R^n)$ strictly, and the following desired property holds: if $\varphi \in C_{c}^{\infty}(\R^{n})$ and $f \in h^{p}(\R^n)$ then $\varphi f \in h^{p}(\R^n)$. From the comparison between $H^{p}(\R^n)$ and $h^{p}(\R^n)$ (see \cite[Lemma 4]{Goldberg1979}), a natural atomic decomposition for $h^p(\R^n)$ arises, in which the atoms require vanishing moment conditions only when their supports are contained in balls $B$ with radius $r(B) < 1$. For $r(B)\ge1$, any moment conditions are required. Very recently, the authors in \cite{DafniPicon22, DafniPicon23} discussed necessary cancellation conditions on $h^{p}(\R^n)$ motivated by the boundedness of certain linear operators. In particular, they characterized the boundedness of inhomogeneous Calder\'on-Zygmund type operators on $h^{p}(\R^n)$ in terms of local Campanato-Morrey type estimates on $R^{\ast}(x^{\alpha})$ for $|\alpha|\leq n\lfloor \left(\frac{1}{p} -1\right) \rfloor$  (see for instance \cite{Nakai2006} for definition of these spaces on $(X,d,\mu)$). Note that for $\frac{n}{n+1}<p\leq 1$, the condition is exactly on $R^{\ast}(1)$. 

In the setting of spaces of homogeneous type, inhomogeneous Calder\'on-Zygmund operators of type $(\nu,s)$ are operators that satisfy conditions \eqref{czstd1}, \eqref{czs}, where the condition \eqref{stonga} is replaced by strong control
\begin{equation}\label{forteb}
\abs{K(x,y)}\leq C\min\lla{\frac{1}{V(x,y)}, \frac{1}{V(x,y)d(x,y)^\nu} }, \quad \forall \, x\neq y.
\end{equation}					  
for some $\nu>0$. Note that at the Euclidean setting, if we take the canonical distance $d(x,y):=|x-y|$ and $\mu$ the Lebesgue measure, then $V(x,y) \approx |x-y|^{n}$. Examples of inhomogeneous Calder\'on-Zygmund operators are given by pseudodifferential operators $OpS^{0}_{1,0}(\R^n)$ (see \cite{DafniPicon22, DafniPicon23}).   

Following the scope of  Goldberg's atomic decomposition for $h^p(\R^n)$ and the construction of $H^{p}_{cw}(X)$, an atomic version of local Hardy spaces, denoted here by $h^{p}_{cw}(X)$, with an appropriate convergence in local Lipschitz spaces, may be obtained in terms of $(p,q)$-atoms where the vanishing moment condition, i.e. $\int_{X}a(x)d\mu=0$, is required only for atoms supported on balls $B(x_{B},r_{B})$ with $r_{B}<1$. 
We denote these type of atoms by local $(p,q)$-atoms. Analogous to Coifman \& Weiss in \cite{CoWe77}, if $R$ is an inhomogeneous Calder\'on-Zygmund operator of order $(\nu,s)$ (see Definition \ref{def:CZO1} below), then $R$ can be extended from $h^{p}_{cw}(X)$ to $L^{p}(X)$ for all $\frac{\gamma}{\gamma+\min\left\{\nu,s\right\}}<p< 1$.
We state this result at Theorem \ref{teo:boundextenCZO1} below. 

A complete study of local Hardy spaces on spaces of homogeneous type was originally presented by Dafni et al. in \cite{DafniMomYue16} when $ p=1 $. In the cited work, the authors defined a local Hardy space, denoted here by $h_{g}^{1}(X)$, via maximal function approach and gave an atomic characterization of this space in terms of  atoms with a special cancellation condition that recovers the natural atomic space given by $h^{1}_{cw}(X)$ defined by action of local $(1,q)$-atoms for $1<q\leq \infty$ analogous to \eqref{Hp} with convergence in $bmo(X)$. Using the orthonormal wavelet basis, in the same spirit of $H^{p}(X)$, He et al. in \cite{HeYangWen21} presented a new maximal definition of local Hardy spaces for $0<p\leq 1$, denoted here by $h^{p}(X)$, in which the associated atomic space (with convergence in an appropriate set of distributions) coincides with $h^{p}_{cw}(X)$. 

The aim of this paper is to present sufficient conditions on $R^{\ast}(1)$ in order to obtain the boundedness of inhomogeneous Calder\'on-Zygmund operators of type $(\nu,s)$ on local Hardy spaces defined on spaces of homogeneous type in the sense of Coifman \& Weiss. Our method is based on a new atomic local Hardy space, denoted by $h^{p}_{\#}(X)$, described in terms of atoms and molecules satisfying approximate moment conditions also called inhomogeneous cancellation conditions, in the same sense of Dafni et al. in \cite{DafniPicon22, DafniPicon23, DafniMomYue16}. Such atomic and molecular structures are fundamental to capture the cancellation expected on $Ra$ in local Hardy spaces, where $a$ is a local $(p,q)$-atom. Moreover, using the maximal characterization of $h^{p}(X)$ we compare the atomic spaces $h^{p}_{cw}(X)$ and $h^{p}_{\#}(X)$ and under a natural assumption on $R^{\ast}(1)$, we  obtain the boundedness of inhomogeneous Calder\'on-Zygmund operators on local Hardy spaces $h^{p}(X)$.       

Our main result in this paper is the following:

\begin{theorem} \label{teo1.2}
Let $0<p<1$ and  $R$ be an inhomogeneous Calder\'on-Zygmund operator of order $ (\nu,s) $. 
If there exists $C>0$ such that for any ball $B(x_{B},r_{B})\subset X$ with $r_{B}<1$ we have that $f:=R^{\ast}(1)$ satisfies
\begin{equation}\label{1.8}
\left(\fint_{B}|f-f_{B}|^{2} d\mu\right)^{1/2} \leq C \mu(B(x_{B},1))^{1-\frac{1}{p}} \mu(B(x_{B},r_{B}))^{{\frac{1}{p}-1}}, 
\end{equation}	
where $\displaystyle f_B= \fint_{B}f\, d\mu$, 
then $R$ can be extended to a bounded operator from $h^{p}(X)$ to itself provided that $\min \left\{\nu,s\right\}>\gamma\left(\frac{1}{p}-1\right)$.
\end{theorem}

In the Proposition \ref{prop:well-def-T*1}, we show that $R^{\ast}(1) \in L^{2}_{loc}(X)$ and then the condition given in \eqref{1.8} is well defined. Estimates as \eqref{1.8} for any balls define a type of generalized Campanato space, see for instance \cite{Nakai2006}. The previous theorem is presented as Theorem \ref{cor:last} below, emphasizing the space of distributions where $h^{p}(X)$ is defined. The key of the proof is stated at Theorem  \ref{teo:boundextenCZO}, where \eqref{1.8} and some extra condition on atomic space $h^{p}_{cw}(X)$ are sufficient to show that  the operator can be extended from 
$ h^{p}_{cw}(X)$ to $ h^{p}_{\#}(X)$.    

Our second goal is a version of the Theorem \ref{teo1.2} for $p=1$, where a stronger cancellation condition is assumed.  
\begin{theorem}\label{teo:boundextenCZOp1ga}
	Let $ R $ be an inhomogeneous Calder\'on-Zygmund operator of order $ (\nu,s)$. If there exists $C>0$ such that for any ball $B:=B(x_{B},r_{B})\subset X$ with $r_{B}<1$ we have that $f:=R^{\ast}(1)$ satisfies
	\begin{equation}\label{eq:atomtomol3p1ga}
		\bigg( \fint_B \abs{f-f_B}^2 d\mu \bigg)^{1/2} \leq C \frac{2}{\log(2+{1}/{r_B})},
	\end{equation}
then the operator $ R $ can be extended as a linear bounded operator on $ h^1_g(X) $.
\end{theorem}
It is not difficult to see that the condition \eqref{eq:atomtomol3p1ga} is stronger in comparison to \eqref{1.8} when $ p=1 $. 
We point out the result can be stated replacing $h^1_g(X)$ by $ h^1(X) \cap L^1_{loc}(X)$, since the spaces {are equivalents} with comparable norms (see Proposition \ref{propf}).  
We refer to the works   \cite{DafniPicon22, DafniPicon23}
for a complete discussion of cancellation conditions on $R^{\ast}(1)$ in order to obtain the boundedness of inhomogeneous {Calder\'on}-Zygmund {operators on} $h^p(\R^{n})$.

The organization of the paper is as follows. In Section \ref{preliminares}, we present a brief discussion on spaces of homogeneous type in the sense of Coifman \& Weiss, including a basic material on local Lipschitz spaces. In Section \ref{sec:approxatom}, we introduce a new atomic local Hardy space $h^{p}_{\#}(X)$, extending the notion of $h^{p}_{cw}(X)$, in which local Goldberg's atoms are replaced by approximate atoms satisfying inhomogeneous cancellation conditions. The molecular decomposition and the dual characterization of {$h^{p}_{\#}(X)$} are also presented. The Section \ref{sec:rel} is devoted to {discussing} the relation between $h^{p}_{\#}(X) $ and the local Hardy spaces $h^p(X)$ introduced by He et al. in \cite{HeYangWen21}. In Section \ref{sec:app}, we present boundedness results for inhomogeneous Calder\'on-Zygmund operators on local Hardy spaces, including the proof of Theorem \ref{teo1.2} at the Subsection \ref{sec:5.1} and on Lebesgue spaces in Subsection \ref{sec:5.2}. Finally in Section \ref{sec:p=1}, we present the proof of Theorem \ref{teo:boundextenCZOp1ga} and a formal relation between $ h^1_g(X) $ and $ h^1(X) $.

\section{Preliminaries}\label{preliminares}

\subsection{Spaces of Homogeneous type} Let $ X $ be a nonempty set and $ d:X\times X\to \R_{+} $ a nonnegative function in which there exists $ A_0>0 $ such that for any $ x,y,z\in X $ we have: $ d(x,y) =0\Leftrightarrow  x=y $, $ d(x,y)=d(y,x) $ and
\begin{equation}\label{disntance}
d(x,y) \leq A_0 (d(x,z)+d(z,y)).
\end{equation}
The function $ d $ is called a \emph{quasi-metric} and the pair $ (X,d) $ a \emph{quasi-metric space}. It has to be noted that $ A_0\geq 1 $ and if $ A_0=1 $, then $ d $ defines a metric on $ X $.\\

A quasi-metric $ d $ defines a natural topology $ \tau_d $ in $ X $ where the $ d $-balls defined by $B_d(x,r):=\lla{y\in X:\ d(x,y)<r}$ form a basis. More precisely, a subset $ U \subseteq X $ is an open set if for every $ x\in U $, there exists $ r>0 $ such that $ B_d(x,r) \subset U$. In general, it is not true that $ B_d(x,r) $ is an open set. When there is no risk of confusion, we omit the use of the subscript $ d $ in the notation of balls and we will always denote by $B:=B(x_B,r_B)$ the ball in $X$ centered in $x_B$ with radius $r_B$. \\

A \emph{space of homogeneous type} $ (X,d,\mu) $ is a quasi-metric space $ (X,d) $ along with a nonnegative measure $ \mu $ defined on the $ \sigma $-algebra of subsets of $ X $ which contains all $ d $-balls and the $ \sigma $-algebra generated by $ \tau_d $, such that $\mu$ satisfies the doubling property, that is, there exists $ A^\prime>0 $ such that for all $ x\in X $ and $ r>0 $ we have $ \mu(B_d(x,r)) <\infty$ and 
\begin{equation}\label{thII} 
	\mu(B_d(x,2r))\leq A^\prime \mu(B_d(x,r)).
\end{equation}
In this case, the measure $ \mu $ is called a \emph{doubling measure}. We may also use the simplified notation $\mu(\cdot):=|\cdot|$ along the paper. To avoid any confusion with the Lebesgue measure, we will always use $\mathcal{L}(\cdot)$ to denote this later. \\

In this work we will always assume that $\mu(X)>0 $. As a consequence, since $ X $ can be exhausted by $d$-balls, turning $ (X,\mu) $ into a $ \sigma $-finite space, we get from the doubling condition that $ \mu(B_d(x,r))>0 $ for all $ x\in X $ and $ r>0 $. In particular, if for some $ y\in X $ there exists $ r_y>0 $ such that $ B_d(y,r_y)=\lla{y}$, then $ \mu(\lla{y})>0 $. Moreover, as shown in \cite[Theorem 1.]{MaSe79}, the set $ M=\lla{x\in X: \mu(\lla{x})>0} $ is countable and it is equal to $ I=\lla{ x\in X : B_d(x,r_x)=\lla{x} \text{ for some } r_x>0} $. \\

If for any $ \lambda\in [1,\infty) $, there exists $\gamma>0$ such that
\begin{equation}\label{upper}
	\mu(B(x,\lambda r))\leq A^\prime \lambda^\gamma \mu(B(x,r)),
\end{equation}
for all $ x\in X$ and $ r>0$, we call this constant $\gamma$ the \emph{upper dimension of $ X $}. Note that if the measure $\mu$ is doubling, then condition \eqref{upper} holds with $\gamma=\log_2 A^\prime $.\\

	We denote the \emph{volume functions}  by $V_r(x)=|B(x,r)|$ and $V(x,y)=|B(x,d(x,y))|$.	The standard maximal function is defined as
	\[
	\mathcal{M}f(x):=\sup_{r>0} \frac{1}{\abs{B(x,r)}} \int_{B(x,r)} \abs{f} \,d\mu.
	\]
	It is well known that $ \mathcal{M} $ is bounded on $ L^p(X) $ with $ p\in (1,\infty] $ and bounded from $ L^1(X) $ to $ L^{1,\infty}(X) $ (see \cite[pp. 71-72]{CoWe71}). Next we state some well known results of the volume and maximal functions (see for instance Lemma 2.2 in \cite{HeYangWen21} or \cite{HeHanLiLYY19}).
	\begin{proposition}\label{propfvfm1}\label{prop:fv}
		Let $ (X,d,\mu) $ be a space of homogeneous type. For any $ x,y\in X $ and $ r\in (0,\infty) $ one has
		\begin{enumerate}[i)]
			\item\label{eq:fv1} $ V(x,y)\approx V(y,x) $ and
			\[
			V_r(x)+V_r(y)+V(x,y) \approx V_r(x)+V(x,y) \approx V_r(y)+V(x,y) \approx |B(x,r+d(x,y))|,
			\]
			where the constants in these equivalences are independent of $ x, y $ and $ r $.
			\item\label{eq:fv2} If $ a>0 $ and $ \delta>0 $, we have
			\[
			\int_{d(x,y)\leq \delta} \frac{d(x,y)^a}{V(x,y)}d\mu(y) \leq C_a \delta^a \quad \text{and} \quad  \int_{d(x,y)\geq \delta} \frac{d(x,y)^{-a}}{V(x,y)}d\mu(y) \leq C_a \delta^{-a}
			\]
			where $ C_a>0 $ is independent of $ x $ and $ \delta $.
			\item\label{eq:fv3}
				If $ a>0 $
				\[
				\int_X \frac{1}{V_{r}(x)+V(x,y)}\frac{r^a}{(r+d(x,y))^a}d\mu(y)\leq C_a
				\]
				uniformly in $ x\in X $ and $ r>0 $.
			\item\label{eq:fv4} There exists a constant $ C>0 $ such that for every
			$ x\in X $, $ r\in (0,\infty) $ and $ f\in L^1_{loc}(X) $, we have
			\[
			\int \frac{1}{V_r(x)+V(x,y)}\corc{ \frac{r}{r+d(x,y)} }^\theta \abs{f(y)}d\mu(y)\leq C \mathcal{M}f(x),
			\]
			where $ C$ does not depend of $ x $, $ r $ or $ f $.
		\end{enumerate}
	\end{proposition}

\subsection{Local Lipschitz spaces}
We denote by $L^\infty_{loc}(X) $ the space of all measurable functions $ f $ defined in $ X $ such that $ f\in L^ \infty(B) $ for all balls $ B \subset X $. Let $ T >0$ be a fixed constant. For each ball $ B=B(x_B,r_B) \subset X $, $ \alpha >0$, and measurable function $f$, we define the functional 
\begin{equation*}
	\mathfrak{N}_{\alpha,\,T}^B(f):=
	\begin{cases}
		\displaystyle{\frac{1}{|B|^\alpha}\sup_{x,y\,\in\, B} \abs{f(x)-f(y)}},  & \text{if } r_B< T; \\
		\displaystyle{\frac{1}{\abs{B}^{\alpha}} \, \norm{f}_{L^\infty(B)}}, & \text{if }r_B\geq T. \\
	\end{cases}
\end{equation*}
The \emph{local Lipschitz space} $ \ell_{\alpha,\,T}(X) $ is defined to be the set of functions $ f\in L^\infty_{loc}(X) $ such that
\[
\norm{f}_{\ell_{\alpha,T}}:= \sup_{B\subset X} \mathfrak{N}_{\alpha,\,T}^B(f)<\infty . 
\]

These spaces correspond to the local version of the classical \emph{Lipschitz spaces}, that we define in the sequel (see \cite{CoWe77} for more details). For any measurable function $f$ defined on $X$, let
\begin{equation}\label{eq:lipspace}
	\mathfrak{N}^{B}_{\alpha}(f):= \sup_{x,y \,\in\, B} \frac{\abs{f(x)-f(y)}}{\abs{B}^\alpha},
\end{equation}
and the space
$
\mathcal{L}_{\alpha}(X) = \left\{ f:X \rightarrow \C: \ \  \displaystyle \sup_{B \subset X} \mathfrak{N}^{B}_{\alpha}(f) < \infty \right\}
$
equipped with the norm
$$
\norm{ f }_{\mathcal{L}_{\alpha}} := \begin{cases}
	\mathfrak{N}^{B}_{\alpha}(f),  & \text{if } \mu(X)=\infty; \\
	\displaystyle\mathfrak{N}^{B}_{\alpha}(f)+\abs{\int_X fd\mu }, & \text{if } \mu(X)=1. \\
\end{cases}
$$
In fact, if $ \mu(X)=\infty $ then the functional $ \norm{ \cdot }_{\mathcal{L}_{\alpha}} $ does not define a norm in $ \mathcal{L}_{\alpha}(X) $. In this case, we may redefine the space taking the quotient by constant functions. It follows from definition that $\norm{f}_{\mathcal{L}_{\alpha}} \leq 3 \norm{f}_{\ell_{\alpha,T}}$. This shows the continuous inclusion $\ell_{\alpha,T}(X) \hookrightarrow \mathcal{L}_\alpha(X)$ holds for every $\alpha>0$ and $T>0$. As pointed out in \cite[pp. 191]{DafniMomYue16} for the space $bmo(X)$, if $T=\infty$ or $T>diam(X)$, we immediately get that $\mathcal{L}_{\alpha}(X)=\ell_{\alpha,\,T}(X)$. So, we may assume that $T<diam(X)$. \\

The motivation behind the definition of the local Lipschitz space as above is that as shown in \cite[Remark 7.2 and Proposition 7.3]{HeYangWen21}, $\ell_{1/p-1, \, 1}(\R^n)= {\Lambda}_{n(1/p-1)}(\R^n)$ for $n/(n+1)<p<1$, where $\Lambda_{\alpha}(\R^n)$ denotes the non-homogeneous Lipschitz space, defined as the set of measurable functions $f\in L^{\infty}(\R^{n})$ such that $$\|f\|_{\Lambda_{\alpha}}:=\|f\|_{L^{\infty}}+\sup_{x\neq y}\frac{|f(x)-f(y)|}{|x-y|^{\alpha}}<\infty.$$ 

\begin{remark}\label{rem:ell}
We claim that if $ 0<T<T'$, then $ \ell_{\alpha,T}(X)=\ell_{\alpha,T'}(X)$ with comparable norms, that is,
		\begin{equation}\label{eq:ellrel}
		{\frac13} \norm{f}_{\ell_{\alpha,T'}}
			\leq \norm{f}_{\ell_{\alpha,T}} \leq
			3\corc{A'\pare{\frac{T'}T}^\gamma}^{\alpha}  \norm{f}_{\ell_{\alpha,T'}}.
		\end{equation}
In fact, let $ f\in \ell_{\alpha,T}(X) $. Clearly, we have 
\begin{equation}\nonumber
\sup_{r_B<T} \frac1{\abs{B}^{\alpha}}\sup_{x,y\in B}\abs{f(x)-f(y)} \leq 
\sup_{r_B<T'} \frac1{\abs{B}^{\alpha}}\sup_{x,y\in B}\abs{f(x)-f(y)} \leq  \norm{f}_{\ell_{\alpha,T'}} 
\end{equation}	
Now, if $T\leq r_B < T' $ and $ x\in B $ then we may control
		\begin{align*}
			\abs{f(x)} \leq \abs{f(x)-f(x_B)}+\abs{f(x_B)} 
			&\leq  \norm{f}_{\ell_{\alpha,T'}} \abs{B}^\alpha +  \norm{f}_{\ell_{\alpha,T'}} \abs{B(x_B,T')}^\alpha \\
			&\leq \norm{f}_{\ell_{\alpha,T'}} \abs{B}^\alpha + \norm{f}_{\ell_{\alpha,T'}}\corc{A'\pare{\frac{T'}{r_B}}^\gamma}^\alpha \abs{ B }^\alpha \\
			&\leq 2\corc{A'\pare{\frac{T'}{T}}^\gamma}^\alpha \norm{f}_{\ell_{\alpha,T'}} \abs{B}^\alpha.
		\end{align*}
Then
\begin{align*}
\sup_{r_B\geq {T}} \frac1{\abs{B}^{\alpha}}\norm{f}_{L^\infty(B)} & \leq \sup_{r_B\geq {T'}} \frac1{\abs{B}^{\alpha}}\norm{f}_{L^\infty(B)} + \sup_{ T \geq r_B\geq {T'}} \frac1{\abs{B}^{\alpha}}\norm{f}_{L^\infty(B)} \\
& \leq 3\corc{A'\pare{\frac{T'}{T}}^\gamma}^\alpha \norm{f}_{\ell_{\alpha,T'}}.
\end{align*}	
Summarizing 	
\begin{equation*}
\norm{f}_{\ell_{\alpha,T}} \leq
			3\corc{A'\pare{\frac{T'}T}^\gamma}^{\alpha}  \norm{f}_{\ell_{\alpha,T'}}.
\end{equation*}
The comparison $\norm{\cdot}_{\ell_{\alpha,T'}} \leq 3 \norm{\cdot}_{\ell_{\alpha,T}}$ follows \textit{bis in idem} as before and it will be omitted. 
\end{remark}

\section{Atomic local Hardy spaces $ h^{p}_\#(X) $}\label{sec:approxatom}

In what follows, we assume that $ (X,d,\mu) $ is a space of homogeneous type and $T>0$ is fixed.

\subsection{Approximate atoms}

\begin{definition} \label{def:pq-function}
Let $0<p<1\leq q \leq \infty$. We say that a $ \mu $-measurable function $ a $ is a \textit{$(p,q,T)$-approximate atom} if it satisfies:
\begin{enumerate}[(i)]
	\item (\emph{Support condition}) There exist $ x_B\in X $ and $ r_B>0 $ such that $ \supp(a)\subset B(x_B,r_B) $;
	\item (\emph{Size condition}) $ \norm{a}_{ {L^q}} \leq {|B(x_B,r_B)|}^{\frac1q-\frac1p}$;
	\item\label{item:pqfunctiii} (\emph{Moment condition}) 
\begin{equation}\label{momento}
	\abs{ \int a \, d\mu }\leq {|B(x_B,T)|} ^{1-\frac1p}.\\
\end{equation}
\end{enumerate}
\end{definition}

\begin{remark}\label{rem:approx_atoms} \textcolor{white}{.}
	\begin{itemize}
		\item[(i)] Condition \eqref{momento} is a local one, that is, from the support and size assumptions we have 
		$$ \abs{ \int a \, d\mu }\leq  \|a\|_{L^{q}} |B(x_{B},r_{B})|^{\frac{1}{q'}}\leq |B(x_{B},r_{B})|^{1-\frac{1}{p}}  $$
		and clearly the moment condition is immediately satisfied for $r_{B}\geq T$. In this sense, this parameter $T$ can be seen as the localization of the atoms, since the moment condition is actually only required when $r_B<T$.
		\item[(ii)] $(p,q,T)$-approximate atoms are comparable for different values of $T$, that means, the decay of moment condition \eqref{momento} is comparable for different values of $T$, since for $T<T'$ we have
		$$ \abs{B(x_B,T)} \leq \abs{B(x_B,T')}\leq A^\prime \pare{\frac{T'}{T}}^\gamma \abs{B(x_B,T)}. $$  
	\end{itemize}
\end{remark}

When condition (iii) is replaced by a \textit{local vanishing moment condition}, that is,
$$
\text{(iii)'} \quad \int a \, d\mu =0, \quad \text{if } r_B<T,
$$
we say the function $a$ is a local $(p,q,T)$-atom. These atoms correspond to the local $(p,q)$-atoms defined by Goldberg \cite{Goldberg1979} in the context of $\R^n$ and were naturally extended in \cite[Definition 4.1]{HeYangWen21} for spaces of homogeneous type (both for $T=1$). If instead of (iii)', $a$ satisfies a \textit{global vanishing moment condition}, that is 
$$
 \quad \int a \, d\mu =0,
$$
then $a$ is called a $(p,q)$-atom. These atoms were considered in \cite[pp. 591]{CoWe77} to define the atomic Hardy spaces $H^{p}_{cw}(X)$ over spaces of homogeneous type. \\

We should also mention that for the case $p=1$ and $q>1$, Dafni et al. considered in \cite[Definition 7.3]{DafniMomYue16} atoms with the following approximate moment condition
$$
\abs{ \int a \, d\mu }\leq \frac{2}{\log(2+T/r_B)}.
$$
An approach of approximate moment conditions for atoms in $h^{p}(\R^n)$ was recently presented by the second and third authors in \cite{DafniPicon22} and \cite{DafniPicon23} (see also the previous works \cite{DafniThesis93,DL19}). We will discuss more about the case $p=1$ in Section \ref{sec:p=1}.

\begin{proposition}\label{prop:pqfunct1}
	Let $0<p<1\leq q \leq \infty$. Then, any $ (p,q,T) $-approximate atom defines a continuous linear functional on $ \ell_{1/p-1,T}(X) $, and its dual $ \ell^{\ast}_{1/p-1,T}$-norm does not exceed 2.
\end{proposition}

\begin{proof}
	Let $ a $ be a $ (p,q,T) $-approximate atom supported on a ball $ B=B(x_B,r_B) $ and $ f\in \ell_{1/p-1,T}(X)$. If $ r_B<T $, then by the support, size and moment condition of $ a $, we have
	\begin{align*}
		\abs{\int a(x)f(x) d\mu(x)}&\leq \int_{B} \abs{a(x)}\abs{ f(x)-f(x_B) }d\mu(x)+ \abs{f(x_B)} \abs{ \int_B a(x)d\mu(x) }\\
		& \leq 	\mathfrak{N}_{1/p-1,T}^B(f) \abs{B}^{\frac1p-1} \norm{a}_{L^1} + \norm{f}_{L^\infty(B(x_B,T))} \abs{B(x_B,T)}^{1-\frac1p}\\
		&\leq \mathfrak{N}_{1/p-1,T}^B(f)+\mathfrak{N}_{1/p-1,T}^{B(x_B,T)}(f)\leq  2\norm{f}_{\ell_{1/p-1,T}}.
	\end{align*}
	If $ r_B \geq T$, from the support and size conditions of $ a $, we obtain
	\begin{align*}
		\abs{\int afd\mu}\leq \norm{ f }_{L^\infty(B)} \int \abs{ a } d\mu \leq \norm{ f }_{L^\infty(B)}\abs{ B }^{1-\frac1p}\leq \mathfrak{N}_{1/p-1,T}^{B}(f)\leq \norm{f}_{\ell_{1/p-1,T}}. 
	\end{align*}
	Then, the mapping $ f \mapsto \int_X a f d\mu $ is a continuous linear functional on $ \ell_{1/p-1,T}(X) $, with norm not exceeding $ 2 $.
\end{proof}

\begin{proposition}\label{prop:414}
	Let $0<p<1\leq q\leq \infty$, $ \lla{a_j}_{j \in \N}$ be a sequence of $ (p,q,T) $-approximate atoms and $ \lla{\lambda_j}_j \subset \C$ such that $\displaystyle \sum_{j=1}^{\infty} \abs{\lambda_j}^p <\infty $. Then the series $\displaystyle \sum_{j=1}^{\infty} \lambda_j a_j $ converges in $ \ell^{\ast}_{1/p-1,T} $-norm to a distribution $ g \in \ell^{\ast}_{1/p-1,T}(X) $ such that
	\begin{equation}\label{eq:401}
		\norm{g}_{\ell^{\ast}_{1/p-1,T}}\leq 2 \, \pare{\sum_{j=1}^{\infty} \abs{\lambda_j}^p}^{1/p},
	\end{equation}
	where $ C>0 $ is a constant independent of $ g $, $ a_j $ and $ \lambda_j $. 
\end{proposition}

\begin{proof}
	Let $\{a_j\}_{j \in \N}$ be a sequence of $(p,q,T)$-approximate atoms. For $ m,n\in \N $ with $ n<m $ and any $ \varphi \in \ell_{1/p-1,T}(X) $ such that $ \norm{\varphi}_{\ell_{1/p-1,T}}\leq 1 $, by the same argument as in the proof of Proposition \ref{prop:pqfunct1} we have
	\begin{align}
		\abs{\ip{ \sum_{j=n+1}^{m} \lambda_j a_j, \varphi } }&\leq \sum_{j=n+1}^{m}\abs{\lambda_j}\abs{ \ip{a_j,\varphi} }\leq \sum_{j=n+1}^{m} \abs{\lambda_j} 2 \norm{ \varphi }_{\ell_{1/p-1,T}}
		\leq 2 \pare{\sum_{j=n+1}^{m} \abs{\lambda_j}^{p}}^{1/p}. \label{eq:400}
	\end{align}
Then, the sequence of partial sums of $ \sum_j \lambda_j a_j $ is a Cauchy sequence in $ \ell^{\ast}_{1/p-1,T}(X)$, which is a Banach space and so, $ \sum_j \lambda_j a_j $ converges to some $ g\in \ell^{\ast}_{1/p-1,T}(X) $. Moreover, from \eqref{eq:400} we obtain the desired estimate \eqref{eq:401}.
\end{proof}

\subsection{Local atomic Hardy spaces}

We define $ h^{p,q}_\#(X)$ consisting of elements $ g \in \ell^{\ast}_{1/p-1,T}(X) $ for which there exist a sequence $ \lla{a_j}_j $ of $ (p,q,T) $-approximate atoms and a sequence $ \lla{\lambda_j}_j\in \ell^p(\C) $ such that 
\begin{equation}\label{eq:pqatom}
	g=\sum_{j=0}^{\infty} \lambda_j a_j, \quad \text{in } \ell^{\ast}_{1/p-1,T}(X),
\end{equation}
that means $\displaystyle \ip{g,\vp}=\sum_{j=0}^{\infty} \lambda_j \int_X a_j\, \vp \, d\mu$
for all $ \vp \in \ell_{1/p-1,T}(X) $. We refer to the sum in \eqref{eq:pqatom} as an \emph{atomic decomposition} in terms of $ (p,q,T) $-approximate atoms of $ g $.\\

We define 
\[
\norm{g}_{p,q}:=\inf \left\{ \pare{ \sum_j \abs{\lambda_j}^p }^{1/p} \right\},
\]
where the infimum is taken over all such atomic representations of $g$. It is clear that $ \norm{a}_{p,q} \leq 1$ for any $ (p,q,T) $-approximate atom. Note that from Remark \ref{rem:approx_atoms}\,(ii), each $(p,q,T)$-approximate atom is a multiple of $(p,q,T')$-approximate atom for any $T' \neq T$, where the constant does not depend on the atom. Moreover, from Remark \ref{rem:ell} we conclude that $h^{p,q}_\#(X)$ does not depend on the choice of $T$, and we will omit it from the notation.

We point out that $ \norm{\cdot}_{p,q} $ defines a $ p $-norm in $ h^{p,q}_\#(X) $. In effect, by \eqref{eq:401} we have
\begin{align}
	\norm{g}_{\ell^{\ast}_{1/p-1,T}}\leq 2 \norm{g}_{p,q}, \quad \forall \, g\in h^{p,q}_\#(X). \label{eq:402}
\end{align}
Moreover $ \norm{g}_{p,q}=0  \Leftrightarrow g=0$. Also, by definition it is not difficult to see that $ \norm{\lambda g}_{p,q}=\abs{\lambda} \norm{g}_{p,q} $, and $ \norm{ g+g' }_{p,q}^p\leq \norm{g}_{p,q}^p + \norm{g'}_{p,q}^p $, for any $ \lambda\in \C $, $ g,g^\prime\in  h^{p,q}_\#(X) $. As a consequence $ d_{p,q}(g,h):=\norm{g-h}_{p,q}^p $ for $g,h \in h^{p,q}_\#(X)$ defines a metric in $ h^{p,q}_\#(X) $.\\

Let $ h^{p,q}_{fin,\#}(X) $ be the subspace of $\ell^{\ast}_{1/p-1,T}(X) $ consisting of all finite linear combinations of $ (p,q,T) $-approximate atoms. The Proposition \ref{prop:414} also shows that the convergence in \eqref{eq:pqatom} is not just in distribution, but in $\ell^{\ast}_{1/p-1,T}$-norm too. Thus  $ h^{p,q}_{fin,\#}(X) $ is a dense subspace of $ (h^{p,q}_\#(X),\norm{\cdot}_{\ell^{\ast}_{1/p-1,T}})$.\\

Recall that if the functions $ a_j $ are $ (p,q) $-atoms, then the series \eqref{eq:pqatom} defines a continuous linear functional not just on $ \ell_{1/p-1,T}(X) $, but on $ \mathcal{L}_{1/p-1}(X) $. The elements in $ \mathcal{L}^{\ast}_{1/p-1}(X) $ having a decomposition in terms of such atoms define the atomic Hardy space $ H_{cw}^{p}(X)$ due to Coifman \& Weiss in \cite{CoWe77}.

\begin{remark}\label{rem:520}
	The space $ h^{p,q}_{fin,\#}(X) $ is also dense in $ (h^{p,q}_\#(X),d_{p,q}) $. In fact, let $ f $ be an element in $ h^{p,q}_\#(X) $, with decomposition $ f=\sum_j \lambda_j a_j $.	For an arbitrary $ \ve>0 $, there exists $ N(\ve)\in \N $ such that $ \sum_{j=m+1}^\infty \abs{\lambda_j}^p <\ve $, for any $ m\geq N(\ve) $. Since $\displaystyle{f-\sum_{j=1}^m \lambda_ja_j= \sum_{j=m+1}^\infty \lambda_ja_j}$ in distribution sense, and $\displaystyle{ \pare{\sum_{j=m+1}^{\infty} \abs{\lambda_j}^{p}}^{1/p} <\infty}$, we have that $\displaystyle{f-\sum_{j=1}^m \lambda_ja_j \in h^{p,q}_\#(X)}$ for all $ m\geq N $. From the definition of $ d_{p,q}(\cdot,\cdot) $ (and $ \norm{\cdot}_{p,q} $) we have
	\[
	d_{p,q}\big(f,\sum_{j=1}^m \lambda_j a_j\big)<\ve, \quad \text{ for any } m\geq N(\ve).
	\]
	This shows then the density of $ h^{p,q}_{fin,\#}(X) $ in $ (h^{p,q}_\#(X),d_{p,q})$.
\end{remark}

\begin{proposition}\label{prop:500}
	$ h^{p,q}_\#(X) $ equipped with the distance $ d_{p,q}(\cdot,\cdot) $ defines a complete metric space.
\end{proposition}

\begin{proof}
	Let $ \lla{f_n}_n $ to be a sequence in $ h^{p,q}_\#(X) $ such that $ \sum_{n=1}^{\infty} \norm{f_n}_{p,q}^p $ converges. Since $ \norm{\cdot}_{p,q} $ is a $ p $-norm, from \cite[Proposition A1]{Miroslav2013} it is sufficient to show that $ \sum_{n=1}^{\infty} f_n $ converges in $ (h^{p,q}_\#,d_{p,q}(\cdot,\cdot))$.
	
	By \eqref{eq:402} we have that $\displaystyle{  \sum_{n=1}^{\infty} \norm{f_n}_{\ell^{\ast}_{1/p-1,T}}^p }$ converges 
	and since $ p<1 $ we have $\displaystyle{ \sum_{n=1}^\infty \norm{f_n}_{\ell^{\ast}_{1/p-1,T}} }$ also converges.
	By completeness of $\ell^{\ast}_{1/p-1,T} $, follows $\displaystyle{ \sum_{n=1}^{\infty} f_n }$ converges to some $ f $ in $\ell^{\ast}_{1/p-1,T}$-norm, and so
	\begin{equation}\label{eq:520}
		\qquad \qquad \qquad \qquad \ip{f,\varphi}=\lim_{m\to \infty} \sum_{n=1}^{m}\ip{f_n,\varphi},\qquad \qquad \forall \varphi \in \ell_{1/p-1,T}(X).
	\end{equation}
	
	For each $ n\in \N $, let $\displaystyle{ f_n=\sum_{i=1}^{\infty} \lambda_i^n a_i^n }$ be a decomposition of $f_n$ in $ (p,q,T) $-approximate atoms such that
	\begin{align*}
		\sum_{i=1}^\infty \abs{\lambda_i^n}^p \leq \norm{f_n}_{p,q}^p + 2^{-n}.
	\end{align*}
	By Proposition \ref{prop:414}, the sum $\displaystyle \sum_{n=1}^\infty f_n = \sum_{n=1}^\infty\sum_{i=1}^\infty \lambda_i^n a_i^n $ converges in $\ell^{\ast}_{1/p-1,T}$-norm.
	From \eqref{eq:520} we have
	\begin{align*}
		\ip{f,\varphi}= \lim_{m,k\to\infty} \sum_{n=1}^m\sum_{i=1}^k \lambda_i^n \ip{ a_i^n,\varphi},\qquad \forall \,\, \varphi\in  \ell_{1/p-1,T}(X).
	\end{align*}
	So, this means that $f \in  h^{p,q}_\#(X) $ since it can be decomposed as $\displaystyle{ f= \sum_{n}\sum_{i} \lambda_i^n a_i^n }$. Following the argument in Remark \ref{rem:520} we have $\displaystyle{ \sum_{n}\sum_{i} \lambda_i^n a_i^n }$ converges to $ f $ in the metric $ d_{p,q} $, as desired.
\end{proof}

Note that any $ (p,\infty,T) $-approximate atom is in particular a $ (p,q,T) $-approximate atom, for every $ 1\leq q< \infty $. Moreover, we have the continuous embedding $h^{p,\infty}_\#(X) \subset h^{p,q}_\#(X)$, i.e.    
	\begin{equation*}
		\norm{f}_{p,q} \leq \norm{f}_{p,\infty}, \quad \quad \forall \, f\in h^{p,\infty}_\#(X). 
	\end{equation*}
In the next theorem, we prove the converse of this inclusion assuming the atomic decomposition theorem \cite[Theorem A]{CoWe77} for homogeneous Hardy spaces $H_{cw}^{p}(X)$  that is stated under assumption that $\mu$ is a Borel regular measure.

\begin{proposition} \label{prop:502}
	Let $ (X,d,\mu) $ be a space of homogeneous in which $ \mu $ is a Borel regular measure, and $ 0<p<1 \leq q<\infty$. Then
	$ h^{p,q}_\#(X)=h^{p,\infty}_\#(X) $ with comparable norms, i.e., there exists $ C>0 $, depending only on $ p $ and $ q $ such that 
	\begin{equation}\label{eq:413}
		\norm{\cdot}_{p,q}\leq \norm{\cdot}_{p,\infty} \leq C \norm{\cdot}_{p,q}.
	\end{equation}
\end{proposition}

\begin{proof}
	Let $ 1\leq q<\infty $. By the previous considerations, it remains to show that $ h^{p,q}_\# (X)\subset h^{p,\infty}_\#(X) $ with $\norm{\cdot}_{p,\infty} \leq C \norm{\cdot}_{p,q}$ .	\\
	
	We start by showing that any $ (p,q,T) $-approximate atom has a decomposition in $ (p,\infty,T) $-approximate atoms.
	Let $ a $ be a $ (p,q,T) $-approximate atom such that $\supp(a) \subset B:=B(x_B,r_B) $. Then we can write
	\begin{equation}\label{eq:411}
		\teal{a=a_B\1_B + 2 \frac{\1_B(a-a_B)}{2}}.
	\end{equation}
	It is straightforward to see that $ \frac{\1_B}{2}(a-a_B) $ is a $ (p,q) $-atom in $H_{cw}^{p}(X)$ and then from \cite[Theorem A pp. 592]{CoWe77} we have 
	\begin{equation}\label{eq:410}
		\frac{\1_B}{2}(a-a_B) = \sum_{j=1}^\infty\lambda_j a_{j}
	\end{equation}
	 in distribution $\mathcal{L}^{\ast}_{1/p-1}(X) $ (in particular in distribution $\ell^{\ast}_{1/p-1,T}(X)$), where each $ a_j $ is a $ (p,\infty) $-atom (in particular a $ (p,\infty,T) $-approximate atom), and
	\begin{equation}\label{eq:410s}
	\sum_{j=1}^\infty \abs{\lambda_j}^p\leq c	
	\end{equation}
	where $ c $ is depending on $ p $ and $ q $ but it is independent of $ a $. On the other hand, note that $ \supp(a_B\1_B)\subset B $ and $ \norm{a_B\1_B}_{L^\infty}\leq \abs{B}^{-1/p}$. Moreover, since $a$ is a $(p,q)$-function, if $ r_B<T $ we have
	\begin{align*}
		\abs{ \int a_B\1_B d\mu} = \abs{ \int_B ad\mu }\leq \abs{ B(x_B,T) }^{1-\frac1p}.
	\end{align*}
	This means $ a_B\1_B $ is a $ (p,\infty,T) $-approximate atom. Thus, from \eqref{eq:411}, \eqref{eq:410} and \eqref{eq:410s} we obtained a decomposition
	$
	a=\sum_{j} \beta_j b_j,
	$
	where each $b_j$ is a $ (p,\infty,T) $-approximate atom such that $ \pare{\sum_j\abs{\beta_j}^p}^{1/p} < c $, where $ c $ is a positive constant depending on $ p $ and $ q $ but independent of $ a $. \\
	
	Now, let $ f\in h^{p,q}_\#(X) $ and $ f=\sum_j \theta_j a_j $ any decomposition of $ f $ in $ (p,q,T) $-approximate atoms $ a_j $. From the previous construction, let $ \sum_{k} \beta_k^j b_k^j $ be the decomposition of each $a_j$ in $ (p,\infty,T) $-approximate atoms. Then
	\[
	{f=\sum_{j,k} (\theta_j \beta_{k}^j) b_{k}^j}, \qquad \text{in }\ell^{\ast}_{1/p-1,T}(X)
	\]
	is a decomposition in $ (p,\infty,T) $-approximate atoms since that
	\begin{equation}\label{eq:414}
		\pare{\sum_{j,k} \abs{\theta_j \beta_{k}^j}^p}^{1/p} \leq c\pare{\sum_{j} \abs{ \theta_j }^p}^{1/p}<\infty.
	\end{equation}
	So, $ f\in h^{p,\infty}_\#(X) $, and 
	by the arbitrariness of the decomposition $  f=\sum_j \theta_j a_j  $ we have
	\[
	\norm{f}_{p,\infty} \leq c\norm{f}_{p,q}.
	\]
\end{proof}

In view of the previous theorem, from now on we may denote the space $h^{p,q}_{\#}(X)$, for any $1\leq q \leq \infty$, simply by $h^{p}_{\#}(X)$, and its semi-norm by $\norm{\cdot}_{h^{p}_{\#}}:= \norm{\cdot}_{p,q}$.

In the same way, we denote by  $h^{p}_{cw}(X)$ the set of $f\in \ell^{\ast}_{1/p-1,T}(X)$ such that 
$\displaystyle{ f=\sum_{j=1}^\infty \lambda_j a_j}$  in   $\ell^{\ast}_{1/p-1,T}(X),$
for some  $\lla{\lambda_j}_{j} \in \ell^p(\C)$ and $\{a_j\}_{j}$  local  $(p,q)$-atoms, equipped with the norm 
$$\norm{ f }_{h^{p}_{cw}}:= \inf \pare{\sum_{j=1}^{\infty}  \abs{\lambda_j}^p}^{1/p},$$
where the infimum is taken over all such decompositions. Analogously as $h^{p}_{\#}(X)$, the space $h^{p}_{cw}(X)$  does not depend on $1\leq q \leq \infty$, 
assuming $ \mu $ is Borel regular. 
By the space $h^{p,q}_{fin}(X)$, we denote the set of $L^{q}(X)$ functions such that  $\displaystyle{f=\sum_{j=1}^n \lambda_j a_j}$ for some  $n\in \N $ (finite sum) and $\{a_j\}_{j}$ are local $(p,q)$-atoms. For this space, we consider the norm
$
\displaystyle{\norm{ f }_{h^{p,q}_{fin}}:= \inf \pare{\sum_{j=1}^{n}  \abs{\lambda_j}^p}^{1/p}},
$
where the infimum is taken over all finite decompositions of $f$. It is clear that $h^{p,q}_{fin}(X)$ is a dense subspace of $ h^{p}_{cw}(X)$ and
\begin{equation}
{\norm{ f }_{h^p_{cw}}} \leq \norm{ f }_{h^{p,q}_{fin}}, \quad \quad  \forall f \in h^{p,q}_{fin}(X).
\end{equation}
The converse in general is not true (see \cite{Bownik05} for the case $X=\R^n$ and $\mu=\mathcal{L}$). However, the norms $\norm{ \cdot }_{h^{p}_{cw}}$ and $\norm{ \cdot }_{h^{p,q}_{fin}}$ are equivalents on $ {h^{p,q}_{fin}(X)} $  for $ q\in (p,\infty)\cap [1,\infty) $ and for $q=\infty$ on $  {h^{p,\infty}_{fin}(X)}\cap UC(X) $, where $UC(X)$ denotes the space of absolutely continuous functions on $X$ (see \cite[Proposition 7.1]{HeYangWen21} and similar results for $ H^p(\R^n)$ see \cite[Theorem 3.1]{Meda08} and \cite[Theorem 5.6]{GraLiuYang08}).

\subsection{The local Campanato spaces as the dual of $ h^{p}_\#(X) $.}

Let $ \alpha \geq 0 $ and $ q\in [1,\infty] $. We call by \textit{Campanato space}, denoted by $\mathcal{C}_{\alpha,q}(X)$, the set of all $\mu$-measurable functions $f$ such that
$$
\| f \|_{\mathcal{C}_{\alpha,q}} := \sup_{B \subset X} \frac{1}{\abs{B}^{\alpha+\frac1q}} \norm{f-f_B}_{L^q(B)}.
$$ 
Clearly, if $\alpha=0$ and $q=1$, then $\mathcal{C}_{0,1}(X)=BMO(X)$. This space is also denoted in \cite{MaSe79} by $Lip(\alpha,q)$. \\

Here we are interested in the non-homogeneous version of such spaces. Given $T>0$ fixed, for any ball $ B=B(x_B,r_B) $ and $f\in L^{q}_{loc}(X) $, we define the functional
\begin{equation}\label{eq:7070}
	m_{B,T}(f):=\begin{cases}
		f_B, &r_B<T\\
		0,& r_B\geq T,\\
	\end{cases}
\end{equation}
and
\begin{equation}\label{eq:7071}
	\mathfrak{M}_{\alpha,q,T}^B(f):=\frac{1}{\abs{B}^{\alpha+\frac1q}} \norm{f-m_{B,T}(f)}_{L^q(B)}.
\end{equation}
We define the \emph{local Campanato space} as
\[
c_{\alpha,q, T}(X):= \lla{ f\in L^q_{{loc}}(X) : \norm{f}_{c_{\alpha,q,T}}:= \sup_{B\subset X} \mathfrak{M}_{\alpha,q,T}^B(f)<\infty }.
\]
The functional $\norm{\cdot}_{c_{\alpha,q,T}}$ defines a norm in $c_{\alpha,q, T}(X)$. It is not difficult to see that $ c_{\alpha,q,T}(X) \subset \mathcal{C}_{\alpha,q}(X) $ continuously  with $ \norm{f}_{\mathcal{C}_{\alpha,q}} \leq 2 \norm{f}_{c_{\alpha,q,T}} $.

\begin{remark}\label{eq:comparison-lips-camp} 
	The spaces $ \ell_{\alpha,T}(X)$ and $c_{\alpha,q,T}(X)$ can be identified and have comparable norms. Clearly $ \ell_{\alpha,T}(X) \subset c_{\alpha,q,T}(X) $ continuously. Conversely, let $ q\in [1,\infty] $ and $f \in L^{q}_{loc}(X)$ belonging to $c_{\alpha,q,T}(X)$. We claim that there exists a function $\tilde{f} \in L^{\infty}_{loc}(X)$ such that $\tilde{f}=f$ almost everywhere and $\tilde{f}\in \ell_{\alpha,T}(X) $ with  $ \norm{\tilde{f}}_{\ell_{\alpha,T}}\lesssim \norm{f}_{c_{\alpha,q,T}}$. Indeed, since $c_{\alpha,q,T}(X) $  is continuously embedded in $\mathcal{C}_{\alpha,q}(X) $, it follows by \cite[Theorem 4]{MaSe79} that there exists $ \tilde{f} $ equal to $ f $ a.e. and $C=C(\alpha,q)>0$ such that for any ball $ B \subset X $,
\begin{equation}\nonumber 
	\abs{\tilde{f}(x)-\tilde{f}(y)}\leq C\norm{f}_{\mathcal{C}_{\alpha,q}}\abs{B}^\alpha,\qquad \qquad \forall\, x,y\in B.
\end{equation}
Now if $ B $ is a ball with 	$ r_B\geq T $ then for each $ x\in B $ we have 
\begin{align*}
	\abs{\tilde{f}(x)}\leq \abs{\tilde{f}(x)-\tilde{f}_B}+\abs{\tilde{f}_B}&\leq\frac1{\abs{B}}\int_{B} \abs{ \tilde{f}(x)-\tilde{f}(y) }d\mu(y) + \frac{1}{\abs{B}}\int_B \abs{f(y)}d\mu(y) \nn\\
	&\leq C\norm{f}_{\mathcal{C}_{\alpha,q}}\abs{B}^\alpha+\frac{1}{\abs{B}^\frac1q}\norm{f}_{L^q(B)} \nn\\
	&\leq (2C+1)\norm{f}_{c_{\alpha,q,T}}\abs{B}^\alpha.\label{eq:comp2}
\end{align*}
From the previous estimates we obtain $ \tilde{f}\in \ell_{\alpha,T}(X) $
with $ \norm{\tilde{f}}_{\ell_{\alpha,T}}\lesssim \norm{f}_{c_{\alpha,q,T}}$.\\
\end{remark}

When $\alpha=0$ and $1\leq q <\infty$ we have $c_{0,q, T}(X)=bmo(X)$, where $bmo(X)$ denotes the local $BMO$ space over $X$ (see \cite[Corollary 3.3]{DafniMomYue16}). From Remarks \ref{rem:ell} and \ref{eq:comparison-lips-camp} it follows that $ c_{\alpha,q, T}(X)=c_{\alpha,q, T'}(X)$ with equivalent norms for $ T \neq T' $. \\

In what follows we present an alternative characterization of $c_{\alpha,q,T}(X)$, inspired by the analogous result for $bmo(X)$, proved in \cite[Lemma 6.1]{DafniMomYue16}. This result will be useful in Proposition \ref{prop:duality} to show a duality relation between local Campanato and Hardy spaces.

\begin{proposition} \label{prop:406}
	Let $ \beta>0 $, $ 1\leq q\leq \infty $, and $ f\in L^{q}_{loc}(X) $. Then $ f\in c_{\beta,q,T}(X) $ if and only if for every ball $ B=B(x_B,r_B) $ in $ X $ there exists a constant $ C_B $ such that
	\begin{enumerate}[(i)]
		\item $ \displaystyle M_1 := 
		\sup_B \frac{1}{|B|^{\beta+\frac1q}}  \norm{ f-C_B }_{L^q(B)}  < \infty$;
		\item $\displaystyle M_2 := \sup_{\substack{B\subset X}} \frac{|C_B|}{|B(x_B,T)|^{\beta} 
		} <\infty$;
	\end{enumerate}
	and
	\[
	\norm{f}_{c_{\beta,q,T}}\approx \inf \max\lla{M_1,M_2	}
	\]
	where the infimum is taken over all choices of the $ \lla{C_B} $ such that $ (i) $ and $ (ii) $ hold.
\end{proposition}

\begin{proof}
	For each $ f\in c_{\beta,q,T}(X) $ let  $ C_B=m_{B,T}(f) $. Clearly, $M_{1} = \|f\|_{c_{\beta,q,T}}$ and it will be sufficient to show (ii) for $ B=B(x_B,r_B) $ with $ r_B<T $. Suppose first that $ \abs{B}=\abs{B(x_B,T)} $. Then
	\begin{align*}
		\abs{C_B}&=\frac{1}{\abs{B}}\abs{\int_B f(x)d\mu(x)} \leq \frac{1}{\abs{B(x_B,T)}} \int_{B(x_B,T)} \abs{f(x)}d\mu(x) \\
		& \leq \abs{B(x_B,T)}^{-\frac1q} \norm{f}_{L^q(B(x_B,T))} \\
		& \leq \abs{B(x_B,T)}^\beta \norm{f}_{c_{\beta,q,T}}.
	\end{align*}
	
	Suppose now that $ \abs{B}< \abs{B(x_B,T)} $. Following the same ideas as \cite[Lemma 3]{MaSe79}, let $ m $ be a non-negative integer such that
	\begin{equation}\label{eq:0032}
		(A^\prime)^m\abs{B}< \abs{B(x_B,T)}\leq (A^\prime)^{m+1}\abs{B}.
	\end{equation}
	We claim that there exist positive constants $ r_0:=r_{B}<r_1<r_2<\cdots < r_m<{r_{m+1}:=T} $ such that
	\begin{equation}\label{eq:0033}
		(A^\prime)^{k-1}\abs{B}< \abs{B(x_B,r_k)}\leq (A^\prime)^k\abs{B}.
	\end{equation}
	In fact, note first that \eqref{eq:0033} holds for $ k=0 $ and $ k=m+1 $ and for $ k\in \lla{1,2,,\cdots,m} $ we define
	\[
	r_k:=\max \lla{ s:\abs{B(x_B,s)}\leq (A^\prime)^k\abs{B} }.
	\]
	The existence of maximum $ r_k $ is given by the continuity from the left of the function $ s\mapsto \abs{B(x_B,s)} $. Then
	\begin{equation}\label{eq:000321}
		\abs{B(x_B,r_k)}\leq (A^\prime)^k\abs{B}, \qquad \qquad \forall\, 0\leq k \leq m+1 .
	\end{equation}
	On the other hand, from the doubling condition we obtain
	\[
	\abs{B(x_B,2r_k)}\leq 
	(A^{\prime})^{k+1} \abs{B}, \qquad \qquad \forall\, 0\leq k\leq m-1,
	\]
	Then, $ r_k<2r_k\leq r_{k+1} $ for $ 0\leq k\leq m-1 $ and from the left inequality in \eqref{eq:0032} we obtain $ r_m<T $. This, along with definition of $ r_k $'s, implies
	\begin{equation}\label{eq:000322}
		(A^\prime)^{k-1}\abs{B}< \abs{B(x_B,r_k)}, \qquad\qquad \forall\, 0\leq k\leq m+1.
	\end{equation} 
	Therefore, from \eqref{eq:000321} and \eqref{eq:000322} we get \eqref{eq:0033}.\\
	
	Coming back to the proof of (ii), denote by $ B_k:=B(x_B,r_k) $. From \eqref{eq:0033} we have
	\begin{align*}
		\sum_{k=1}^{m+1} \abs{ B_k }^\beta &\leq (A')^{\beta} \abs{B}^\beta \sum_{k=0}^{m} (A^\prime)^{\beta k} = (A')^{\beta} \abs{B}^\beta \frac{(A^{\prime})^{\beta(m+1)}-1}{(A^{\prime})^\beta-1} \leq \frac{(A')^{2\beta}}{(A^{\prime})^\beta-1} [(A')^m\abs{B}]^\beta.
	\end{align*}
	From the definition of $c_{\beta,q,T}(X)$, the previous inequality and \eqref{eq:0033} we get
		\begin{align*}
		\abs{C_B}&\leq \sum_{k=1}^m \abs{ m_{B_{k-1},T}(f) - m_{B_{k},T}(f)}+\abs{m_{B_m,T}(f)} \\
		& \leq \sum_{k=1}^m \frac{1}{\abs{B_{k-1}}}\int_{B_{k-1}} \abs{ f(x) - m_{B_k,T}(f)}d\mu(x)
		+ \frac{1}{\abs{B_m}} \int_{B(x_B,T)} \abs{f(x)}d\mu(x)\\
		& \leq \sum_{k=1}^m \frac{\abs{B_k}}{\abs{B_{k-1}}}\frac{1}{\abs{B_k}}\int_{B_{k}} \abs{ f(x) - m_{B_k,T}(f)}d\mu(x)
		+\frac{\abs{B_{m+1}}}{\abs{B_m}} \abs{B_{m+1}}^\beta\norm{f}_{c_{\beta,q,T}}\\
		& \leq (A^\prime)^2\sum_{k=1}^m \abs{B_k}^{-\frac1q}\norm{ f - m_{B_k,T}(f)}_{L^q(B_k)}+ (A^\prime)^2 \abs{B_{m+1}}^\beta \norm{f}_{c_{\beta,q,T}}\\
		&\leq  (A^\prime)^2 \norm{f}_{c_{\beta,q,T}}\sum_{k=1}^{m+1} \abs{ B_k }^\beta \\
		& \leq \frac{(A^\prime)^{2(\beta+1)}}{(A^{\prime})^\beta-1} \norm{f}_{c_{\beta,q,T}} [(A^\prime)^m \abs{B}]^\beta \\
		&\leq \frac{(A^\prime)^{2(\beta+1)}}{(A^{\prime})^\beta-1} \norm{f}_{c_{\beta,q,T}} \abs{B(x_B,T)}^\beta,
	\end{align*}
	which concludes the proof of (ii).

Conversely, suppose that for each ball $B$ there exists a constant $C_{B}$ such that (i) and (i) hold. Assume first that $ B $ is such that $ r_B<T $. Then, 
	\begin{align*}
		\frac{1}{|B|^{\beta+\frac1q}} \norm{ f-m_{B,T}(f)}_{L^q(B)} &\leq \frac{1}{|B|^{\beta+\frac1q}} \norm{ f-C_B}_{L^q(B)}+ \frac{1}{|B|^{\beta+\frac1q}} \norm{ m_{B,T}(f)-C_B}_{L^q(B)}\\
		& \leq M_1 + \frac{1}{|B|^{\beta+\frac1q}} \norm{ f-C_B}_{L^q(B)}\leq 2M_1. 
	\end{align*}
	Now, suppose $ r_B\geq T $. In this case $ m_{B,T}(f)=0 $ and
	\begin{align*}
		\frac{1}{|B|^{\beta+\frac1q}} \norm{ f-m_{B,T}(f)}_{L^q(B)} &= \frac{1}{|B|^{\beta+\frac1q}} \norm{ f}_{L^q(B)}  \\
		&\leq  \frac{1}{|B|^{\beta+\frac1q}} \norm{ f-C_B}_{L^q(B)}  + \frac{1}{|B|^{\beta+\frac1q}} \norm{C_B}_{L^q(B)}  \\
		& \leq M_1+\frac{\abs{C_B}}{ \abs{B}^{\beta} }  \leq M_1 + \frac{\abs{C_B}}{ \abs{B(x_B,T)}^{\beta} }  \\
		&\leq M_1+M_2.
	\end{align*}
	Then
	\[
	\norm{f}_{c_{\beta,q,T}}\leq 2\max\lla{M_1,M_2}
	\]
	and by the arbitrariness of the family $ \lla{C_B}_B $ we have
	\[
	\norm{f}_{c_{\beta,q,T}}\leq 2\inf\max\lla{M_1,M_2},
	\]
	where the infimum is taken over all choices of $ \lla{C_B}_{B} $.
\end{proof}

\begin{remark}\label{rem:408}
	Note that elements in $ c_{\beta,q,T}(X) $ define naturally bounded linear operators on finite linear combinations of $ (p,q,T) $-approximate atoms. In fact, let $ 0<p<1\leq q\leq \infty $ and $ a $ be a $ (p,q,T) $-approximate atom supported in $ B=B(x_B,r_B) $. For any $ f\in c_{1/p-1,q',T}(X) $, from condition (ii) in Proposition \ref{prop:406} we may control
	\begin{align}
		\abs{\int afd\mu}&\leq \int_{B} \abs{a(x)}\abs{ f(x)-m_{B,T}(f) }d\mu(x)+ \abs{m_{B,T}(f)} \abs{ \int_B a(x)d\mu(x) }\nn \\
		& \leq  \norm{a}_{L^q} \norm{f-m_{B,T}(f)}_{L^{q'}(B)} + \abs{m_{B,T}(f)} \abs{ \int_{B} a(x)d\mu(x) }\nn \\
		& \leq |B|^{1-\frac{1}{q'}-\frac1p} \norm{ f -m_{B,T}(f) }_{L^{q'}(B)} + \frac{\abs{m_{B,T}(f)}}{\abs{B(x_B,T)}^{1/p-1} }\nn \\
		& \leq \pare{1+\frac{(A^\prime)^{\frac2p}}{(A^{\prime})^{\frac1p-1}-1} }\norm{f}_{c_{1/p-1,q',T}} . \label{eq:523}
	\end{align}
	If $ r_B\geq T $, then $ m_{B,T}(f)=0 $, and so
	\begin{align}\label{eq:524}
		\abs{\int af d\mu }\leq \norm{a}_{L^q} \norm{f-m_{B,T}(f)}_{L^{q'}(B)}\leq \norm{f}_{c_{1/p-1,q',T}}.
	\end{align}
\end{remark}

	\begin{proposition}\label{prop:duality}
		Let $ (X,d,\mu) $ be a space of homogeneous type and $0<p<1$. Then: 
		\begin{itemize}
			\item[(i)] $(h^{p,q}_\#)^* = c_{1/p-1,q',T}(X) $ with equivalent norms for $ 1\leq q<\infty $ and $ c_{1/p-1,1,T}(X) \subset (h^{p,\infty}_\#)^*  $ continuously.
			\item[(ii)] If in addition $ \mu $ is a Borel regular measure, $ (h^{p,\infty}_\#)^* \subset c_{1/p-1,1,T}(X) $ continuously.
		\end{itemize}
	\end{proposition}

\begin{proof}
	Let $ f\in c_{1/p-1,q',T}(X) $. We start defining an operator $ \Lambda_f $ on $ h^{p,q}_{fin}(X) $ by
	\begin{equation}\label{eq:8000*}
		\Lambda_f(g) :=  \int fg\, d\mu =\sum_{j=1}^{n} \lambda_j \int a_jf\ d\mu, 
	\end{equation}
	for $ g \in h^{p,q}_{fin}(X) $ given by $ g=\sum_{j=1}^{n} \lambda_j a_j $. From Remark \ref{rem:408},  we obtain
	\begin{equation}\label{eq:7040*}
		\qquad \qquad \qquad \quad \qquad \abs{\Lambda_f(g)}\lesssim \norm{f}_{c_{1/p-1,q',T}} \pare{\sum_{j=1}^{n} \abs{\lambda_j}^p}^{1/p}, \qquad \qquad  \forall \ g\in h^{p,q}_{fin}(X).
	\end{equation}
	Now we extend this functional to infinity sums. Consider $ G \in h^{p,q}_\#(X) $ having a decomposition $ G=\sum_{j=1}^{\infty} \lambda_j a_j $ in $ \ell^{\ast}_{1/p-1,T}(X) $. Then from \eqref{eq:7040*} we obtain that $ \lla{\Lambda_f \pare{\sum_{j=1}^{m} \lambda_j a_j} }_m $ is a Cauchy sequence in $\C$. By completeness, we may define
	\begin{equation}\label{eq:8010*}
		\tilde{\Lambda}_f ( G ):=\lim_{m\to \infty} \Lambda_f\pare{\sum_{j=1}^{m} \lambda_j a_j}= \lim_{m\to \infty} \sum_{j=1}^{m} \lambda_j \int a_j f \ d\mu.
	\end{equation}
	From Remark \ref{eq:comparison-lips-camp} there exists $ \tilde{f}\in \ell_{1/p-1,T}(X) $ with $ f(x)=\tilde{f}(x) $ a.e. $ x\in X $. Then if $ \sum_{j=1}^m \beta_j b_j $ is another decomposition for $ G $ we obtain
	\begin{align*}
		\lim_{m\to \infty} \sum_{j=1}^{m} \lambda_j \int a_j f \ d\mu= \lim_{m\to \infty} \sum_{j=1}^{m} \lambda_j \int a_j \tilde{f} \ d\mu = \ip{G,\tilde{f}}&= \lim_{m\to \infty} \sum_{j=1}^{m} \beta_j \int b_j \tilde{f} \ d\mu\\
		&= \lim_{m\to \infty} \sum_{j=1}^{m} \beta_j \int b_j f \ d\mu.
	\end{align*}
	This is, the definition of $ \tilde{\Lambda}_f(G)$ is independent of the choice of the decomposition of $ G $.	So, $ \tilde{\Lambda}_f: h^{p,q}_\#(X) \to \C$ is a well defined operator and $ \tilde{\Lambda}_f|_{h^{p,q}_{fin}(X)} = \Lambda_f $.
	Moreover, for $ G\in h^{p,q}_\#(X) $, taking a decomposition $\displaystyle{G=\sum_{j=1}^{\infty} \lambda_j a_j }$ such that $ \pare{\sum_{j=1}^\infty \abs{\lambda_j}^p}^\frac1p \leq 2\norm{G}_{h^{p,q}_\#(X)} $, by \eqref{eq:7040*} we obtain
	\begin{align}
		\abs{\tilde{\Lambda}_f(G)}
		&\lesssim \bigg| \tilde{\Lambda}_f(G)-\Lambda_f\big(\sum_{j=1}^m \lambda_j a_j\big) \bigg| + 2 \norm{ f }_{c_{1/p-1,q',T}}\norm{ G }_{h^{p,q}_\#(X)}, \label{eq:70700*}
	\end{align}
	for any $ m\in \N $. This means that
	\begin{equation*}
		\abs{\tilde{\Lambda}_f(G)} \lesssim 2 \norm{ f }_{c_{1/p-1,q',T}}\norm{ G }_{h^{p,q}_\#(X)}
	\end{equation*}
	and therefore shows that $ c_{1/p-1,q^\prime,T}(X) \subset (h^{p,q}_\#(X))^* $ for $ q\in [1,\infty] $.\\ \\
	
	Conversely, let $ 1\leq q<\infty $ and $ \Lambda \in (h^{p,q}_\#)^{\ast} $. 
	For a ball $ B \subset X $ with radius $ r_B\geq T $ and a function $ g\in L^{q}(B) $ with $ \norm{g}_{L^{q}(B)}>0 $, the function $ \tilde{g}:=(\abs{B}^{\frac1q-\frac1p}\norm{g}_{L^q(B)}^{-1})\, g \1_B  $ is a $ (p,q,T) $-approximate atom (in particular, a  local $ (p,q) $-atom), and so
	\begin{align*}
		\abs{ \ip{\Lambda,\1_B g} } 
		& \leq \abs{B}^{\frac1p-\frac1q} \norm{\Lambda}_{(h^{p,q}_\#(X))^*} \norm{ g }_{L^q(B)}.
	\end{align*}
	This means that $ \ip{\Lambda,\1_B (\cdot)} $ defines a bounded linear operator on $ L^q(B) $ and hence from Riesz representation theorem there exists a unique $ f^{(B)} \in L^{q^\prime}(B)$ such that
	\begin{align}
		\qquad \qquad \qquad \qquad \qquad	\ip{ \Lambda,\1_B g }= \int_B f^{(B)} g \, d\mu , \qquad \qquad \qquad \quad  \forall \, g\in L^q(B),	\label{eq:10100*}
	\end{align}
	and 
	\begin{align}
		\norm{f^{(B)}}_{L^{q^\prime}(B)} \leq \norm{\Lambda}_{(h^{p,q}_\#(X))^* } \abs{B}^{\frac1p-\frac1q}. \label{eq:10200*} 
	\end{align}
	Moreover,  if $ B_1\subset B_2 $ with $ r_{B_1} \geq T $ and $ g\in L^q(B_1) $, then 
	$ \1_{B_1} g \in L^q(B_2)$	and from \eqref{eq:10100*}
	\begin{align*}
		\int_{B_1} f^{(B_2)} g d\mu= \int_{B_2} f^{(B_2)} \1_{B_1}g d\mu = \ip{ \Lambda,\1_{B_2} \, \1_{B_1} g } = \ip{ \Lambda, \1_{B_1} g } =\int_{B_1} f^{(B_1)} g d\mu .
	\end{align*}
	By the uniqueness of $ f^{(B_1)} $ we have $ f^{(B_1)}=f^{(B_2)}\chi_{B_1} $. Consider $L^{q}_{c}(X)$ the set of $ g\in L^q_{loc}(X) $ such that $ g\in L^q(X) $ with bounded support. Thus from \eqref{eq:10100*} we have
	\begin{equation}\label{eq:10110*}
		\int fg d\mu = \ip{\Lambda, g}, \quad \forall\, g \in L^{q}_{c}(X).
	\end{equation} 
	Also, from \eqref{eq:10200*} we have
	\begin{equation}
		\norm{f}_{L^{q^\prime}(B)} \leq \norm{\Lambda}_{(h^{p,q}_\#(X))^* } \abs{B}^{\frac1p-1+\frac1{q^\prime}}, \label{eq:10210*}
	\end{equation}
	for any ball $ B $ with $ r_B\geq T $. Note that, in
	particular 
	\eqref{eq:10110*} it is true for 
	$ g \in  h^{p,q}_{fin}(X) $.
	
	Now, suppose that $ B $ is a ball with $ r_B<T $. Let $ \vp \in L^q(B) $ such that $ \norm{\vp}_{L^q(B)} =1$. Then, the function
	$
	\displaystyle{\tilde{\varphi} = \frac{[\vp-m_{B,T}(\vp)]\1_B }{ 2\abs{B}^{\frac1p-\frac1q}}}
	$ 
	is a $ (p,q) $-atom (in particular a $ (p,q,T) $-approximate atom), so $ \norm{\tilde{\vp}}_{p,q}\leq 1 $. From \eqref{eq:10110*}
	\begin{align*}
		\abs{ \int_B (f-m_{B,T}(f))\vp d\mu } =
		\abs{ \int_B f(\vp-m_{B,T}(\vp))d\mu } 
		&=2\abs{B}^{\frac1p-\frac1q} \abs{\ip{ \Lambda, \frac{ (\vp-m_{B,T}(\vp))\1_B }{2\abs{B}^{\frac1p-\frac1q}} }} \\ 
		& \leq 2\abs{B}^{\frac1p-\frac1q} \norm{ \Lambda }_{(h^{p,q}_\#(X))^*} .
	\end{align*}
	Then $ f-m_{B,T}(f) \in L^{q^\prime}(B)$ and 
	\begin{equation}\label{eq:5100*}
		\frac{1}{\abs{B}^{\frac1p-1+\frac1{q^\prime}}} \norm{f-m_{B,T}(f)}_{L^{q^\prime}(B)} \leq 2 \norm{\Lambda}_{(h^{p,q}_\#(X))^*} 
	\end{equation}	
	Summarizing, from \eqref{eq:10210*} and \eqref{eq:5100*} 
	we have $ f\in c_{1/p-1,q^\prime,T}(X) $ 
	and 
	\begin{equation}\label{eq:dual001*}
		\norm{f}_{c_{1/p-1,q^\prime,T}} \lesssim \norm{\Lambda}_{(h^{p,q}_\#(X))^*}. 
	\end{equation}
	This shows $ (h^{p,q}_\#(X))^*\subset c_{1/p-1,q^\prime,T}(X) $ for $ q\in [1,\infty) $.\\
	
	Now we move on to the case $q=\infty$. Let  $ \Lambda \in (h^{p,\infty}_{\#})^{\ast} $. If $ B \subset X $ is a ball with radius $ r_B\geq T $, for any $ g\in L^2(B) $ with $ \norm{g}_{L^2(B)}>0 $, we have that
	\[
	\tilde{g}:=\frac{  g \1_B }{\abs{B}^{\frac1p-\frac12}\norm{g}_{L^2(B)}}
	\]
	is a local $ (p,2) $-atom (in particular a $ (p,2) $-approximate atom). By \eqref{eq:413} in Proposition \ref{prop:502}, we have
	$ 	\norm{\tilde{g}}_{p,2}\leq 1$ and $ \norm{\tilde{g}}_{p,\infty}\leq C_{p,2} $.
	Then
	\begin{align*}
		\abs{ \ip{\Lambda,\1_B g} } 
		& \leq C_{p,2}\abs{B}^{\frac1p-\frac12} \norm{\Lambda}_{(h^{p,\infty}_{\#}(X))^*} \norm{ g }_{L^2(B)}.
	\end{align*}
	It means that $ \ip{\Lambda,\1_B (\cdot)} $ defines a bounded linear operator on $ L^2(B) $ and hence from Riesz representation theorem there exists unique $ f^{(B)} \in L^2(B)$ such that
	\begin{align*}
		\ip{ \Lambda,\1_B g }= \int_B f^{(B)} g \, d\mu , \qquad \text{for all } g\in L^2(B),	
	\end{align*}
	and
	\begin{align*}
		\norm{f^{(B)}}_{L^2(B)} \leq {C_{p,2}}\norm{\Lambda}_{(h^{p,\infty}_{\#}(X))^* } \abs{B}^{\frac1p-\frac12}. 
	\end{align*}
	As before, it allows us to define $ f\in L^2_{loc}(X) $ such that
	\begin{equation}\label{eq:1011*}
		\int fg d\mu =	 \ip{\Lambda, g},
	\end{equation} 
	for any $ g\in L^2(X) $ with bounded support and 
	\begin{equation}
		\norm{f}_{L^2(B)} 	\leq {C_{p,2}}\norm{\Lambda}_{(h^{p,\infty}_{\#}(X))^* } \abs{B}^{\frac1p-\frac12}, \label{eq:1021*}
	\end{equation}
	for any ball $ B $ with $ r_B\geq T $. In particular we have \eqref{eq:1011*} for elements in  $ h^{p,\infty}_{fin}(X) $. On the other hand, if $ B $ is a ball such that $ r_B\geq T $, by \eqref{eq:1021*} we have 
	\begin{align}
		\frac{1}{\abs{B}^{\frac1p}}\int \abs{f-m_{B,T}(f)}d\mu &
		\leq 	\frac{1}{\abs{B}^{\frac1p-\frac12}} \norm{f}_{L^2(B)}
		\leq {C_{p,2}} \norm{\Lambda}_{(h^{p,\infty}_{\#}(X))^*}. \label{eq:511*}
	\end{align} 
	If $ B $ is a ball with $ r_B<T $ and $ \vp \in L^2(B) $ such that $ \norm{\vp}_{L^2(B)} =1$,  then
	$
	\displaystyle{\tilde{\varphi} = \frac{[\vp-m_{B,T}(\vp)]\1_B }{ 2\abs{B}^{\frac1p-\frac12} }}
	$ 
	is a local $ (p,2) $-atom. From \eqref{eq:1011*}, we obtain
	\begin{align*}
		\abs{ \int_B (f-m_{B,T}(f))\vp d\mu } &	\leq \abs{ \int_B f(\vp-m_{B,T}(\vp))d\mu } 
		=2\abs{B}^{1/p-1/2} \abs{ \int \frac{ f(\vp-m_{B,T}(\vp))\1_B }{2\abs{B}^{1/p-1/2}} d\mu } \\
		& \leq 2{C_{p,2}}\abs{B}^{1/p-1/2} \norm{ \Lambda }_{(h^{p,\infty}_{\#}(X))^*} .
	\end{align*}
	So, 
	\begin{equation*}
		\norm{ f-m_{B,T}(f) }_{L^2(B)}\leq 2{C_{p,2}}\abs{B}^{1/p-1/2}\norm{ \Lambda }_{(h^{p,\infty}_{\#}(X))^*}.
	\end{equation*}
	and
	\begin{align}
		\frac{1}{\abs{B}^{1/p}}\int_B \abs{f-m_{B,T}(f)}& \leq \frac{1}{\abs{B}^{1/p-1/2}} \norm{f-m_{B,T}(f)}_{L^2(B)} 
		\leq 2{C_{p,2}}  \norm{\Lambda}_{(h^{p,\infty}_{\#}(X))^*} \label{eq:510*}
	\end{align}
	From \eqref{eq:511*} and \eqref{eq:510*}, we have $ f\in c_{1/p-1,1,T}(X) $ and 
	$$\norm{f}_{{c_{1/p-1,1,T}(X)}}\leq 2C_{p,\infty}\norm{\Lambda}_{(h^{p,\infty}_{\#}(X))^*}. $$
	This	shows $ (h^{p,\infty}_\#(X))^*\subset {c_{1/p-1,1,T}(X)} $.
\end{proof}

Straightforward from Propositions \ref{prop:502} and \ref{prop:duality} we have:
\begin{corollary}
		If $ \mu $ is a Borel regular measure, then $ c_{1/p-1,q,T}(X)= c_{1/p-1,1,T}(X)$ for all $ q\in(1,\infty]$, with equivalent norms.
\end{corollary}

\subsection{Molecular decomposition}

In what follows, we define molecules with approximate moment conditions in $h^{p}_{\#}(X)$. Such theory was previously established for $H^p(X)$ in \cite[Section 3.1]{LiuChangFuYang2018} to describe the boundedness of Calder\'on-Zygmund operators. We will follow the same notation of \cite[Definition 3.2]{LiuChangFuYang2018}.

\begin{definition}\label{def:appmol}
	Let $0<p<1$ and $1\leq q \leq \infty$ with $p<q$, and $\lambda := \{\lambda_k\}_{k\in \N} \subset [0,\infty)$ satisfying
	\begin{equation} \label{lambda-type-condition}
		\left\| \lambda \right\|_{p}:= \sum_{k=1}^{\infty} k(\lambda_k)^p < \infty.
	\end{equation}
	A measurable function $ M $ in $ X $	 
	is called a $(p,q,T,\lambda)$-approximate molecule if there exists a ball $B=B(x_B,r_{B})\subset X$ such that
	\begin{enumerate}[i)]
		\item $\displaystyle \| M \, \1_{B} \|_{L^q} \leq |B|^{\frac1q-\frac1p}$; \\
		\item For any $k\in \N$, $\| M \, \1_{A_k}  \|_{L^q} \leq \lambda_k \, |2^kB|^{\frac1q-\frac1p}$, where $A_k = 2^kB \setminus 2^{k-1}B$; \\
		\item $\displaystyle 	\abs{ \int M \, d\mu }\leq |B(x_B,T)|^{1-\frac1p}.$
	\end{enumerate}
\end{definition}

As in Remark \ref{rem:approx_atoms} (i), the approximate moment condition (iii) in the previous definition is local, i.e., if $r_B\geq T$, from the size conditions (i) and (ii) of the molecule, we have
\begin{align*}
	\abs{ \int M \, d\mu } &\leq |B|^{1-\frac1q} \| M \, \1_{B} \|_{L^q} + \sum_{k=1}^{\infty} \lambda_k \, |A_k|^{1-\frac1q} \| M \, \1_{A_k}  \|_{L^q} 
	\leq |B|^{1-\frac1p} + \sum_{k=1}^{\infty} \lambda_k \, |2^kB|^{1-\frac1p} \\
	&\leq |B|^{1-\frac1p} \big(1+ \sum_{k=1}^{\infty} (\lambda_k)^p  \big) \\
	&\lesssim B(x_B,T)^{1-\frac1p}.
\end{align*}

It is clear that any $ (p,q,T)$-approximate atom is a $ (p,q,T,\lambda) $-approximate molecule for any sequence $ \lambda=\lla{\lambda_k}_{k\in \N} $ satisfying \eqref{lambda-type-condition}. On the other hand, a  $ (p,q,T,\lambda)$-approximate molecule associated to the null sequence $\lambda=\lla{0}_{k\in \N}$ is a $ (p,q,T) $-approximate atom. We point out the assumption  \eqref{lambda-type-condition} implies the sequence $ \lla{\lambda_k}_{k\in \N} \in \left(\ell^1\cap \ell^p\right)(\R)$. Moreover,  any $ (p,q,T,\lambda) $-approximate molecule $M$ centered in $B$ defines a distribution on $ \ell_{\frac1p-1,T}(X) $. In effect, let $ A_0=B $, $ A_j=2^jB\setminus 2^{j-1}B $ and $ \vp \in \ell_{1/p-1,T}$. Consider the indexes $j \in \N \cup \{0\}$ such that $ 2^jr_B\geq T $, thus using only the properties (i) and (ii) above we have
	\begin{align}
		\abs{\int_{2^jB} M\vp\, d\mu }&\leq \abs{\int_B M\vp\, d\mu} + \sum_{k=1}^j \abs{ \int_{A_k} M\vp \, d\mu } \\
		&\leq \norm{\vp}_{L^\infty(B)}\int_B\abs{ M}  d\mu + \sum_{k=1}^j \norm{\vp}_{L^\infty(A_j)}\int_{A_k}\abs{ M }d\mu \nn \\
		&\leq \norm{\vp}_{L^\infty(2^jB)} \pare{ \abs{B}^{1-\frac1p} + \sum_{k=1}^j\lambda_k  \abs{A_k}^{1-\frac1q} \abs{2^kB}^{\frac1q-\frac1p}} \nn\\
		&\leq \norm{\vp}_{L^\infty(2^jB)} \pare{ \abs{B}^{1-\frac1p} + \sum_{k=1}^j\lambda_k   \abs{2^kB}^{1-\frac1p}} \nn \\ 
		&\leq \norm{\vp}_{L^\infty(2^jB)} \pare{ \abs{B}^{1-\frac1p} + \abs{2^jB}^{1-\frac1p}\sum_{k=1}^j\lambda_k   (A')^{(j-k)(\frac1p-1)}} \nn \\ 
		& \leq \norm{\vp}_{L^\infty(2^jB)} \pare{ (A')^{j(\frac1p-1)}\abs{2^jB}^{1-\frac1p} + \abs{2^jB}^{1-\frac1p}\sum_{k=1}^j\lambda_k   (A')^{(j-k)(\frac1p-1)}} \nn \\
		&\leq \norm{\vp}_{L^\infty(2^jB)} \abs{2^jB}^{1-\frac1p} \pare{(A')^{j(\frac1p-1)}+\sum_{k=1}^j\lambda_k   (A')^{(j-k)(\frac1p-1)}} \nn\\
		& \leq (A')^{j(\frac1p-1)}\norm{\vp}_{\ell_{1/p-1,T}} \pare{1+\sum_{k=1}^j\lambda_k} \label{eq:moldistrib5}  
	\end{align}
	and also
	\begin{align}
		\abs{\int_{A_j} M\vp\, d\mu}\leq \norm{\vp}_{L^\infty(2^jB)}\abs{A_j}^{\frac1{q'}}\norm{M\1_{A_j}}_{L^q} \leq \norm{\vp}_{L^\infty(2^jB)} \lambda_j \abs{2^jB}^{1-\frac1p}\leq \norm{\vp}_{\ell_{1/p-1,T}} \lambda_j. \label{eq:moldistrib2}
	\end{align}
Combining the previous controls and choosing $ j_0\in \N $ such that $ 2^{j_0}r_B\geq T $, we conclude 
	\begin{align}
		\abs{\int M\vp\, d\mu}&\leq \abs{ \int_{2^{j_0}B} M\vp\, d\mu} +\sum_{k\geq j_0+1}  \abs{ \int_{A_k} M\vp\, d\mu} 
		\leq (A')^{j_0(\frac1p-1)}(1+\sum_{j=1}^{\infty}\lambda_{j})\norm{\vp}_{\ell_{1/p-1,T}}\label{eq:moldistrib6}
	\end{align}
that implies 
\begin{equation}\label{distri}
\|M\|_{\ell^{\ast}_{1/p-1,T}(X)} \leq (A')^{j_0(\frac1p-1)}(1+\sum_{j=1}^{\infty}\lambda_{j}),
\end{equation}
where the norm depends on B if $r_{B}<T$ (otherwise if $r_{B} \geq T$ we may choose $j_{0}=0$). Until this moment, we did  not use the moment condition (iii) that will be fundamental in order to obtain the uniform control of \eqref{distri}.  Let $j_0\in \N \cup \{0\} $ such that $ 2^{j_0}r_B<T $ and $ T\leq 2^{j_0+1}r_B $, then 
\begin{align}\label{4.7}
\abs{\int M\vp \, d\mu}& \leq \abs{\int_{B} M(x)(\vp(x)-\vp(x_B)) \, d\mu(x) }+ \abs{\int_{B^{c}} M(x)(\vp(x)-\vp(x_B)) \, d\mu(x) }\\ 
 &+\abs{\vp(x_B)}\abs{\int M(x) \, d\mu(x)}:= (I)+(II)+(III) \nn
\end{align}
where
\begin{align}
(I) &\leq  \norm{\vp}_{\ell_{1/p-1,T}}\abs{B}^{1/p-1}\int_B\abs{ M(x) }\, d\mu(x) \leq \norm{\vp}_{\ell_{1/p-1,T}}, \nn
\end{align}
\begin{align}
(II) &\leq \sum_{j=1}^{j_0} \abs{ \int_{A_j} M(x)(\vp(x)-\vp(x_B)) \, d\mu(x) } 
+ \sum_{j=j_0+1}^{\infty} \abs{ \int_{A_j} M(x)(\vp(x)-\vp(x_B)) \, d\mu(x) } \nn\\
&\leq  \norm{\vp}_{\ell_{1/p-1,T}} \sum_{j=1}^{j_0} \abs{2^jB}^{\frac1p-1} \int_{A_j} \abs{  M(x)} \, d\mu(x) +\sum_{j=j_0+1}^{\infty} 2\norm{\vp}_{L^\infty(2^jB)}\int_{A_j}\abs{M(x)} \, d\mu(x) \nn \\
&\leq  \norm{\vp}_{\ell_{1/p-1,T}}\sum_{j=1}^{j_0} \abs{2^jB}^{\frac1p-1} \abs{A_j}^{1-\frac1q} \lambda_j \abs{2^jB}^{\frac1q-\frac1p} + 2\sum_{j=j_0+1}^{\infty} \norm{\vp}_{L^\infty(2^jB)}\abs{A_j}^{1-\frac1q} \lambda_j \abs{2^jB}^{\frac1q-\frac1p} \nn\\
&\leq\norm{\vp}_{\ell_{1/p-1,T}}\sum_{j=1}^{j_0}  \lambda_j + 2\sum_{j=j_0+1}^{\infty} \norm{\vp}_{L^\infty(2^jB)}\lambda_j \abs{2^jB}^{1-\frac1p}\nn\\
&\leq 2\norm{\vp}_{\ell_{1/p-1,T}}(1+ \sum_{j=1}^{\infty}\lambda_j) \nn \label{eq:moldistrib8}
\end{align}
and  
$$(III) \leq \norm{\vp}_{L^\infty(B(x_B,T))} \abs{B(x_B,T)}^{1-\frac1p} \leq \norm{\vp}_{\ell_{1/p-1,T}}.$$
Plugging into \eqref{4.7} we have
\begin{equation}\label{4.8}
\abs{\int M\vp \, d\mu} \lesssim \norm{\vp}_{\ell_{1/p-1,T}}(1+ \sum_{j=1}^{\infty}\lambda_j),
\end{equation} 
and then $\|M\|_{\ell^{\ast}_{1/p-1,T}(X)} \lesssim (1+\sum_{j=1}^{\infty}\lambda_{j}).$
Analogously, from the size conditions (i) and (ii) of molecules, we have  $ M\in L^q(X) $. In effect, 
\begin{align*}
\norm{M}_{L^q}\leq \norm{M\1_B}_{L^q} +\sum_{k=1}^{\infty} \norm{M\1_{A_k}}_{L^q} 
&\leq \abs{B}^{\frac1q-\frac1p} +\sum_k\lambda_k\abs{2^{k}B}^{\frac1q-\frac1p}
&\leq \abs{B}^{\frac1q-\frac1p}\left(1+\sum_k\lambda_k\right). 
\end{align*}
The same argument shows that for any $k \in \N$ such that $T<2^{k}r_B$ and $\lambda_{k}\neq 0$, then the function $ M_{k}:=(\lambda_k)^{-1}M\1_{A_k} $ is a $ (p,q,T)$-approximate atom supported in the ball $2^kB$. 
In fact, since $ \abs{B(x_0,T)}\leq \abs{2^kB} $ we have
\begin{align*}
	\abs{ \int M \1_{A_k} d\mu}&\leq \norm{M\1_{A_k}}_{L^q}\norm{\1_{A_k}}_{L^{q'}} 
	\leq \lambda_k \abs{ 2^kB }^{\frac1q-\frac1p}\abs{2^kB}^{\frac{1}{q'}}
	\leq \lambda_k \abs{B(x_0,T)}^{1-\frac1p}.
\end{align*}
Assuming without loss of generality that $\lambda_{k} \neq 0$ for any $k \in \N$ and taking 
$ k_0 $ the smallest positive integer such that $ T\leq 2^{k_0} r $ we may write
\begin{align*}
M &= \left(  M\1_B+\sum_{j=1}^{k_0-1} M\1_{A_j}  \right)+ \sum_{j=k_0}^{\infty} \lambda_j M_{j} 
= M\1_{2^{k_{0}-1}B}+ \sum_{j=k_0}^{\infty} \lambda_j M_{j} :=M_{a}+M_{b}
\end{align*}
with $\norm{M_b}_{p,q}\leq \pare{\displaystyle{\sum_{j=k_0}^{\infty} \lambda_j^p}}^{1/p} $. \\

The Definition \ref{def:appmol} covers the approximate molecules defined in \cite{DafniPicon22} when $X=\R^{n}$ equipped with the Lebesgue measure $\mu=\mathcal{L}$ and $\frac{n}{n+1}<p<1$. In fact, recall from \cite[Definition 3.5]{DafniPicon22} that in this setting we say that a mensurable function $M$ is a $(p,q,\lambda,\omega)$-molecule for $1\leq q <\infty$ and $\lambda>n(q/p-1)$ if there exist a ball $B \subset \R^{n}$ and a constant $C>0$ such that
\begin{enumerate}
\item [M1.] $\displaystyle\|M\|_{L^{q}(B)}\leq C(r_{B})^{n\left(\frac{1}{q}-\frac{1}{p}\right)}$
\item [M2.] $\displaystyle \|M|\cdot - x_{B}|^{\frac{\lambda}{n}}\|_{L^{q}(B^{c})}\leq C(r_{B})^{\frac{\lambda}{q}+n\left(\frac{1}{q}-\frac{1}{p}\right)}$
\item [M3.] $\displaystyle \left| \int_{\R^{n}} M(x)dx  \right|\leq \omega $ 
\end{enumerate} 
Choosing $C:=\mathcal{L}(S^{n-1})^{\frac{1}{q}-\frac{1}{p}}$, $\omega:= |B(x_B,T)| ^{1-\frac1p}$ and $\lambda_{k}:= 2^{\frac{\lambda}{q}-k \left[\frac{\lambda}{q}-n\left(\frac{1}{q}-\frac{1}{p}\right)\right] }$, then the conditions
M1-M3 imply (i)-(iii) at Definition \ref{def:appmol} with 
\begin{equation}\label{lp}
\sum_{k=1}^{\infty}(\lambda_{k})^{p}=2^{\frac{\lambda p}{q}} \sum_{k=1}^{\infty} 2^{-kp\left[\frac{\lambda}{q}-n\left(\frac{1}{q}-\frac{1}{p}\right)\right] }<\infty, 
\end{equation}
since $\lambda>n\left({q}/{p}-1 \right)$. We remark that the condition \eqref{lp} is weaker in comparison to 
\eqref{lambda-type-condition} (see also Remark \ref{remark:a-h-mol}). \\

In the next proposition, we show the fundamental property that approximate molecules can be decomposed  in terms of approximate atoms with uniform control in $h^{p,q}_{\#}(X)$-norm.

\begin{proposition}\label{prop:unifboundmol} 
	Let $ 0<p<1\leq q\leq \infty $ and $ M $ be a $(p,q,T,\lambda)$-approximate molecule. Then there exist a sequence $\{ \beta_{j} \}_{j} \in \ell^{p}(\C) $ and $\{ a_{j} \}_{j}  $  of $ (p,q,T) $-approximate atoms such that
		\begin{equation}\label{eq:mol5}
			M=\sum_{j=0}^\infty \beta_j a_j, \quad \quad \text{in} \,\,\, L^{q}(X) 
		\end{equation}
with $ \displaystyle{\bigg(\sum_j\abs{\beta_j}^p\bigg)^{1/p}\leq C_{A,p}\| \lambda \|_p} $. Moreover,  the convergence of \eqref{eq:mol5} is in $ \ell^{\ast}_{1/p-1,T}(X) $ and $\| M \|_{{p,q} }\leq C_{A,p} \| \lambda \|_p$.
\end{proposition}

\begin{proof}
	Let $M$ be a $(p,q,T,\lambda)$-approximate molecule concentrated on $B=B(x_{B},r_{B})$ and consider  $A_k = 2^kB \setminus 2^{k-1}B$ for $k \geq 1$ and $ A_0:=B $. Define
	\[
	M_k := M \, \1_{A_k} - \frac{\1_{2^kB}}{|2^kB|} \int_{X}M \, \1_{A_k} d\mu \quad \text{and} \quad \widetilde{M}_k := \frac{\1_{2^kB}}{|2^kB|} \int_{X}M \, \1_{A_k} d\mu.
	\]
	Then
	\begin{equation}\label{eq:ubm1}
		M=\sum_{k=0}^{\infty} M_k + \sum_{k=0}^{\infty} \widetilde{M}_k, 
	\end{equation}
where each $M_k/\left(2\lambda_k\right)$ is a $(p,q)$-atom in $H_{cw}^p(X)$ supported in $2^{k}B$ (here consider $ \lambda_0:=1 $). In effect, it is straightforward from the definition that $\supp(M_k)\subset 2^kB$ and that it satisfies vanishing moment conditions. Moreover
	\begin{align}
		\norm{M_k}_{L^q}\leq \norm{M \1_{A_k}}_{L^q}+\abs{\int_X M\1_{A_k}\, d\mu}\frac{\norm{\1_{2^kB}}_{L^q}}{|2^kB|} &\leq \lambda_k\abs{2^kB}^{\frac1q-\frac1p}
		+\norm{M\1_{A_k}}_{L^q}\norm{\1_{A_k}}_{L^{q'}}\abs{2^kB}^{\frac1q-1} \nn\\
		& \leq 2\lambda_k \abs{2^kB}^{\frac1q-\frac1p}. \label{eq:mol3}
	\end{align}
	Thus, we may write $\displaystyle \sum_{k=0}^{\infty} M_{k}= \sum_{k=0}^{\infty} 2\lambda_k\pare{{M_k}/{2\lambda_k}} $ is an element of $ h^{p,q}_{\#}(X) $, and moreover
	$
	\displaystyle{\left\| \sum_{k=0}^{\infty} M_k \right\|_{p,q} \leq 2 \left( \sum_{k=0}^{\infty} (\lambda_k)^p \right)^{1/p} < \infty.}
	$
	To control the second term in \eqref{eq:ubm1}, let
	$$
	\chi_k=\frac{\1_{2^kB}}{|2^kB|}, \quad \widetilde{m}_k = \int_{X} M \, \1_{A_k} d\mu, \quad \text{and} \quad N_j = \sum_{k=j}^{\infty} \widetilde{m}_k.
	$$
	Then,
	\[
	\sum_{k=0}^{\infty} \widetilde{M}_k = \sum_{k=0}^{\infty} \chi_{k} \widetilde{m}_k = \sum_{k=0}^{\infty} \chi_{k} [N_k-N_{k+1}] = \chi_{0} N_0 + \sum_{k=0}^{\infty} \sum_{j=k+1}^{\infty} b_{k,j},
	\]
with $b_{k,j}:=[\chi_{k+1}-\chi_{k}] \widetilde{m}_j$.  
We claim that $b_{k,j}$ is a multiple of $(p,q)$-atom  in $H_{cw}^p(X)$ supported in $2^{k+1}B$. In fact, clearly $\int b_{k,j}\, d\mu=0$ and moreover 
	\begin{align*}
		\Norm{[\chi_{k+1}-\chi_{k}] \widetilde{m}_j}_{L^q(X)}&\leq \abs{ \int M\1_{A_j}d\mu }(\norm{\chi_{k+1}}_{L^q(X)}+\norm{\chi_k}_{L^q(X)})\\
		&\leq \norm{ M\1_{A_j} }_{L^q(X)}\norm{ \1_{A_j} }_{L^{q'}(X)}\pare{ \frac{1}{\abs{2^{k+1}B}}\norm{\1_{2^{k+1}B}}_{L^q(X)} + \frac{1}{\abs{2^{k}B}}\norm{\1_{2^{k}B}}_{L^q(X)}} \\
		&\leq \lambda_j \abs{ 2^jB }^{\frac1q-\frac1p} \abs{A_j}^{1-\frac1q}\frac{ 2\abs{2^{k+1}B}^{\frac1q} }{\abs{ 2^kB }}\\
		&= \left(2\lambda_j \abs{ 2^jB }^{\frac1q-\frac1p} \abs{A_j}^{1-\frac1q} {\frac{\abs{ 2^{k+1}B }^{\frac1p}}{\abs{2^kB}}} \right) \abs{2^{k+1}B}^{\frac1q - \frac1p}. 
	\end{align*}
Now, note that
\begin{equation}\label{eq:mol10}
\abs{ 2^jB }^{\frac1q-\frac1p} \abs{A_j}^{1-\frac1q} {\frac{\abs{ 2^{k+1}B }^{\frac1p}}{\abs{2^kB}}}\leq {A'} \abs{ 2^jB }^{1-\frac1p}\abs{ 2^{{k+1}}B }^{\frac1p -1}\leq  {A'} 
\end{equation}
where for $j \geq k+1$ we use the simple control 
$$ 
{|2^{k+1}B|\leq |2^{j}B|.}
$$
Thus  $b_{k,j}/\left( 2\lambda_j {A'} \right)$ is a $(p,q)$-atom in $H_{cw}^p(X)$ supported in $2^{k+1}B$,
 $\displaystyle{ \sum_{k=0}^{\infty} \sum_{j=k+1}^{\infty} b_{k,j} \in h^{p,q}_{\#}(X) }$ and moreover by \eqref{lambda-type-condition}
	\[
	\left\| \sum_{k=0}^{\infty} \sum_{j=k+1}^{\infty} b_{k,j} \right\|_{h^{p,q}_{\#}} \lesssim \left[ \sum_{k=0}^{\infty} \sum_{j=k+1}^{\infty} (\lambda_j)^p \right]^{1/p} \sim \left[ \sum_{j=1}^{\infty} j(\lambda_j)^p \right]^{1/p}<\infty.
	\]	

Also note that, by \eqref{eq:mol10} we also have 
\begin{align}
	\sum_{k=0}^{\infty}\sum_{j=k+1}^{\infty} \norm{b_{k,j}}_{L^q}\leq 2 A' \sum_{k=0}^{\infty}\sum_{j=k+1}^{\infty} \lambda_j \abs{ 2^{k+1}B }^{\frac1q-\frac1p}
	&\leq 2 A' \abs{B}^{\frac1q-\frac1p}\sum_{k=0}^{\infty}\sum_{j=k+1}^{\infty} \lambda_j \nn\\
	&\lesssim  2 A' \abs{B}^{\frac1q-\frac1p}\sum_{j=1}^{\infty} j\lambda_j^p. \label{eq:mol2}
\end{align}

Finally, we claim that $ \chi_0 N_0=\abs{B}^{-1}( \int_X M d\mu)\1_B $ is a multiple constant of a $ (p,q,T) $-approximate atom supported in $ B $. First note that 
	\begin{align*}
		\sum_{j=0}^{\infty}\abs{ \int M\1_{A_j} \, d\mu }\leq \sum_{j=0}^{\infty} \norm{M\1_{A_j}}_{L^q(X)} \norm{\1_{A_j}}_{L^{q'}(X)} \leq \sum_{j=0}^{\infty} \lambda_j \abs{ 2^jB }^{\frac1q-\frac1p}\abs{2^jB}^{1-\frac1q} 
		&\leq \sum_{j=0}^{\infty} \lambda_j \abs{2^jB}^{1-\frac1p} \\ &\leq \abs{B}^{1-\frac1p}\sum_{j=0}^{\infty} \lambda_j .
	\end{align*}
	Then, we have
	\begin{align*}
		\norm{\chi_0 N_0}_{L^q(X)}
		\leq \abs{B}^{\frac1q-1}\abs{ \sum_{j=0}^{\infty}\int M\1_{A_j} d\mu }  \leq \abs{B}^{\frac1q-1}\abs{B}^{1-\frac1p}\sum_{j=0}^{\infty} \lambda_j 
		 = \abs{B}^{\frac1q-\frac1p} \|\lambda\|_{\ell^1}
	\end{align*}
and clearly the approximate moment condition follows immediately from (iii) since \\$\int_{X}\chi_0 N_0 \, d\mu=\int_X M d\mu$.	

		On the other hand, from \eqref{eq:moldistrib2} we have for any $ \vp\in \ell_{1/p-1,T}(X)$
		\[
		\abs{\int \bigg(M-\sum_{j=1}^{m} M_k-\sum_{j=1}^m \widetilde{M}_k \bigg)\vp \, d\mu}  = \abs{\int_{X\setminus 2^m B} M \vp \, d\mu} \leq \sum_{j=m+1}^\infty 
		\abs{\int_{A_j}M\vp d\mu}\leq \norm{\vp}_{\ell_{\frac1p-1,T} }\sum_{j=m+1}^\infty\lambda_j.
		\]
		Taking $ m $ sufficiently large this shows the convergence in
		\begin{equation}\label{eq:mol4}
			M=\chi_0 N_0+\sum_{k=0}^{\infty} M_k+\sum_{k=0}^\infty \sum_{j=k+1}^{\infty} b_{k,j}
		\end{equation}
		is also in $ \ell^{\ast}_{1/p-1,T}(X) $.
	
	Summarizing we have $ M \in h^{p,q}_{\#}(X) $ with $ \norm{M}_{p,q}\leq C \| \lambda\|_p $ where $C=C(A',p)>0$, and also by \eqref{eq:mol3} and \eqref{eq:mol2}  the decomposition in \eqref{eq:mol4} converges in $ L^q $-norm.
	
\end{proof}

A direct consequence in the proof of last proposition is the following:

\begin{corollary} \label{corollary:molecular-decomp}
	Let $\{M_j\}_{j}$ be a sequence of $(p,q,T,\lambda)$-approximate molecules and $\{\beta_j\}_{j} \in \ell^{p}(\C)$. Then $\displaystyle f :=\sum_{j=0}^{\infty} \beta_j \, M_j  \in h^{p,q}_{\#}(X)$ and 
	\begin{equation}\label{enda}
	\|f\|_{p,q} \leq C_{A,p} \, \| \lambda \|_p \, \bigg( \sum_{j=1}^{\infty} |\beta_j|^p\bigg)^{1/p}. 
	\end{equation}
\end{corollary}

\begin{proof} We first remark that
$\displaystyle f \in \ell^{\ast}_{1/p-1,T}(X)$. In fact, it follows from \eqref{4.8}
that
\begin{align*}
		\|f\|_{\ell^{\ast}_{1/p-1,T}(X)} \leq 2\bigg(1+\sum_{j=1}^{\infty}\lambda_j\bigg) \bigg( \sum_{j=1}^{\infty} |\beta_j|^p\bigg)^{1/p}. 
	\end{align*}
In particular, the convergence $\displaystyle{\sum_{j=0}^{\infty} \beta_j \, M_j}$ is in $\ell^{\ast}_{1/p-1,T}(X)$ and by \eqref{eq:mol5} each $\displaystyle M_{j}=\sum_{k=1}^{\infty}\theta_{jk}a_{jk}$ with   $a_{jk}$ $(p,q,T,\lambda)$-approximate atoms and $ \displaystyle{\bigg(\sum_k\abs{\theta_{jk}}^p\bigg)^{1/p}\leq C_{A,p}\| \lambda \|_p} $, it follows
$$f=\sum_{jk}\beta_{j}\theta_{jk}a_{jk}$$
and analogously as done in \eqref{eq:414} follows \eqref{enda}.    	
\end{proof}

\begin{remark} \label{remark:a-h-mol}
	Condition \eqref{lambda-type-condition} can be weakened to the natural one $\displaystyle{\sum_{k=1}^\infty (\lambda_k)^p <\infty}$, when 
in addition to \eqref{upper} we also have the next special case of the \emph{reverse  doubling condition}: there exists a $ A''\in (0,1] $ such that for all $ x\in X $, $ r>0 $ and $ \lambda\geq 1 $, we have
	\begin{equation}\label{rupper}
		A''\lambda^{\gamma} \abs{B(x,r)} \leq \abs{B(x,\lambda r)}.
	\end{equation}
In effect, from the previous proof we have obtained that $b_{k,j}/\left( 2A'\lambda_j \abs{2^jB}^{1-\frac1p}\abs{2^{k+1}B}^{\frac1p-1} \right)$
	is a $ (p,q) $-atom supported in $ 2^{k+1}B $. Then, using \eqref{upper} and \eqref{rupper}, the $ h^{p,q}_\#$-norm of $\displaystyle{ \sum_{k=0}^{\infty} \sum_{j=k+1}^{\infty} b_{k,j} }$  is bounded up to a constant by
	\[
	\sum_{k=0}^{\infty} \corc{2^{(k+1)(1-p)\textcolor{olive}{\gamma}}\sum_{j=k+1}^{\infty} \lambda_j^p 2^{j(p-1)\textcolor{olive}{\gamma}} } \leq \sum_{j=1}^\infty \lambda_j^p\sum_{k=0}^{j-1} 2^{(p-1)\textcolor{olive}{\gamma} k} \leq \frac{1}{1-2^{(p-1)\textcolor{olive}{\gamma}}}\sum_{j=1}^\infty \lambda_j^p. 
	\] 
Examples of doubling measures satisfying the reverse condition \eqref{rupper} are giving by Alhfors regular measures, i.e. measures characterized by the property $ \abs{B(x,r)}\approx r^\theta$ for all $ r>0 $ and some $ \theta>0 $.
\end{remark}

\section{Relation between $ h^{p}_{\#}(X) $ and $h^p(X)$ }\label{sec:rel}

In this section, we present the relation between $h^{p}_{\#}(X)$ and the  local Hardy space $h^{p}(X)$ introduced in \cite{HeYangWen21}. We will set $T=1$ in the definition of $h^{p}_{\#}(X)$  and we always  assume $ \mu $ a Borel regular measure in $ X $.

We start presenting the definition of test functions considered by \cite[Definition 2.1]{HeYangWen21}.

\begin{definition}\label{def:testfuncchina}
	Let $ x_0\in X $, $ r\in (0,\infty) $, $ \beta\in(0,1] $ and $ \theta\in (0,\infty) $. A function $ f $ defined on $ X $ is called by a \emph{test function of type} $ (x_0,r,\beta,\theta) $, denoted by $ f\in \mathcal{G}(x_0,r,\beta,\theta) $, if there exists a positive constant $ C $ such that
	\begin{enumerate}[(i)]
		\item (\emph{Size condition}) for all $ x\in X $,
		\[
		\abs{f(x)}\leq C\frac{1}{V_r(x_{0})+V(x_0,x)}\left[ \frac{r}{r+d(x_0,x)} \right]^\theta
		\]
		\item (\emph{Regularity condition}) for any $ x,y\in X $ satisfying $ d(x,y)\leq (2A_0)^{-1}(r+d(x_0,x)) $,
		\[
		\abs{f(x)-f(y)}\leq C \frac{1}{V_r(x_0)+V(x_0,x)}\left[ \frac{d(x,y)}{r+d(x_0,{x})}\right]^\beta\left[ \frac{r}{r+d(x_0,x)}\right]^\theta. 
		\]
	\end{enumerate}
	For $ f\in\mathcal{G}(x_0,r,\beta,\theta) $ it is defined 
	\[
	\norm{f}_{\mathcal{G}(x_0,r,\beta,\theta)}:= \inf \lla{C\in (0,\infty) : C\text{ satisfying (i) e (ii)} },
	\]
	and also the set
	\[
	\dot{\mathcal{G}}(x_0,r,\beta,\theta):= \lla{f\in \mathcal{G}(x_0,r,\beta,\theta)  \ : \ \int_X f\,d\mu=0 }.
	\]
	equipped with norm $ \norm{\cdot}_{\dot{\mathcal{G}}(x_0,r,\beta,\theta)}:=\norm{\cdot}_{\mathcal{G}(x_0,r,\beta,\theta)} $.
\end{definition}

We highlight in the following remark some properties of the set $ \mathcal{G}(x_0,r,\beta,\theta) $ discussed in \cite{HeYangWen21}.

\begin{remark}\label{r6.2}
	\textcolor{white}{.}
	\begin{enumerate}	
		\item[(i)] For each $ x_0 $ fixed, we have $ \mathcal{G}(x,r,\beta,\theta) = \mathcal{G}(x_0,1,\beta,\theta)$ for any $ x\in X $ and $ r>0 $. Moreover, there exists $C=C(x,r)>0$ such that
		\begin{equation}\label{en:num111}
			C\norm{f}_{\mathcal{G}(x_0,1,\beta,\theta)} \leq  \norm{f}_{\mathcal{G}(x,r,\beta,\theta)} \leq C^{-1}\norm{f}_{\mathcal{G}(x_0,1,\beta,\theta)}.
		\end{equation}
		For this, we denote $\mathcal{G}(\beta,\theta)= \mathcal{G}(x_0,1,\beta,\theta) $ and $\dot{\mathcal{G}}(\beta,\theta)= \dot{\mathcal{G}} (x_0,1,\beta,\theta) $. \\
		\item[(ii)] $ \mathcal{G}(x_0,1,\beta,\theta) $ is a Banach space with the norm $ \norm{\cdot }_{\mathcal{G}(x_0,1,\beta,\theta)} $. \\
		
		\item[(iii)] If $ 0<\beta_1<\beta_2 $, then $ \mathcal{G}(\beta_2,\theta) \subset \mathcal{G}(\beta_1,\theta)$ continuously, for all $ \theta>0 $. Analogously, if $ 0<\theta_1<\theta_2 $ then $ \mathcal{G}(\beta,\theta_2) \subset \mathcal{G}(\beta,\theta_1)$ continuously, for all $ \beta \in (0,1] $.\\
		\item[(iv)] For $ \ve\in (0,1] $ and $ \beta,\theta\in (0,\ve] $, it is denoted by $ \mathcal{G}^\ve_{0}(\beta,\theta) $ [resp. $ \dot{\mathcal{G}}^\ve_{0}(\beta,\theta) $ ] the completion of $ \mathcal{G}(\ve,\ve) $ [resp. $ \dot{\mathcal{G}}(\ve,\ve) $ ] in $ \mathcal{G}(\beta,\theta) $, and it is defined the norms
		$
		\norm{\cdot}_{\mathcal{G}^\ve_0(\beta,\theta)}:= \norm{\cdot}_{\mathcal{G}(\beta,\theta)}$, 
		$
		\norm{\cdot}_{\dot{\mathcal{G}}^\ve_0(\beta,\theta)}:= \norm{\cdot}_{\mathcal{G}(\beta,\theta)}.
		$\\
		\item[(v)] The spaces $ \mathcal{G}^\ve_{0}(\beta,\theta) $, $ \dot{\mathcal{G}}^\ve_{0}(\beta,\theta) $ are closed subspaces of $ \mathcal{G}(\beta,\theta) $.\\
	\end{enumerate}	
\end{remark}

The \emph{space of distributions} associated to $ \mathcal{G}^\ve_0(\beta,\theta) $  is denoted by $ (\mathcal{G}^\ve_0(\beta,\theta))^\ast $  equipped with the weak* topology. \\

An important class of distributions on $ (\mathcal{G}^\ve_0(\beta,\theta))^\ast $ for all $ \beta,\theta>0 $ is given by functions $f  \in L^{q}(X)$ for  $ q\in [1,\infty]$ associated to the functional 
$$ \Lambda_f(\varphi):=\int f\varphi d\mu, \quad \quad \forall \, {\varphi} \in {\mathcal{G}(\beta,\theta)}. $$ 
In fact, for every $ \varphi \in \mathcal{G}(\beta,\theta)=\mathcal{G}(x_0,1,\beta,\theta)$, by the size condition in Definition \ref{def:testfuncchina} and H\"older inequality for $1<q<\infty$, we have  
\begin{align*}
	\abs{\int f(y)\varphi(y)d\mu(y)}&\leq C \norm{f}_{L^q} \corc{\int \frac{1}{(V_1(x_0)+V(x_0,y))^{q'}}\, \frac{1}{(1+d(x_0,y))^{q'\theta}} d\mu(y)}^{\frac1{q^\prime}}\\
	& \leq C \norm{f}_{L^q} \corc{{V_1(x_0)^{-\frac{1}{q}}} } \corc{\int \frac{1}{V_1(x_0)+V(x_0,y)}\frac{1}{(1+d(x_0,y))^{q'\theta}} d\mu(y)}^{\frac1{q^\prime}}\\
	&\leq C_{q,\theta} \corc{{V_1(x_0)^{-\frac{1}{q}}} } \norm{f}_{L^q(X)},
\end{align*}
where the last inequality follows by Proposition \ref{prop:fv} (iii). The case $q=\infty$ holds analogously chosen  $q'=1$. For $ q=1 $, it is sufficient to remark that $|\varphi(y)|\leq C V_1(x_0)^{-1} $. 
Summarizing   
\begin{equation}\label{lq}
	\abs{\Lambda_f(\varphi)}\lesssim \norm{f}_{L^q}, \quad \quad \forall \,  \norm{\varphi}_{\mathcal{G}(\beta,\theta)} \leq 1.
\end{equation}

The next result is useful  to compare the spaces of test functions $ \mathcal{G}^\ve_0(\beta,\theta) $ and $ \ell_{\frac1p-1,1}(X) $ .

\begin{proposition}\label{prop:RTF}
	Let $ \beta\in(0,1] $, $ \theta\in (0,\infty)$. If $ \vp\in \mathcal{G}(\beta,\theta) $ then there exists a constant $ C>0 $ independent of  $ \vp $ such that
	\begin{equation}\label{eq:667}
		\abs{\vp(x)-\vp(y)}\leq C\norm{\vp}_{\mathcal{G}(\beta,\theta)} V(x,y)^{\frac\beta\gamma}, \quad \quad \forall \, x,y \in X.
	\end{equation}	
	Let $x_{B} \in X$ and $r_{B}>0$ fixed. Then 	
	\begin{equation}\label{6.2}
		\abs{\vp(x)-\vp(y)}\leq C\norm{\vp}_{\mathcal{G}(\beta,\theta)} \abs{B(x_B,r_B)}^{\frac\beta\gamma},  \quad \quad \forall \, x,y \in B(x_B,r_B). 
	\end{equation}
	Moreover, if $ r_B\geq 1 $ then 
	\begin{equation}\label{6.3}
		\abs{\vp(x)}\leq C \norm{\vp}_{\mathcal{G}(\beta,\theta)} \abs{ B(x_B,r_B)}^{\frac\theta\gamma},  \quad \quad \forall \, x \in B(x_B,r_B). 
	\end{equation}
	In particular, if $ \vp\in \mathcal{G}({\beta,\beta}) $ then we have that $ \vp\in \ell_{\frac\beta\gamma,1}(X)$ and 
	\begin{equation}\label{6.5}
		\norm{ \vp }_{\ell_{\frac\beta\gamma,1}}\leq C \norm{\vp}_{\mathcal{G}(\beta,\beta)}.
	\end{equation}
\end{proposition}

\begin{proof} 
	The proof follows the steps from \cite[Lemma 4.15]{HeHanLiLYY19}. We start by showing \eqref{eq:667}. Let $ \vp\in \mathcal{G}(\beta,\theta) = \mathcal{G}(x_0,1,\beta,\theta) $ for some $ x_0\in X $. For any  $ x,y\in X $, we suppose first 	
	that $ d(x,y) \leq (2A_0)^{-1}[1+d(x_0,x)]$. Then, by the regularity condition 
	we have 
	\begin{align}
		\abs{\vp(x)-\vp(y)}&\leq \norm{\vp}_{\mathcal{G}(\beta,\theta)} \corc{ \frac{d(x,y)}{1+d(x_0,x)}}^\beta  \frac{1}{V_1(x_0)+V(x_0,x)} \corc{ \frac{1}{1+d(x_0,x)}}^\theta \nonumber \\
		&\leq \norm{\vp}_{\mathcal{G}(\beta,\theta)} \frac{1}{V_1(x_0)} \corc{ \pare{\frac{1+d(x_0,x)}{d(x,y)}}^{-\gamma} }^{\frac\beta\gamma}  \label{eq:part1}
	\end{align}
	In order to continue the previous estimate, note that since $\displaystyle{ \frac{1+d(x_0,x)}{d(x,y)}\geq 2A_0>1 }$, we may write
	\[
	\abs{ B\pare{x, {1+d(x_0,x)} } }=\abs{ B\pare{x, d(x,y) \frac{1+d(x_0,x)}{d(x,y)} } }\leq A'\pare{\frac{1+d(x_0,x)}{d(x,y)}}^{\gamma} \abs{B(x,d(x,y))}
	\]
	and then
	$$
	\pare{\frac{1+d(x_0,x)}{d(x,y)}}^{-\gamma} \leq A' \, \frac{|B(x,d(x,y))|}{|B(x,1+d(x_0,x))|} \leq C \, A' \, \frac{|B(x,d(x,y))|}{V_1(x_0)+V(x_0,x)},
	$$
	where the last inequality follows from Proposition \ref{prop:fv} (i) and the constant $C>0$ is independent of $x_{0},x,y$.
	Combining the previous estimates in \eqref{eq:part1} we obtain
	\begin{align*}
		\abs{\vp(x)-\vp(y)}
		&\leq \norm{\vp}_{\mathcal{G}(\beta,\theta)} \frac{(CA^\prime)^{\beta/\gamma}}{V_1(x_0)}  \left(\frac{\abs{ B(x,d(x,y))}}{\corc{V_1(x_0)+ V(x_0,x)} }
		\right)^{\frac\beta\gamma} \\
		&\leq \norm{\vp}_{\mathcal{G}(\beta,\theta)} \left( \frac{(CA^\prime)^{\beta/\gamma}}{V_1(x_0)^{1+\frac\beta\gamma}} \right) 
		\abs{ B(x,d(x,y))}^{\frac\beta\gamma} \\
		&= \norm{\vp}_{\mathcal{G}(\beta,\theta)} \left( \frac{(CA^\prime)^{\beta/\gamma}}{V_1(x_0)^{1+\frac\beta\gamma}} \right) 
		V(x,y)^{\frac\beta\gamma}.
	\end{align*}
	On the other hand, if $d(x,y) > (2A_0)^{-1}[1+d(x_0,x)]$ we first note that
	$$ 
	\abs{B(x,1+d(x,x_0))}\leq  \abs{B(x,(2A_{0})d(x,y))}\leq A^\prime (2A_0)^\gamma \abs{B(x,d(x,y))} = A^\prime (2A_0)^\gamma \, V(x,y)
	$$
	and then
	$$
	V_1(x_0) = |B(x_0,1)| \leq |B(x,1+d(x,y))| \leq A^\prime (2A_0)^\gamma \, V(x,y).
	$$
	Under the size condition, the previous estimate and again Proposition \ref{prop:fv} (i) we conclude 
	\begin{align*}
		\abs{\vp(x)-\vp(y)}
		&\leq \norm{\vp}_{\mathcal{G}(\beta,\theta)}\pare{   \frac{(1+d(x_0,x))^{-\theta}}{V_1(x_0)+V(x_0,x)} + \frac{(1+d(x_0,y))^{-\theta}}{V_1(x_0)+V(x_0,y)} } \\
		&\leq 2\norm{\vp}_{\mathcal{G}(\beta,\theta)} \frac{1}{V_1(x_0)} =  2\norm{\vp}_{\mathcal{G}(\beta,\theta)} \frac{1}{V_1(x_0)^{1+\frac\beta\gamma}} V_1(x_0)^{\frac\beta\gamma} \\
		&\leq \norm{\vp}_{\mathcal{G}(\beta,\theta)} \frac{2\corc{A^\prime (2A_0)^\gamma}^{\frac\beta\gamma} }{V_1(x_0)^{1+\frac\beta\gamma}} V(x,y)^{\frac\beta\gamma}.
	\end{align*}
	Summing up the two cases, we have shown that there exists $ C>0 $ independent of $ \vp $, such that \eqref{eq:667} holds.
	
	The estimate \eqref{6.2} is a particular case of the previous one. In fact, if $ x,y\in B(x_B,r_B) $, applying the quasi-triangular inequality we can show $ B(x,d(x,y))\subset B(x_B, A_0(2A_0+1)r_B) $, which implies $V(x,y) \leq |B(x_B, A_0(2A_0+1)r_B)|$. Then, from \eqref{eq:667}
	\begin{align*}\label{eq:668}
		\abs{\vp(x)-\vp(y)} &\leq C\norm{\vp}_{\mathcal{G}(\beta,\theta)} V(x,y)^{\frac{\beta}{\gamma}} \\
		&\leq C \, [A_0(2A_0+1)]^{\gamma} \, \norm{\vp}_{\mathcal{G}(\beta,\theta)} |B(x_B,r_B)|^{\frac{\beta}{\gamma}}.
	\end{align*}
	
	Now we move on to prove \eqref{6.3}.  
	Let $ x\in B(x_B,r_B) $ and assume first that $ r_B\leq (2A_0)^{-1}[1+d(x,x_0)] $. Noticing that $r_{B}\geq 1$ and proceeding as before, we may estimate   
	\begin{align*}
		\abs{\vp(x)}\leq \norm{\vp}_{\mathcal{G}(\beta,\theta)}   \frac{1}{V_1(x_0)+V(x_0,x)} \corc{ \frac{1}{1+d(x_0,x)}}^\theta
		& \leq \norm{\vp}_{\mathcal{G}(\beta,\theta)} \frac{1}{V_1(x_0)}
		\corc{ \pare{\frac{1+d(x_0,x)}{r_B}}^{-\gamma}}^{\frac\theta\gamma} \\
		&{\lesssim} \norm{\vp}_{\mathcal{G}(\beta,\theta)} \frac{(A^\prime)^{\frac\theta\gamma}}{V_1(x_0)}\corc{ \frac{\abs{B(x,r_B)}}{\abs{B(x,{1+d(x_0,x)})}} }^{\frac\theta\gamma} \\
		& \lesssim \norm{\vp}_{\mathcal{G}(\beta,\theta)} \frac{(A^\prime)^{\frac\theta\gamma}}{V_1(x_0)^{1+\frac\theta\gamma}} {\abs{B(x,r_B)}}^{\frac\theta\gamma}\\
		& \lesssim \norm{\vp}_{\mathcal{G}(\beta,\theta)} \frac{(A^\prime)^{2\frac\theta\gamma} (2A_0)^\theta  }{V_1(x_0)^{1+\frac\theta\gamma}} {\abs{B(x_B,r_B)}}^{\frac\theta\gamma}.
	\end{align*}
	Suppose now that $ (2A_0)^{-1}(1+d(x_0,x))\leq r_B $. Again, from size condition of $ \vp $, the Proposition \ref{prop:fv} item (i) and the inclusion $ B(x,r_B)\subset B(x_B,2A_0r_B) $, we have
	\begin{align*}
		\abs{\vp(x)} \leq \norm{\vp}_{\mathcal{G}(\beta,\theta)} \frac{1}{V_1(x_0)}&= \norm{\vp}_{\mathcal{G}(\beta,\theta)} \frac{1}{[V_1(x_0)]^{1+\frac\theta\gamma}} [V_1(x_0)]^{\frac\theta\gamma} \\
		& \leq \norm{\vp}_{\mathcal{G}(\beta,\theta)} \frac{1}{[V_1(x_0)]^{1+\frac\theta\gamma}} \abs{B(x_0,1+d(x_0,x))}^{\frac\theta\gamma} \\
		& \leq \norm{\vp}_{\mathcal{G}(\beta,\theta)} \frac{1}{[V_1(x_0)]^{1+\frac\theta\gamma}} \abs{B(x,2A_0r_B)}^{\frac\theta\gamma}\\
		&\leq \norm{\vp}_{\mathcal{G}(\beta,\theta)} \frac{(A^{\prime})^{\frac\theta\gamma}(2A_0)^\theta}{[V_1(x_0)]^{1+\frac\theta\gamma}} \abs{B(x,r_B)}^{\frac\theta\gamma} \\
		& \leq \norm{\vp}_{\mathcal{G}(\beta,\theta)} \frac{(A^{\prime})^{\frac\theta\gamma}(2A_0)^\theta}{[V_1(x_0)]^{1+\frac\theta\gamma}} \abs{B(x_B,2A_0r_B)}^{\frac\theta\gamma} \\
		&\leq \norm{\vp}_{\mathcal{G}(\beta,\theta)} \frac{(A^{\prime})^{2\frac\theta\gamma}(2A_0)^{2\theta}}{[V_1(x_0)]^{1+\frac\theta\gamma}} \abs{B(x_B,r_B)}^{\frac\theta\gamma}
	\end{align*}
	Combining the both inequalities we have the inequality \eqref{6.3} desired. 
	
	The inclusion \eqref{6.5} follows directly from \eqref{6.2} and \eqref{6.3} taking $\theta=\beta$.   
\end{proof}

\begin{remark}
	In order to prove \eqref{eq:667}, it is sufficient to assume that
	\[
	\abs{\varphi(x)}\leq \frac{C}{V_r(x_{0})},\quad \forall x \in X 
	\] 
	and for any $ x,y\in X $ satisfying $ d(x,y)\leq (2A_0)^{-1}(r+d(x_0,x)) $,
	\[
	\abs{\varphi(x)-\varphi(y)}\leq  \frac{C}{V_r(x_0)}\left[ \frac{d(x,y)}{r+d(x_0,x)}\right]^\beta. 
	\] \\
\end{remark}

In what follows, we will always denote by $ \eta>0 $ the Hölder regularity index of wavelets given in \cite[Theorem 7.1]{AusHyto13}. For any $ \theta,\beta\in (\gamma(\frac1p-1),\eta)$, from Remark \ref{r6.2} (iii) and the relation \eqref{6.5} we get
\begin{equation}
	\mathcal{G}(\beta,\theta) \subset \mathcal{G}\left(\gamma(\tfrac1p-1),\gamma(\tfrac1p-1)\right)\subset \ell_{1/p-1,1}(X).
\end{equation}
Moreover, denoting by $\mathcal{G}_0^\eta(\beta,\theta)$ the closure of $\mathcal{G}(\eta,\eta)$  in $\mathcal{G}_0^\eta(\beta,\theta)$, i.e.
$\mathcal{G}_0^\eta(\beta,\theta):=\ovl{\mathcal{G}(\eta,\eta)}^{\mathcal{G}(\beta,\theta)}$, we have
\begin{equation}\label{eq:684}
	{\mathcal{G}_0^\eta(\beta,\theta) \subset  \ell_{1/p-1,1}(X)},
\end{equation}
continuously due to Remark \ref{r6.2} (iii) and \eqref{6.5}.

	\begin{remark}\label{rem:hpcwdistinG}
		Note that \eqref{eq:684} shows that elements in $ \ell_{1/p-1,1}^\ast(X) $ define elements in $ (\mathcal{G}_0^\eta(\beta,\theta))^\ast $. In particular, elements in $ h^p_\#(X) $ (thus elements in $ h^p_{cw}(X)$ also) define elements in $ (\mathcal{G}_0^\eta(\beta,\theta))^\ast $.
	\end{remark}

\begin{definition}
	Fix $ \beta, \theta\in (0,\eta) $ and $ f\in (\mathcal{G}^\eta_0(\beta,\theta))^\ast $. The \emph{local grand maximal function} of $ f $ is defined as
	\[
	f^*_0(x):=\sup \lla{ \abs{\ip{f,\varphi}} \ : \  \varphi\in \mathcal{G}^\eta_0(\beta,\theta) ,\ \norm{\varphi}_{\mathcal{G}(x,r,\beta,\theta)}\leq 1, \text{ for some } r\in (0,1] }.
	\]
\end{definition}

In the next lemma we show a convenient estimate of the local grand maximal function of atoms-type functions. It will be useful later to show that approximate atoms are uniformly bounded in norm.

\begin{lemma}\label{lem:DA}
	Let $ p\in(\frac{\gamma}{\gamma+\eta},1] $, $ q\in(p,\infty]\cap [1,\infty] $.
	Then, there exists a constant $ C>0 $ such that, 
	if $ a $ is a measurable function supported in a ball $ B=B(x_B,r_B) $ such that $ \norm{a}_{L^q}\leq \abs{B}^{\frac1q-\frac1p} $, for any $ x\in X $
	\begin{align}
		a_0^*(x)\leq C \mathcal{M}a(x)\mathds{1}_{B^{*}}(x)&+ C \mathds{1}_{X \backslash B^{*}}(x) 
		\left( \corc{ \frac{r_B}{d(x_B,x)} }^{\beta \land \theta} \frac{ |B|^{1-1/p}}{V(x_{B},x)} \right. \label{eq:decompatom}\\
		&\quad +\left.  \abs{\int ad\mu}   \frac{1}{V(x,x_B)}\left[ \frac{1}{1+  d(x,x_B)} \right]^\theta  \right) \nn 
	\end{align}
	when $ r_B \leq 1$, and
	\begin{align}\label{eq:decompatom2}
		a_0^*(x)\leq C \mathcal{M}a(x)\mathds{1}_{B^{\ast}}(x)+ C \mathds{1}_{X \backslash B^{*}}(x)\pare{\corc{ \frac{r_B}{d(x_B,x)} }^{\beta \land \theta} \frac{ |B|^{1-1/p}}{V(x_B,x)} }
	\end{align}
	when $ r_B> 1 $, where $B^{*}:=B(x_B,2A_0r_B)$ and $ \beta \land \theta $ denotes $ \min\lla{\beta,\theta} $. 
\end{lemma}

\begin{proof}
	The proof of \eqref{eq:decompatom2}  follows the same steps of  \cite[Lemma 4.2]{HeYangWen21} and it will be omitted. Assume $r_{B} \leq 1$ and let $ \vp\in \mathcal{G}_0^\eta(\beta,\theta)=\mathcal{G}_0^\eta(x,r,\beta,\theta) $ such that $ \norm{\vp}_{\mathcal{G}(x,r,\beta,\theta)}\leq 1 $ for some $ r\in (0,1] $.
	From the size condition on $\varphi$ and
	Proposition \ref{prop:fv} \eqref{eq:fv4}, we may estimate
	\begin{align*}
		\abs{\ip{ a,\varphi }}=\abs{\int_X a(y)\vp(y)d\mu(y) }&\leq \int_X \abs{a(y)}\abs{\vp(y)}d\mu(y) \\
		& \leq  \int_X \abs{a(y)}\frac{1}{V_r(x)+V(x,y)} \corc{ \frac{r}{r+d(x,y)} }^\theta d\mu(y)\\
		& \leq C \mathcal{M}a(x). 
	\end{align*}
	Then, from the arbitrariness of $ \vp $, for $x \in B^{\ast}$ we have
	\begin{equation}\label{eq:661A}
		a_0^*(x)\lesssim  \mathcal{M}a(x). 
	\end{equation}
	Consider now $ {x\in X \backslash B^{*}}$, i.e., $ d(x,x_{B}) \geq (2A_0)r_B $. Note that if $ y\in B $ we obtain
	\begin{equation*}\label{eq:p1}
		d(x_B,y)< r_B\leq (2A_0)^{-1}d(x,x_B)\leq (2A_0)^{-1}(r+d(x,x_B)) .
	\end{equation*}
	So, using the regularity and size conditions of $\varphi$, we have that
	\begin{align*}
		\abs{\ip{ a,\varphi }}&=\abs{\int_{B} a(y)\vp(y)d\mu(y) } = \abs{ \int_{B} a(y)[\vp(y)-\vp(x_B)]d\mu(y) } + \abs{\vp(x_B)} \abs{\int ad\mu} \\
		&\leq \int_B \abs{a(y)}\abs{\vp(x_B)-\vp(y)}d\mu(y)   
		+ \abs{\int ad\mu} \, \frac{1}{V_r(x)+V(x,x_B)}\left[ \frac{r}{r+d(x,x_B)} \right]^\theta \\
		&\leq \int_B \abs{a(y)} \frac{1 }{V_r(x)+V(x,x_B)} \corc{ \frac{d(x_B,y)}{r+d(x,x_{B})} }^\beta \corc{ \frac{r}{r+d(x,x_B)} }^\theta d\mu(y) \\
		& \qquad \quad + \abs{\int ad\mu}  \,  \frac{1}{V(x,x_B)}\left[ \frac{1}{1 + d(x,x_B)} \right]^\theta 
		\\
		& {\leq} 
		\frac{1}{V(x,x_B)} \corc{\frac{r_B}{d(x,x_B)}}^\beta \int_B \abs{a(y)}d\mu(y) 
		+ \abs{\int ad\mu}   \frac{1}{V(x,x_B)}\left[ \frac{1}{1+  d(x,x_B)} \right]^\theta \\
		& {\leq} 
		\frac{\abs{B}^{1-1/p}}{V(x,x_B)} \corc{\frac{r_B}{d(x,x_B)}}^{\beta \land \theta } + \abs{\int ad\mu} \frac{1}{V(x,x_B)}\left[ \frac{1}{1+ d(x,x_B)} \right]^\theta 
	\end{align*}
	where in the third line we use the fact that $ r\leq 1 $ and in the last one follows since $\frac{r_B}{d(x,x_B)}\leq \frac{1}{2A_0}<1 $.	Then, from the arbitrariness of $ \vp $, we obtain
	\[ a_0^*(x)\lesssim \frac{\abs{B}^{1-1/p}}{V(x,x_B)} \corc{\frac{r_B}{d(x,x_B)}}^{\beta \land \theta } +  \abs{\int ad\mu}   \frac{1}{V(x,x_B)}\left[ \frac{1}{1+ d(x,x_B)} \right]^\theta 
	\] 
	when $ x\in X\setminus B^* $. So, from this last inequality and from \eqref{eq:661A} we obtain \eqref{eq:decompatom}.
\end{proof}

Next we present a definition of local Hardy space on $(X,d,\mu)$ in the sense of Coifman \& Weiss in terms of the grand local maximal function due to \cite[pp. 909]{HeYangWen21} denoted by $h^{\ast,p}(X)$ that will be simplified as $h^{p}(X)$ in this work.

\begin{definition}
	Let $ p\in (0,\infty] $. The local Hardy space $h^{p}(X)$ is defined as
	\[
	h^{p}(X)= \lla{ f\in ( \mathcal{G}^\eta_0(\beta,\theta) )^\ast \ : \ \norm{f}_{ h^{p}(X)}:= \norm{ f^*_0 }_{L^p} <\infty }.  
	\]
\end{definition}

The authors showed that $ h^{p}(X) $ is complete metric space for any $p\in (0,\infty]$, $ h^{p}(X)=L^p(X)$ if $p>1$ and that each local $(p,q)$-atom belongs to $ h^{p}(X)$ with uniformly bounded norm \cite[Lemma 4.2]{HeYangWen21}. Moreover, they proved that such spaces have an atomic decomposition theorem in terms of local $(p,\infty)$-atoms that we state below.

\begin{proposition}[Proposition 4.12 in \cite{HeYangWen21}]
	Let $ p\in (\frac{\gamma}{\gamma+\eta},1] $ and $ \beta, \theta \in (\gamma(1/p-1),\eta)$. For each $ f\in h^{p}(X) $, there exist a constant $C>0$,  a sequence of local $ (p,\infty) $-atoms $ \lla{a_j}_{j=1}^\infty $ and $ \lla{\lambda_j}_{j=1}^\infty\subset \C $ such that 
	$$ f=\sum_{j=1}^\infty \lambda_j a_j  \quad \text{in} \quad  (\mathcal{G}_0^\eta(\beta,\theta))^\ast$$ and 
	$ \sum_{j=1}^\infty \abs{\lambda_j}^p \leq C\norm{f_0^*}_{L^p}^p $ .
\end{proposition}

A direct consequence is that $h^{p}_{cw}(X)$ and $h^{p}(X)$ coincide in $(\mathcal{G}_0^\eta(\beta,\theta))^\ast$ with comparable norms (see \cite[Theorem 4.13]{HeYangWen21}).

The main contribution of this section is to show that, under natural restrictions, the spaces $ h^{p}_{\#}(X)$ and  $h^{p}(X) $ coincide with equivalent norms. We start showing that any $(p,q,1) $-approximate atom belongs to $ h^{p}(X)$ with uniform norm control.

\begin{proposition}\label{teo:670}
	Let $ p\in(\frac{\gamma}{\gamma+\eta},1) $, $ q\in (p,\infty]\cap [1,\infty] $ and a $ (p,q,1) $-approximate atom $ a $ supported on $ B=B(x_B,r_B)\subset X $. Then, there exists a constant $ C>0 $, which does not depend on $ a $ (but it can depend on $ p, A_0,A^\prime,\gamma,\theta,\beta $ ), such that
	\begin{align}\label{eq:unifboundatom}
		\norm{ a_0^* }_{L^p}\leq C , 
	\end{align}
with $ \beta,\theta\in (\gamma(\frac{1}{p}-1),\eta) $.
\end{proposition}

We point out each $(p,q,1)$-approximate atom is considered as a distribution in $ \mathcal{G}_0^\eta(\beta,\theta)$
by  Remark \ref{rem:hpcwdistinG}.

\begin{proof}
	Since the case $r_B> 1$ follows directly from \cite[Lemma 4.2]{HeYangWen21}, we will only consider $ r_B \leq 1$. Let $B^{\ast}=B(x_B, 2A_0r_B)$,
	\begin{align*}
		I& := \norm{a^*_0 \mathds{1}_{X\setminus B^{\ast}}}_{L^p}^p \qquad \text{and} \qquad 
		II :=\norm{a^*_0 \mathds{1}_{B^{\ast}}}_{L^p}^p.
	\end{align*}
	We begin with the estimate of $I$. Note that from \eqref{eq:decompatom} in Lemma \ref{lem:DA} we have
	\begin{align}
		I&\lesssim \int_{X\setminus B^{\ast}}\corc{ 
			\corc{ \frac{r_B}{d(x_B,x)} }^{\beta \land \theta} \frac{ |B|^{1-\frac1p}}{V(x_B,x)}}^p d\mu(x) 
		+
		\abs{\int ad\mu}^p\int_{X\setminus B^{\ast}}\corc{     \frac{(1+ d(x,x_B))^{-\theta}}{V(x,x_B)}  }^p d\mu(x) \nn \\
		& := I_1+I_2. \nn
	\end{align} 
	For $ I_1 $, note that for $C_j=B(x_B,2^{j+2}A_0r_B) \setminus B(x_B,2^{j+1}A_0r_B)$ 
	\[
	{\int_{C_j} V({x_B},x)^{-1}d\mu(x) \leq A^\prime,}
	\]
	and for $ j\in \N $, if $ x\in B(x_B,2^{j+2}A_0r_B)$ then
	\begin{align*}
		\frac{\abs{B(x_B,d(x_B,x))}^{1-p}}{ \abs{B(x_B,r_B)}^{1-p} } \leq \frac{\abs{B(x_B,2^{j+2}A_0r_B)}^{1-p}}{ \abs{B(x_B,r_B)}^{1-p} } \leq (A^\prime)^{1-p} (2^{j+2}A_{0})^{\gamma(1-p)}.
	\end{align*}
	Then, we obtain
	\begin{align*}
		I_1&\leq \int_{X \setminus B^{\ast}}  \corc{\frac{r_B}{d(x_B,x)}}^{(\beta\land\theta)p} \frac{\abs{B}^{p-1}}{V(x_B,x)^p}d\mu(x) \\
		&\leq  r_B^{(\beta\land \theta)p} \sum_{j=0}^{\infty} (2^{j+1}A_0r_B)^{-(\beta\land \theta)p} \int_{C_j} \frac{\abs{B(x_B,r_B)}^{p-1}}{\abs{B(x_B,d(x_B,x))}^{p}} d\mu(x) \\
		&\leq (A^\prime)^{1-p}{(A_0)^{-(\beta\land\theta)p+\gamma(1-p)}} \, \sum_{j=0}^{\infty}  2^{\gamma (j+2)(1-p)} 2^{-(j+1)(\beta\land \theta)p}  \int_{C_j} V(x_B,x)^{-1}d\mu(x)\\
		&\leq (A^\prime)^{2-p}{\left( 2^{{2}\gamma(1-p)} \,  (A_0)^{-(\beta\land\theta)p+\gamma(1-p)} \right)} 
		\sum_{j=0}^{\infty} (2^{(\beta\land \theta)p -\gamma(1-p)})^{-(j+1)}.
	\end{align*}
	Since $\beta,\theta > \gamma \left(\frac{1}{p}-1\right) $ we have $(\beta\land\theta)p-\gamma(1-p)>0$, and so the sum in the last inequality converges. Then  
	\begin{equation}\label{eq:20006}
		I_1\leq  C_{A_0,A^\prime,p,\beta,\theta,\gamma} .
	\end{equation}
	Now we estimate $ I_2 $. Note that if $ \abs{B} = \abs{B(x_B,1)}$ we obtain
	\begin{align}
			\int_{2A_0r_B \leq d(x,x_B)<2A_0} V(x,x_B)^{-p} d\mu(x)&\lesssim \abs{B(x_B,2A_0r_B)}^{-p}\abs{B(x_B,2A_0)} \nonumber \\
			& {\leq}  \abs{B(x_B,r_B)}^{-p}\abs{B(x_B,1)} \nn \\
			&= \abs{B(x_B,1)}^{1-p}{,} \label{eq:20001}
	\end{align}
	and if $ \abs{B}<\abs{B(x_B,1)} $, there will exist a non-negative integer $ m $ and positive numbers $ r_1,r_2,\cdots,r_m $
	such that $ r_B<r_1<\cdots <r_m<r_{m+1}:=1 $ and
	\begin{equation}\label{eq:unifboundpqfunc1}
		(A^\prime)^{j-1}\abs{B}<\abs{B(x_B,r_j)}\leq (A^{\prime})^{j}\abs{B},
	\end{equation}
	for all $ 0\leq j\leq m+1 $ (see Proposition \ref{prop:406} for the proof of this construction). Then by \eqref{eq:unifboundpqfunc1} we obtain
	\begin{align}
		\int_{2A_0r_B\leq d(x,x_B)<2A_0} &V(x,x_B)^{-p} d\mu(x)
		= \sum_{j=0}^{m}\int_{2A_0r_j\leq d(x,x_B)<2A_0r_{j+1}} V(x,x_B)^{-p} d\mu(x) \nn\\
		& \lesssim  \sum_{j=0}^{m}  \abs{ B(x_B,2A_0r_j) }^{-p} \abs{ B(x_B,2A_0r_{j+1}) } \nn\\
		& \lesssim  \sum_{j=0}^{m}  \abs{ B(x_B,r_j) }^{-p} \abs{ B(x_B,r_{j+1}) } 
		\leq  (A^\prime)^{1+p}\abs{B}^{1-p}\sum_{j=0}^m  (A^\prime)^{(1-p)j} \nn\\
		&\leq \frac{(A^\prime)^{1+p}}{(A^{\prime})^{1-p}-1}\abs{B}^{1-p} [(A^{\prime})^{(1-p)(m+1)}-1] \nn\\
		&\leq  \frac{(A^\prime)^2}{(A^{\prime})^{1-p}-1}[(A^{\prime})^{m}\abs{B}]^{1-p}  \nn \\
		&\leq \frac{(A^\prime)^2}{(A^{\prime})^{1-p}-1} \, \abs{B(x_B,1)}^{1-p}, \label{eq:20002}
	\end{align}
	where in the last inequality we have used \eqref{eq:unifboundpqfunc1} with $j=m+1$. So, from \eqref{eq:20001} and \eqref{eq:20002} we obtain
	\begin{equation}\label{eq:20003}
		\int_{2A_0r_B\leq d(x,x_B)<2A_0} V(x,x_B)^{-p} d\mu(x) \lesssim \abs{B(x_B,1)}^{1-p}.
	\end{equation}
	On the other hand
	\begin{align}
		\int_{2A_0\leq d(x,x_B)} \frac{d(x,x_B)^{-\theta p}}{V(x,x_B)^p}  d\mu(x) 
		&\lesssim  \sum_{j=1}^{\infty} (2^jA_0)^{-\theta p} \int _{2^{j}A_0\leq d(x,x_B)\leq 2^{j+1}A_0} \frac{1}{\abs{B(x_B,2^{j}A_0)}^p}  d\mu(x) \nn \\
		& \lesssim \sum_{j=1}^{\infty} 2^{-j\theta p} \frac{\abs{B(x_B,2^{j+1}A_0)}}{\abs{B(x_B,2^{j}A_0)}^p} 
		\lesssim \sum_{j=1}^{\infty} 2^{-j\theta p} \abs{B(x_B,2^{j}A_0)}^{1-p} \nn\\
		& \lesssim \sum_{j=1}^{\infty} 2^{j[\gamma(1-p)-\theta p]} \abs{B(x_B,A_0)}^{1-p} \nn\\
		& {\leq (A'A_0^\gamma)^{1-p}} \abs{B(x_B,1)}^{1-p}\sum_{j=1}^{\infty}  2^{j[\gamma(1-p)-\theta p]} \lesssim \abs{B(x_B,1)}^{1-p} \label{eq:20004}
	\end{align}
	where the constant in the last inequality depends only on $A_0,\gamma,A^\prime, p $ and $ \theta $ and the convergence of the sum follows since $ \gamma(\frac1p-1) <\theta$. Then, from \eqref{eq:20003}, \eqref{eq:20004} and approximate moment condition of $ a $ we obtain
	\begin{align}
		I_2&\leq
		\abs{\int ad\mu}^p \lla{ \int_{2A_0r_B \leq d(x,x_B)<2A_0}\frac{1}{V(x,x_B)^{p}}   d\mu(x) + \int_{2A_0\leq d(x,x_B)}\corc{ \frac{  d(x,x_B)^{-\theta}}{V(x,x_B)}  }^p d\mu(x)} \nn \\
		& \lesssim \abs{\int ad\mu}^p  \abs{B(x_B,1)}^{1-p}\lesssim 1 \label{eq:20005}
	\end{align}
	where the constant in the last inequality only depends on $A_0,\gamma,A^\prime, p $ and $ \theta $. So, from \eqref{eq:20006} and \eqref{eq:20005} we obtain
	\[
	I\leq C_{A_0,A^\prime,p,\beta,\theta,\gamma}.
	\]
	
	To estimate $ II $, we follow the same lines as in the proof of \cite[Lemma 4.2]{HeYangWen21} (Case 1 and Case 2) since the moment condition of $ a $ does not play any role in the argument.
\end{proof}

{Now we are ready to state the desired result.}

\begin{theorem}\label{prop:comparison}	
	Let $ p\in (\frac\gamma{\gamma+\eta}, 1) $, $ q\in [1,\infty] $ and $ \beta,\theta\in \, (\gamma(\frac1p-1),\eta) $. In regard $ h^{p}_{\#}(X) $ and $ h^{p}(X) $ as subspaces of $ (\mathcal{G}^\eta_0(\beta,\theta))^{\ast} $, then $h^{p}_{\#}(X)=h^{p}(X) $ with equivalent norms.
\end{theorem}

\begin{proof}
	In view of Proposition \ref{prop:502}, it will be sufficient to show the theorem for $ q=\infty $.
	We start showing that $ h^{p,\infty}_{\#}(X)\subset h^{p}(X) $, following the proof presented at \cite[Section 4.1]{HeHanLiLYY19}. 
	
	Let $ f\in h^{p,\infty}_{\#}(X) $. Then, there exist $\lla{\lambda_j}\in \ell^{p}(\C) $ and $ (p,\infty,1) $-approximate atoms $ a_j $ such that
$f=\sum_{j=1}^{\infty} \lambda_j a_j,$
in $ \ell^\ast_{1/p-1,1}(X)$. Since by \eqref{eq:684} we have $\mathcal{G}_{0}^{\eta}(\beta,\theta)\subset \ell_{1/p-1,1}(X)$, the restriction $ g:=f \,\vline_{ \,\mathcal{G}^\eta_0(\beta,\theta)} \in (\mathcal{G}^\eta_0(\beta,\theta))^\ast $ and $ g=\sum_{j=1}^{\infty} \lambda_j a_j $ in $ (\mathcal{G}^\eta_0(\beta,\theta))^\ast $. Then,
	\begin{align*}
		g_{0}^*(x)\leq \sum_{j=1}^{\infty} \abs{\lambda_j} (a_j)_{0}^*(x), \qquad \forall \, x \in X
	\end{align*}
	and from the Proposition \ref{teo:670}
	\begin{align*}
		\norm{g_{0}^*}_{L^p(X)}^p\leq \sum_{j=1}^{\infty} \abs{\lambda_j}^p \norm{(a_j)_{0}^*}_{L^p(X)}^p \lesssim \sum_{j=1}^{\infty} \abs{\lambda_j}^p.
	\end{align*}
	This shows that $ g\in h^{p}(X) $ and also by the arbitrariness of the decomposition
	\begin{equation*}
		\norm{g}_{h^{p}}\lesssim \norm{f}_{{p,\infty}}.
	\end{equation*}
	In this sense, $ h^{p,\infty}_{\#}(X)\subset h^{p}(X) $. 
	
	Now we deal with the inclusion $ h^{p}(X)\subset h^{p,\infty}_{\#}(X) $. Recall  that $h^{p}(X) = h^{p}_{cw}(X)$, where here $h^{p}_{cw}(X)$ is characterized by the atomic space defined in terms of local $(p,\infty)$-atoms with convergence in $(\mathcal{G}^{\eta}_0(\beta,\theta))^\ast$ (see \cite[Definition 4.1 and Proposition \olive{4.13}]{HeYangWen21}.) So, given $ f\in h^{p}_{cw}(X) \cap (\mathcal{G}^{\eta}_0(\beta,\theta))^\ast$, there exist $ \lla{\lambda_j}_{j} \in \ell^p(X) $ and a sequence $\left\{a_j \right\}$ of local $ (p,\infty)$-atoms  such that $ f=\sum_{j=1}^{\infty} \lambda_j  a_j $ in $ (\mathcal{G}^\eta_{0}(\beta,\theta))^\ast $. From \cite[Theorem 7.4 and Corollary 7.5]{HeYangWen21} we have that $ (h^{p}_{cw}(X))^\ast = \ell_{1/p-1,1}(X) $, and so the action $ \vp(f) $ is well defined  for any $ \vp\in \ell_{1/p-1,1}(X) $. Then we define	$ \Lambda_f : \ell_{1/p-1,1}(X)\longrightarrow \C  $ by $ \Lambda_f(\vp):=\vp(f) $.
	Since the sequence of partial sums of $ \sum_{j=1}^{\infty}\lambda_ja_j $ converges to $ f $ in $ h^{p}_{cw}(X)$-norm, 
	we have
	\begin{equation}\label{eq:690}
		\Lambda_f (\vp)= 
		\lim_{n\to \infty} \vp\pare{\sum_{j=1}^n \lambda_j a_j}
		= \lim_{n\to \infty} \sum_{j=1}^n \lambda_j \int a_j \vp \, d\mu
	\end{equation}
	for any $ \vp\in \ell_{1/p-1,1}(X) $ (see \cite[Theorem 7.4]{HeYangWen21}). Note that \eqref{eq:690} shows that $ \Lambda_f $ is independent of the decomposition $ \sum_{j=1}^\infty \lambda_j a_j $ and $ \Lambda_f=f $ on $ \mathcal{G}^\eta_{0}(\beta,\theta)$. 
	From \eqref{eq:690} and Proposition \ref{prop:414} we have 
	$ \Lambda_f\in h^{p,\infty}_{\#}(X) $, and 
	\[
	\norm{\Lambda_f}_{h^{p,\infty}_{\#}}^p\leq \sum_{j=1}^{\infty} \abs{\lambda_j}^{p}.
	\]
	From the arbitrariness of the decomposition 
	of $ f $ we have
	\[
	\norm{\Lambda_f}_{h^{p,\infty}_{\#}}\leq \norm{f}_{h^{p}_{cw}}
	\]
	for all $ f\in h^{p}_{cw}(X) \cap (\mathcal{G}^{\eta}_0(\beta,\theta))^\ast $, as desired. 
\end{proof}

\section{Continuity of inhomogeneous Calder\'on-Zygmund type operators}\label{sec:app}

In this section we discuss conditions for the boundedness of inhomogeneous Calder\'on-Zygmund operators of order $(\nu,s)$ on local Hardy spaces. For the sake of completeness, we write the precise definition of some elements already mentioned at the introduction. Let $ s\in(0,1] $ and denote by $ C(X) $ the space of continuous functions in $ X $. Recall that the space of \emph{$ s $-H\"older continuous functions} (homogeneous) on $ X $ is defined by
\[
C^s(X)=\lla{f\in C(X):\norm{f}_{L^\infty} +\sup_{x\neq y} \frac{\abs{f(x)-f(y)}}{d(x,y)^s}<\infty }.
\]
We denote by $ V_{s} := C_b^s(X) $ the space of functions in $ C^s(X) $ with bounded support. The set of continuous linear functionals on $ C^s_b(X) $ will be denoted by $ (C_b^s(X))^\ast $, and it will be equipped with the weak* topology. We refer to the elements in $ (C_b^s(X))^\ast $ as distributions. 

\begin{definition}\label{def:CZO1}
	A $\mu$-mensurable function $ K: (X\times X)\setminus\lla{ (x,x):x\in X } \to \C $ is called an inhomogeneous Calder\'on-Zygmund kernel of order $ (\nu,s) $ if it satisfies conditions \eqref{forteb} and \eqref{czstd1}. A linear and bounded operator $ R:V_s(X)\to V_s^\ast(X) $ is called an \textit{inhomogeneous Calder\'on-Zygmund operator of order $ (\nu,s) $} if it is associated to an inhomogeneous Calder\'on-Zygmund kernel of order $ (\nu,s) $ in the integral sense \eqref{czs} and is bounded in $L^2(X)$.
\end{definition}

In what follows we prove the well definition of \eqref{czs} and $R^\ast (1)$ when $R$ is an inhomogeneous Calder\'on-Zygmund operator of order $(\nu,s)$ for every $f\in L^{2}_{c}(X)$. 

\begin{proposition} \label{prop:well-def-T*1}
	Let $R$ be an inhomogeneous Calder\'on-Zygmund operator of order $(\nu,s)$. 
 Then $Rf \in L^1(X)$ for every $f\in L^2_{c}(X)$. Moreover $R^{*}(1) \in L^{2}_{loc}(X)$ in the following sense: there exists $ F\in L^2_{loc}(X) $ such that
$$\ip{ R^*(1),g }=\int F(x)g(x)d\mu(x),\qquad\qquad \forall \, g\in {L^2_c(X)}.$$
 \end{proposition}
Consequently $R^{*}(1)$ is well defined in the sense of distributions
\begin{equation}\label{T*1-definition}
\ip{R^*(1),f}:=
\int Rf(x)d\mu(x),\qquad \forall f \in L^2_{c}(X).
\end{equation}

\begin{proof}
	Let $f\in L^2_{c}(X)$ and suppose that $ \supp f\subset B $. Then
	\begin{align*}
		\int_{2A_0B} \abs{Rf(x)}d\mu(x)\leq \abs{2A_0B}^{\frac12}\norm{ Rf }_{L^2(X)}\leq \abs{2A_0B}^{\frac12}\norm{ R }_{L^2\to L^2} \norm{f}_{L^2}.
	\end{align*}
	To deal with the estimate in $(2A_0B)^\complement$, from $ (i') $ in the previous definition and Proposition \ref{prop:fv} \eqref{eq:fv2} we obtain
	\begin{align*}
		\int_{(2A_0B)^\complement} \abs{Rf(x)}d\mu(x)&\leq \int_{(2A_0B)^\complement} \int_B |K(x,y)|\abs{f(y)}d\mu(y) d\mu(x) \\
		& \leq \int_B \abs{f(y)} \int_{(2A_0B)^\complement}  \frac{1}{V(x,y) d(x,y)^\nu}d\mu(x)d\mu(y)\\
		& \leq \int_B \abs{f(y)} \int_{(B(y,r_B))^\complement}  \frac{1}{V(x,y) d(x,y)^\nu}d\mu(x)d\mu(y)\\
		& \lesssim r_B^{-{\nu}} \int_B \abs{f(y)} d\mu(y)\\
		& \leq r_B^{-{\nu}} \abs{B}^\frac12 \norm{f}_{L^2}.
	\end{align*}
Combining the previous estimates we have
	\begin{equation}\label{eq:R1funct}
		\norm{Rf}_{L^1} \lesssim  \abs{B}^{\frac12}\pare{\norm{ R }_{L^2}  + r_B^{-\nu}} \norm{f}_{L^2}.
	\end{equation}
Moreover, for each ball B there exists $ F\in L^2_{loc}(B) $ such that
$$\ip{ R^*(1),g }=\int F(x)g(x)d\mu(x),\qquad\qquad \forall \, g\in {L^2_c(B)}.$$
In fact,  the functional $ g\mapsto \ip{ R^*(1),\1_B g } $ defined on $ L^2(B) $ is bounded from \eqref{eq:R1funct}, and then by the Riesz representation theorem there exists $ F^B\in L^2(B) $ such that
		\begin{equation*}
			\ip{ R^*(1),\1_B g } = \int R(\1_{B}g)(x)d\mu(x)= \int_{B} F^B(x)g(x) d\mu(x)
		\end{equation*}
		for all $ g\in L^2(B) $ (and in particular for all $ g\in L^2_c(X) $ with $ \supp g\subset B $).
We point out that if $ B_1\subset B_2 $ and $ g\in L^2(B_1) $ then
		\begin{equation*}
			\int_{B_1} \1_{B_1} F^{B_2} g\, d\mu = \int_{B_2}  F^{B_2} \pare{\1_{B_1} g}\, d\mu =	\ip{R^*(1),\1_{B_2}\pare{\1_{B_1} g}}=\ip{ R^*(1),\1_{B_1} g } = \int_{B_1} F^{B_1} g\, d\mu.
		\end{equation*}
Thus $ \1_{B_1}\, F^{B_2}=F^{B_1} $ almost everywhere in $ B_1 $. Using a sequence of nested subsets $ B_1\subset B_2\subset \cdots \subset X $ that exhausts $ X $, we are able to define $ F\in L^2_{loc}(X) $ such that $ F\vline_{B_j}=F^{B_j} $. Moreover, for $ g\in L^2_c(X) $ with $ \supp g\subset B $ we obtain
		\[
		\ip{ R^*(1),g }=\ip{ R^*(1),\1_B g } = \int_B F^B(x)g(x) d\mu(x) = \int_B F(x) g(x) d\mu(x)= \int F(x) g(x) d\mu(x).
		\]
\end{proof}

In the next proposition, we show the expected property that if $R$ is an inhomogeneous Calder\'on-Zygmund operator of order $ (\nu,s) $, then it maps local atoms into approximate molecules. This strategy has been used extensively in the Euclidean setting to infer boundedness of Calder\'on-Zygmund type operators in Hardy spaces. However, in contrast to the setting of Hardy spaces in $\R^{n}$, this property is not sufficient to conclude the boundedness in $h^{p}_{\#}(X)$, as it will be discussed later. 

\begin{proposition}\label{prop:atomtomol3}
	Let $ R $ be an inhomogeneous Calder\'on-Zygmund operator of order $ (\nu,s) $ and 
	$ \frac{\gamma}{\gamma+\min\lla{\nu,s}} <p<1$. 
Suppose that there exists $C>0$ such that for any ball $B:=B(x_{B},r_{B})\subset X$ with $r_{B}<T$ we have that $f:=R^{\ast}(1)$ satisfies
		\begin{equation}\label{eq:atomtomol3}
		\bigg( \fint_B \abs{f-f_B}^2 d\mu \bigg)^{1/2} \leq C \abs{B(x_B,T)}^{1-\frac1p} \abs{B(x_{B},r_{B})}^{\frac1p-1} .
		\end{equation}
		If $a$ is a local $ (p,2)$-atom supported in $ B(x_B,r_B)$, then $ Ra $ is a multiple constant of a $ (p,2,T, \lambda)$-approximate molecule centered in $ B(x_B,2A_0r_B) $, for some $ \lambda $ satisfying \eqref{lambda-type-condition}.
\end{proposition}

\begin{proof} Let $a$ be a local $ (p,2)$-atom supported in $B=B(x_B,r_{B})$. We claim that $Ra$ is a multiple of a $ (p,2,T,\lambda)$-approximate molecule centered in $ B^{*}:=B(x_B,2A_{0}r_{B})$. In fact, from the $L^{2}$-boundedness of $R$ it follows
\begin{align}
			\norm{Ra \1_{B^{*}}}_{L^2}&\leq \norm{Ra}_{L^2} \leq \norm{R}_{L^2} \norm{a}_{L^2} 
			\leq {[A'(2A_{0})^\gamma]^{\left(\frac{1}{p}-\frac{1}{2} \right)}} \norm{R}_{L^2} \abs{B^{*}}^{\frac12-\frac1p}. \label{eq:atomtomol10}
		\end{align}
To verify the condition (ii) from Definition \ref{def:appmol}, let us assume $r_{B}<T$ and let $C_k = B(x_{B},A_{0}2^{k+1}r_{B}) \setminus B(x_{B},A_{0}2^{k}r_{B})$ for $k \geq 1$. Then, using the vanishing moments of $a$ we may write
\begin{align*}
\norm{Ra}_{L^2(C_k)} &= \corc{\int_{C_k} \abs{ \int_{B} \left[ K(x,y)-K(x,x_B)\right]a(y)d\mu(y) }^2 d\mu(x)}^\frac12\\
&\leq  \corc{\int_{B} |a(y)| \left[\int_{C_{k}}  |K(x,y)-K(x,x_B)|^{2}d\mu(x) \right]^{\frac{1}{2}}  d\mu(y)}
\end{align*}
For each $y \in B$ and $x \in C_{k}$ we have $d(x,x_{B})>(2^{k}A_{0})d(y,x_{B})$ and by the pointwise kernel estimate \eqref{czstd1} it follows from Proposition \ref{prop:fv}
\begin{align*}
\int_{C_{k}}  |K(x,y)-K(x,x_B)|^{2}d\mu(x)
&\leq  \int_{C_k} \pare{\frac{d(y,x_B)}{d(x,x_B)}}^{2s} \frac{1}{V(x,x_B)^2}d\mu(x) \\
&{\lesssim} (r_{B})^{2s}  \abs{ B(x_B,{2^{k}A_0r_{B}}) }^{-1} \int_{C_k}  \frac{d({x_B,x})^{-2s}}{V({x_B,x})}d\mu(x) \\
&\lesssim (r_{B})^{2s}  \abs{ B(x_B,{2^{k}A_0r_{B}}) }^{-1}  (2^{k}A_0r_{B})^{-2s} \\
&= {(A_0^{-1})^{2s}} \abs{ B(x_B,{2^{k}A_0r_{B}}) }^{-1}  (2^{k})^{-2s}. 
\end{align*}
Then, from the previous estimate 
\begin{align}
\norm{Ra}_{L^2(C_k)} &= \corc{\int_{C_k} \abs{ \int_{B} \left[ K(x,y)-K(x,x_B)\right]a(y)d\mu(y) }^2 d\mu(x)}^\frac12 \nonumber \\
&{\lesssim} \|a\|_{L^1} \abs{ B(x_B,{2^{k}A_0r_{B}}) }^{-\frac{1}{2}}  (2^{k})^{-s} \nonumber \\
&{\leq} |B(x_{B},r_{B})|^{1-\frac{1}{p}} \abs{ B(x_B,{2^{k}A_0r_{B}}) }^{-\frac{1}{2}}  (2^{k})^{-s} \nonumber \\
&{\leq [A'(2^kA_0)^\gamma ]^{\frac1p-1} \abs{B(x_B,2^kA_0r_B)}^{1-\frac1p} \abs{ B(x_B,{2^{k}A_0r_{B}}) }^{-\frac{1}{2}}  (2^{k})^{-s}} \nonumber \\
&{=[A'(2^kA_0)^\gamma ]^{\frac1p-1} \abs{B(x_B,2^kA_0r_B)}^{\frac12-\frac1p} (2^{k})^{-s}} \nonumber \\
&{\leq [A'(2^kA_0)^\gamma ]^{\frac1p-1} (A')^{\frac1p-\frac12} \abs{B(x_B,2^{k+1}A_0r_B)}^{\frac12-\frac1p} (2^{k})^{-s}} \nonumber \\
&{= [A'(A_0)^\gamma ]^{\frac1p-1} (A')^{\frac1p-\frac12} \abs{B(x_B,2^{k+1}A_0r_B)}^{\frac12-\frac1p} 2^{k(\gamma(\frac1p-1)-s)}} \label{eq:estimateRa1}
\end{align}

Now, we move on to the case where $ r_{B}\geq T $. First, note that if $ x\in C_k $ and $ y\in B $, then  
from quasi-triangle inequality we obtain
		\[
		2^{k}A_{0}r_{B}\leq d(x,x_B)\leq A_0(d(x,y)+d(y,x_B))< A_0 d(x,y)+A_{0}r_{B},
		\]
and then $ d(x,y)>(2^k-1)r_{B} $. Moreover, for every $z\in B(x_B,2^{k{+1}}A_{0}r_{B}) $, we have
		\begin{align*}
			d(z,y)\leq A_0(d(z,x_B)+d(x_B,y))< A_0^2r_{B}(2^{k+1}+A_0^{-1})&\leq
			\frac{2^{k+1}+1}{2^k-1}A_0^2 d(x,y)\\
			&=(2+\frac3{2^k-1}) A_0^2 d(x,y) \leq 5A_0^2 d(x,y).
		\end{align*}
		This means that $B(x_B, 2^{k+1}A_0r_B) \subset B(y,{5}A_0^2\,d(x,y))$ and hence
	\begin{align} \label{eq:estimate-volume-ball}
		|B(x_B, 2^{k+1}A_0r_B)| \leq |B(y,5A_0^2\,d(x,y))| \leq A'(5A_0^2)^{\gamma} |B(y,d(x,y))| = A'({5}A_0^2)^{\gamma} V(x,y). 
	\end{align}
	Going back to the estimate of (ii), since for this case the atoms do not necessarily satisfy vanishing moments, it follows by the inhomogeneous kernel condition \eqref{forteb}, Proposition \ref{prop:fv} and the estimate \eqref{eq:estimate-volume-ball}
		\begin{align}
			\norm{Ra}_{L^2(C_k)}
			&\leq \int_{B}\abs{a(y)}\corc{\int_{C_k} \abs{  K(x,y) }^2 d\mu(x)}^{\frac12}d\mu(y) \nonumber \\
			&\leq \int_{B}\abs{a(y)}\corc{\int_{C_k} \frac{d(x,y)^{-2\nu}}{V(x,y)^2} d\mu(x)}^{\frac12}d\mu(y) \nonumber \\
			&\lesssim  \frac{1}{\abs{B(x_B,2^{{k+1}}A_{0}{r_B})}^{\frac12}}\int_{B}\abs{a(y)}\corc{\int_{C_k} \frac{d(x,y)^{-2\nu}}{V(x,y)} d\mu(x)}^{\frac12}d\mu(y) \nonumber \\
			&{\leq} \frac{1}{\abs{B(x_B,2^{k+1}A_{0}r_{B})}^{\frac12}}\int_{B}\abs{a(y)}\corc{\int_{{(2^k-1)}r_{B}\leq d(x,y)} \frac{d(x,y)^{-2\nu}}{V(x,y)} d\mu(x)}^{\frac12}d\mu(y) \nonumber \\
			&\lesssim \frac{1}{\abs{B(x_B,2^{k+1}A_{0}r_{B})}^{\frac12}}\corc{{(2^k-1)r_{B}}}^{-\nu}\int_{B}\abs{a(y)}d\mu(y) \nonumber \\
			&{\leq} \frac{1}{\abs{B(x_B,2^{k+1}A_{0}r_{B})}^{\frac12}} \abs{B(x_B,r_{B})}^{1-\frac1p} (2^{k-1}T)^{-\nu} \nonumber \\
			& \leq 2^{\nu } T^{-\nu}[A'(2A_{0})^\gamma]^{\frac{1}{p}-1}  \, |B(x_B,2^{k+1}A_{0}r_{B})|^{\frac12-\frac1p} \, 2^{k \left(\gamma(\frac1p-1)-\nu\right)}. \label{eq:estimateRa2}
		\end{align}
Then, from estimates \eqref{eq:estimateRa1} and \eqref{eq:estimateRa2} we conclude that $\lambda_k=2^{k\left[\gamma(\frac1p-1)-\min\{\nu,s\}\right]}$ and it clearly satisfies \eqref{lambda-type-condition}.

In order to conclude  the proof, it remains to provide the estimate of the moment condition of $Ra$. Since it suffices to show it when $r_B<T$, by the vanishing condition on $a$
and from \eqref{eq:atomtomol3} we have
	\begin{align*}
		\abs{\int Ra(x)d\mu(x)}= 
		\abs{ \ip{ R^*(1)-(R^*1)_B,a } }
		&\leq \norm{a}_{L^2} \norm{  R^*(1)-m_{B,T}(R^*1) }_{L^2(B)}  \\
		& \leq \abs{B}^{\frac12-\frac1p} \norm{  R^*(1)-m_{B,T}(R^*1) }_{L^2(B)} \\
		& \leq C \abs{B(x_{B},T)}^{1-\frac1p}. 
	\end{align*}
\end{proof}

A natural question arises on how to guarantee a bounded extension of $R$ from $ h^{p}_{cw}(X)$ to $ h^{p}_{\#}(X)$ from Proposition \ref{prop:atomtomol3}. In fact, given $f\in h^{p}_{cw}(X)$ decomposed as $f=\sum_{j} \lambda_j a_j$ for local $(p,2)$-atoms and let us suppose that
\begin{equation}\label{rstrong}
Rf=\sum_{j} \lambda_j Ra_j \quad \quad \text{ in }  \ell^{\ast}_{\frac1p-1,T}(X).
\end{equation}
Then since $Ra_j$ is a multiple of $(p,2,T, \lambda)$-approximate molecule, with constant independent of $a_j$, using Corollary \ref{corollary:molecular-decomp} we can show that 
$$
\| Rf \|_{p,2} \lesssim \bigg( \sum_{j} |\lambda_j|^p \bigg)^{1/p} \approx \| f\|_{h^{p}_{cw}}.
$$
In the next theorem,  we replace \eqref{rstrong} assuming $ \norm{\cdot}_{h^{p,2}_{fin}}\approx \norm{\cdot}_{h^{p}_{cw}} $ in $ h^{p,2}_{fin}(X) $, i.e. the  norms in $h^{p,2}_{fin}(X)$ and $h^{p}_{cw}(X)$ are equivalents.

\begin{theorem}\label{teo:boundextenCZO} 
	Let $ R $ be an inhomogeneous Calder\'on-Zygmund operator of order $ (\nu,s) $ and $ \frac{\gamma}{\gamma+\min\lla{\nu,s}}< p<1 $. Suppose that there exists $C>0$ such that for any ball $B:=B(x_{B},r_{B})\subset X$ with $r_{B}<T$ we have that $f:=R^{\ast}(1)$ satisfies
		\begin{equation*}	
		\bigg( \fint_B \abs{f-f_B}^2 d\mu \bigg)^{1/2} \leq C \abs{B(x_B,T)}^{1-\frac1p} \abs{B(x_{B},r_{B})}^{\frac1p-1}. 
		\end{equation*}
	If  $ \norm{\cdot}_{h^{p,2}_{fin}}\approx \norm{\cdot}_{h^{p}_{cw}} $ in $ h^{p,2}_{fin}(X) $,
	then the operator $ R $ can be extended as a linear bounded operator from $ h^{p}_{cw}(X)$ to $ h^{p}_{\#}(X)$.
\end{theorem}

\begin{proof}	
	Since $ R $ is a bounded linear operator on $ L^2(X) $, it is a well defined linear operator on $ h^{p,2}_{fin}(X) $. Then, 
	given $ f\in h^{p,2}_{fin}(X) $ with $ f=\sum_{j=1}^m \lambda_j a_j $, from Propositions \ref{prop:atomtomol3} and \ref{prop:unifboundmol} we have that 
	\begin{align}\label{eq:7520}
		\norm{Rf}_{{h^{p,2}_{\#}(X)}}&\leq \sum_{j=1}^m \abs{\lambda_j}\norm{Ra_j}_{{h^{p,2}_{\#}(X)}} \lesssim \pare{\sum_{j=1}^{m}  \abs{\lambda_j}^p}^{1/p}, 
	\end{align}
	where the implicit constant does not depend on $ f $. From the arbitrariness of the decomposition for $ f $
	and since $ \norm{\cdot}_{h^{p,2}_{fin}(X)} \approx \norm{\cdot}_{h^{p}_{cw}(X)} 
	$ on $ h^{p,2}_{fin}(X) $ we have
	\begin{equation}\label{eq:752}
		\norm{Rf}_{h^{p,2}_{\#}(X)}\lesssim \norm{f}_{h^{p}_{cw}(X)}, \quad \quad \forall  f\in h^{p,2}_{fin}(X).
	\end{equation}
	
	On the other hand, 
	given $ f\in h^{p}_{cw}(X) $ with
	$ f=\sum_{j=1}^{\infty} \lambda_j a_j $ in $ \ell^{\ast}_{1/p-1,T}(X)$ where $\left\{a_{j}\right\}_{j}$ are local $(p,2)$-atoms, it follows by \eqref{eq:7520} that the sequence of partial sums $\displaystyle{\left\{  \sum_{j=1}^{m}\lambda_j Ra_j \right\}_{m\in\N}}$ is a Cauchy sequence in $ h^{p,2}_{\#}(X) $, and hence it converges in $ h^{p,2}_{\#}(X) $. Thus, we can extend the operator $ R $ on $ h^{p}_{cw}(X) $ as
	\begin{equation}\label{eq:758}
		\tilde{R}(f):=\lim_{m\to \infty} \sum_{j=1}^m \lambda_j R(a_j), \quad \quad  \text{in} \,\, h^{p,2}_{\#}(X).
	\end{equation}
	Note that \eqref{eq:752} gives us the well definition of the extension $ \tilde{R} $. In fact, let $\displaystyle{f=\sum_{j=1}^{\infty} \lambda_j a_j=\sum_{j=1}^{\infty} \tilde\lambda_j \tilde{a}_j }$  in $ \ell^{\ast}_{1/p-1,T}(X)$, then
	
	\begin{align*}
		\Norm{\tilde{R}(f)- \sum_{j=1}^m \tilde{\lambda}_j R(\tilde{a}_j)}_{p,2}&\leq \Norm{ \tilde{R}(f)- \sum_{j=1}^n {\lambda}_jR{a}_j }_{p,2} + \Norm{ R\pare{\sum_{j=1}^n {\lambda}_j{a}_j} -  R\pare{\sum_{j=1}^m \tilde{\lambda}_j\tilde{a}_j}}_{p,2}\\
		&\lesssim \Norm{ \tilde{R}(f)- \sum_{j=1}^n {\lambda}_jR{a}_j }_{p,2} + \Norm{ \sum_{j=1}^n {\lambda}_j{a}_j -  \sum_{j=1}^m \tilde{\lambda}_j\tilde{a}_j}_{h^{p}_{cw}}\\
		&{\leq \Norm{ \tilde{R}(f)- \sum_{j=1}^n {\lambda}_jR{a}_j }_{p,2} + \Norm{ \sum_{j=1}^n {\lambda}_j{a}_j -f}_{h^{p}_{cw}}} + {\Norm{ f -  \sum_{j=1}^m \tilde{\lambda}_j\tilde{a}_j}_{h^{p}_{cw}},}
	\end{align*}
	for any $ m , n\in \N $, this shows the well definition of $ \tilde{R} $.
	Also, from
	\eqref{eq:752} we obtain
	\[
	\norm{\tilde{R}(f)}_{p,2} \lesssim \norm{f}_{h^{p}_{cw}}, \quad \quad \forall \, f\in h^{p}_{cw}(X). 
	\]
\end{proof}

We emphasize that condition $ \norm{\cdot}_{h^{p,2}_{fin}}\approx \norm{\cdot}_{h^{p}_{cw}} $ in Theorem \ref{teo:boundextenCZO} is used to show the boundedness and well definition of the extension of $ R $ in $ h^{p}_{cw}(X)$. This equivalence between norms was a condition used to extend bounded linear operators on local Hardy spaces in \cite[Proposition 7.1 and Theorem 7.4]{HeYangWen21}. In the latter work, the existence of a maximal characterization associated to the atomic decomposition in terms of local $(p,q)$-atoms plays a fundamental role.

\subsection{On local Hardy spaces $h^{p}(X)$}\label{sec:5.1}

In this section, we present the proof of Theorem \ref{teo1.2} as a direct consequence of relation between $h^p_{\#}(X)$ and $h^p(X)$ given by Theorem \ref{prop:comparison} and the results of previous section. We point out that in this section $R$ will denote an inhomogeneous Calder\'on-Zygmund operator of order $(\nu,s)$ for $ s\in (0,\eta] $, where $\eta$ is the same index of regularity considered in Section \ref{sec:rel}.

Next we restate the Theorem \ref{teo1.2} adding precisely details on the considered parameters in $h^p(X)$.

\begin{theorem}\label{cor:last}
	Let $ R $ be an inhomogeneous Calder\'on-Zygmund operator of order $ (\nu,s) $, $ \frac{\gamma}{\gamma+\min\lla{{\nu,s}}}< p<1 $, and $ \beta,\theta\in \, (\gamma(\frac1p-1),\eta) $ 
	. In regard $ h^{p}(X) $ as a subspace of $ (\mathcal{G}^\eta_0(\beta,\theta))^{\ast} $, if there exists $C>0$ such that for any ball $B(x_{B},r_{B})\subset X$ with $r_{B}<1$ we have that $f:=R^{\ast}(1)$ satisfies
	\begin{equation}\label{1.8b}
		\left(\fint_{B}|f-f_{B}|^{2} d\mu\right)^{1/2} \leq C \abs{B(x_{B},1)}^{1-\frac{1}{p}} \abs{B(x_{B},r_{B})}^{\frac{1}{p}-1},
	\end{equation}
	then the operator $ R $ defines a linear bounded operator on $ h^{p}(X)$.
\end{theorem}

\begin{proof}
	Since $ h^{p}(X)=h^{p,2}_{cw}(X) $ (\cite[Theorem 4.13]{HeYangWen21}) with equivalent norms, 
	$ \norm{\cdot}_{h^{p,2}_{fin}}\approx \norm{\cdot}_{h^{p,2}_{cw}} $ in $ h^{p,2}_{fin}(X) $ by item (i) at \cite[Proposition 7.1]{HeYangWen21} 
	and $h^{p,2}_{\#}(X)=h^{p}(X) $ with equivalent norms as consequence of Proposition \ref{prop:comparison}, the proof follows the same argument as presented in the proof of Theorem \ref{teo:boundextenCZO}.
\end{proof}

Analogous to Theorem \ref{teo:boundextenCZO1} we state the following result:
	\begin{theorem}\label{teo:last1}
	Let $ R $ be an inhomogeneous Calder\'on-Zygmund operator of order $ (\nu,s) $, $ \frac{\gamma}{\gamma+\min\lla{{\nu,s}}}< p\leq 1 $, and $ \beta,\theta\in (\gamma(\frac1p-1),\eta)$. In regard $ h^{p}(X) $ as a subspace of $ (\mathcal{G}^\eta_0(\beta,\theta))^{\ast} $, then the operator $ R $ defines a linear bounded operator from $ h^{p}(X)$ to $ L^p(X) $.
\end{theorem}

The proof \textit{is bis idem} the proof of Theorem \ref{teo:boundextenCZO} using the Proposition \ref{prop:Lrnormatom}. Note that $p=1$ is included in the statement of theorem, since $h^{1}(X)=h^{1,2}_{cw}(X)$ by \cite[Theorem 4.13]{HeYangWen21}
 with equivalent norms, 
	$ \norm{\cdot}_{h^{1,2}_{fin}}\approx \norm{\cdot}_{h^{1,2}_{cw}} $ in $ h^{1,2}_{fin}(X)$ by \cite[Proposition 7.1]{HeYangWen21}) and the conclusion follows by Proposition \ref{prop:Lrnormatom}.

\subsection{On Lebesgue spaces $ L^p(X)$.}\label{sec:5.2}
	
In this section we use the previous calculations to obtain the boundedness of inhomogeneous Calder\'on-Zygmund operator  from $ h^p_{cw}(X) $ to $ L^p(X)$, where any assumption on $R^{\ast}(1)$ is required.
	{The next result is a consequence of the proof of Proposition \ref{prop:atomtomol3}.
		\begin{proposition}\label{prop:Lrnormatom}
			Let $ R $ be  an inhomogeneous Calder\'on-Zygmund operator of order $ (\nu,s)$,  
			$ \frac{\gamma}{\gamma+\min\lla{\nu,s}} <p \leq 1$, and  
			$a$ be a local $(p,2)$-atom. Then  
 there exists a constant $ C>0 $, which does not depend on $ a $ (but it can depend on $T,A_0,A',\gamma,p,\nu, s$), such that
			$$ \norm{Ra}_{L^p}\leq C. $$
		\end{proposition}}
		\begin{proof}
			Let $a$ be a local $ (p,2)$-atom supported in $B=B(x_B,r_{B})$, $ B^{*}:=B(x_B,2A_{0}r_{B})$ and $C_k = B(x_{B},A_{0}2^{k+1}r_{B}) \setminus B(x_{B},A_{0}2^{k}r_{B})$ for $k \geq 1$. Then, using Holder's inequality we have	
			\begin{align*}
				\int \abs{Ra}^p d\mu &= \int_{B^*} \abs{Ra}^p d\mu +\sum_{k=1}^\infty \int_{C_k} \abs{Ra}^p d\mu\\
				& \leq \norm{Ra}_{L^2}^p \abs{B^*}^{1-\frac{p}2}+ \sum_{k=1}^\infty \abs{C_k}^{1-\frac{p}2} \norm{Ra}_{L^2(C_k)}^p.
			\end{align*}
			From \eqref{eq:atomtomol10}, \eqref{eq:estimateRa1} case $ r_B<T $, or \eqref{eq:estimateRa2} case $ r_B\geq T $, we obtain 
			\begin{align*}
				\int \abs{Ra}^p d\mu &\leq {[A'(2A_{0})^\gamma]^{\left(1-\frac{p}{2} \right)}} \norm{R}_{L^2}^p  \\
				&\quad + C^p\sum_{k=1}^\infty \abs{B(x_B,2^{k+1}A_0r_B)}^{\frac{p}2-1} 2^{k(\gamma(\frac1p-1)-\min\{\nu,s\})p} \abs{C_k}^{1-\frac{p}2} \\
				&\leq {[A'(2A_{0})^\gamma]^{\left(1-\frac{p}{2} \right)}} \norm{R}_{L^2}^p + C^p  \sum_{k=1}^\infty 2^{k(\gamma(\frac1p-1)-\min\{\nu,s\})p} \\
				& \lesssim 1,
			\end{align*}
			where $ C=\max\lla{ [A'(A_0)^\gamma ]^{\frac1p-1} (A')^{\frac1p-\frac12}, 2^{\nu } T^{-\nu}[A'(2A_{0})^\gamma]^{\frac{1}{p}-1}} $.
		\end{proof}

As a consequence, we obtain the following boundedness result.

	\begin{theorem}\label{teo:boundextenCZO1}
		Let $ R $ be an inhomogeneous Calder\'on-Zygmund operator of order $ (\nu,s) $ and $ \frac{\gamma}{\gamma+\min\lla{\nu,s}}< p<1 $. 
		If  $ \norm{\cdot}_{h^{p,2}_{fin}}\approx \norm{\cdot}_{h^{p,2}_{cw}} $ in $ h^{p,2}_{fin}(X) $,
		then the operator $ R $ can be extended as a linear bounded operator from $ h^{p}_{cw}(X)$ to $ L^p(X)$.
	\end{theorem}
	\begin{proof}
		Using the Proposition \ref{prop:Lrnormatom}, the proof follows the same lines in the proof of Theorem \ref{teo:boundextenCZO}. 
	\end{proof}

\section{The case $ p=1 $} \label{sec:p=1}

In this section, we present a version of Theorem \ref{cor:last} for $ p=1 $. First, note that the definition of $h^{p}_{\#}(X)$ does not cover the case $p=1$ since the convergence of the atomic series is not defined. We start by considering atoms with appropriate cancellation condition.

\begin{definition} \label{def:1q-function}
	Let $1<q \leq \infty$. We say that a $ \mu $-measurable function $ a $ is a \textit{$(1,q,T)$-approximate atom} if it satisfies the usual support and size condition of Definition \ref{def:pq-function} with $p=1$ and
		\begin{equation}\label{momentop1}
	\abs{ \int a \, d\mu }\leq \frac{2}{\log(2+T/r_B)}.\\
\end{equation}
\end{definition}
Atoms satisfying these approximate moment conditions were considered in \cite{DafniMomYue16}. In the same way as the case $ p<1 $, condition \eqref{momentop1} is just a local requirement when $ r_B<T$, since from the support and size conditions we have for $ r_B\geq T $
$$ \abs{ \int a \, d\mu }\leq  \|a\|_{L^{q}} |B(x_{B},r_{B})|^{\frac{1}{q'}}\leq 1 \leq \frac{2}{\log(2+T/r_B)}.$$
Also, with the same proof presented in Remark \ref{rem:approx_atoms} item (ii), we can show that each $(1,q,T)$-approximate atom is a multiple constant of a $(1,q,T^\prime)$-approximate atom for any $ T,T^\prime >0$.

The moment condition \eqref{momentop1} is more restricted than \eqref{momento} when $ p=1 $ with $r_{B}<T$. For more details on this condition we refer  \cite{DafniPicon22, DafniPicon23, DafniMomYue16}.

Now, the convergence of atomic series will be in the dual of the local $bmo(X)$. We recall that $bmo(X)$ is defined as the space of functions $ f $ in $ L^1_{loc}(X)$ such that
\[
\norm{f}_{bmo}:= \|f^{\ast}\|_{L^{\infty}}<\infty,
\] 
where $\displaystyle{f^\ast(x):= \sup_{B \ni x } \mathfrak{M}_{0,1,T}^B(f)}$. Clearly $bmo(X)=c_{0,q,T}(X)$ for any $1\leq q <\infty$, as a consequence of Lemma 6.1 in \cite{DafniMomYue16}, and $(bmo(X), \norm{\cdot}_{bmo}) $ is a normed space where each $(1,q,T)$-approximate atom defines a continuous linear functional with dual $ bmo^\ast(X)$-norm uniform (see \cite[Remarks 7.4]{DafniMomYue16}). This allows us to establish an analogous result to Proposition \ref{prop:414} for $p=1$ and $1< q \leq \infty$ defining $h^{1,q}_\#(X)$ as the set of elements $ g\in bmo^*(X) $ for which there exist a sequence $ \lla{a_j}_j $ of $ (1,q,T) $-approximate atoms and a sequence $\lla{\lambda_j}_j\in \ell^1(\C)$ such that 
\begin{equation}\label{eq:1qatom}
	g=\sum_{j=0}^{\infty} \lambda_j a_j, \quad \text{in } bmo^{\ast}(X),
\end{equation}
with quasi-norm
\[
\norm{g}_{1,q}:=\inf \left\{ \sum_j \abs{\lambda_j}  \right\},
\]
where the infimum is taken over all such atomic representations \eqref{eq:1qatom} of $g$. As before, $ d_{1,q}\pare{g,h}:=\|g-h\|_{1,q}$ defines a metric in $h^{1,q}_{\#}(X)$ making the space complete. 

Adapting the proof of Proposition \ref{prop:502}, we have $h^{1,q}_\#(X)= h^{1,\infty}_\#(X)$ for \textcolor{olive}{$ q\in (1,\infty) $} with equivalent norms, assuming $ \mu $ as a Borel regular measure. We denote by $ h^{1,q}_{fin,\#}(X) $ the subspace of $bmo^{\ast}(X)$ consisting of all finite linear combinations of $(1,q,T)$-approximate atoms, which is dense in $ (h^{1,q}_{\#}(X),d_{1,q}) $.

In the next definition, we consider the molecular structure of $h^1_\#(X)$, as an extension of Definition \ref{def:appmol} for the case $p=1$.

\begin{definition}\label{def:appmolp1}
	Let $1< q \leq \infty$ and $\lambda := \{\lambda_k\}_{k\in \N} \subset [0,\infty)$ satisfying
	\begin{equation} \label{lambda-type-conditionp1}
		\left\| \lambda \right\|_{1}:= \sum_{k=1}^{\infty} k\lambda_k < \infty.
	\end{equation}
	A measurable function $ M $ in $ X $	 
	is called a $(1,q,T,\lambda)$-approximate molecule if there exists a ball $B=B(x_B,r_{B})\subset X$ such that the size conditions (i) and (ii) in Definition \ref{def:appmol} with $p=1$ are satisfied and moreover the following cancellation condition holds
	\begin{equation} \label{moments-molecule-1}
	\displaystyle 	\abs{ \int M \, d\mu }\leq \frac{2}{\log(2+T/r_B)}.
	\end{equation}
\end{definition}

Again, up to a multiplication by a constant, the moment condition \eqref{moments-molecule-1} for molecules is also local since when $ r_B\geq T $ we have $\displaystyle{1\leq \frac{2}{\log(2+T/r_B)}}$ and then
\begin{align*}
	\abs{ \int M \, d\mu } \leq |B|^{1-\frac1q} \| M \, \1_{B} \|_{L^q} + \sum_{k=1}^{\infty} \lambda_k \, |A_k|^{1-\frac1q} \| M \, \1_{A_k}  \|_{L^q} 
	&\leq \big(1+ \sum_{k=1}^{\infty} \lambda_k \big)\frac{2}{\log(2+T/r_B)}.
\end{align*}
Moreover,  each $ (1,q,T,\lambda) $-approximate molecule $M$ centered in $B$ defines a distribution on $ bmo(X) $. In effect, from Corollary 3.3 in \cite{DafniMomYue16}, the same argument employed to prove \eqref{distri} shows that
	\begin{equation}\label{distrip1}
	\|M\|_{bmo^\ast(X)} \leq C(A')^{j_0(1-\frac1q)}(1+\sum_{j=1}^{\infty}\lambda_{j}),
\end{equation}
for some  $ j_0\in \N\cup \{0\} $ such that $ 2^{j_0}r_B\geq T $. 

Next, we state the molecular decomposition of $h^1_\#(X)$. Since its proof makes use of the same idea of the proof of Proposition \ref{prop:unifboundmol}, just taking $p=1$ and $1<q\leq \infty$ with the appropriate moment condition \eqref{momentop1}, we omit the details.

\begin{proposition}\label{prop:unifboundmolp1} 
	Let $1< q\leq \infty $ and $ M $ be	a $(1,q,T,\lambda)$-approximate molecule. Then there exist a sequence $\{ \beta_{j} \}_{j} \in \ell^{1}(\C) $ and $\{ a_{j} \}_{j}  $  of $ (1,q,T)$-approximate atoms such that
	\begin{equation}\label{eq:mol5p1}
		M=\sum_{j=0}^\infty \beta_j a_j, \quad \quad \text{in} \,\,\, L^{q}(X) 
	\end{equation}
	with $ \displaystyle{\sum_j\abs{\beta_j}\leq C_{A}\| \lambda \|_{{1}}} $. Moreover,  the convergence of \eqref{eq:mol5p1} is in $ bmo^\ast(X) $ and $\| M \|_{{1,q} }\leq C_{A,A^\prime} (1+\| \lambda \|_{{1}})$.
\end{proposition}

In the same way, we state a version of Proposition \ref{prop:atomtomol3} for $p=1$.

\begin{proposition}\label{prop:atomtomol3p1}
Let $ R $ be an inhomogeneous Calder\'on-Zygmund operator of order $ (\nu,s) $. Suppose  there exists $C>0$ such that for any ball $B:=B(x_{B},r_{B})\subset X$ with $r_{B}<T$ we have that $f:=R^{\ast}(1)$ satisfies 
		\begin{equation}\label{eq:atomtomol3p1}
			\bigg( \fint_B \abs{f-f_B}^2 d\mu \bigg)^{1/2} \leq C \frac{2}{\log(2+{T}/{r_B})}. 
		\end{equation}
If $a$ is a local $ (1,2)$-atom supported in $ B(x_B,r_B)$, then $ Ra $ is a multiple constant of a $ (1,2,T, \lambda)$-approximate molecule centered in $ B(x_B,2A_0r_B) $, for 
some $ \lambda $ satisfying \eqref{lambda-type-conditionp1}.
\end{proposition}

We emphasize the sequence $\lambda=\left\{ \lambda_{k} \right\}_{k} $ announced at last result is exactly the same found in the proof of Proposition \ref{prop:atomtomol3} taking  $p=1$, namely $\lambda_k:=2^{-k\min\{\nu,s\}} $.

Let $h^{1}_{cw}(X)$ be the set of distributions $f\in bmo^{\ast}(X)$ such that 
$$\displaystyle{ f=\sum_{j=1}^\infty \lambda_j a_j}, \quad  \text{in} \,\,  bmo^{\ast}(X),$$
for some  $\lla{\lambda_j}_{j} \in \ell^1(\C)$ and  $\{a_j\}_{j}$  local $(1,q,T)$-atoms, equipped with the norm 
$\norm{ f }_{h^{1}_{cw}}:= \inf \pare{\sum_{j=1}^{\infty}  \abs{\lambda_j}},$
where the infimum is taken over all such decompositions. Analogously as $h^{1}_{\#}(X)$, the space $h^{1}_{cw}(X)$  does not depend on $1< q \leq \infty$, 
assuming $ \mu $ is Borel regular. We use the notation $ h^{1,q}_{cw}(X)$ to emphasize the type of local $ (1,q,T)$-atoms considered.    
In the similar way, we denote by $ h^{1,q}_{fin}(X) $ the set of finite linear combinations of local $ (1,q,T)$-atoms.

Now, we make a comparison between the spaces $h^{1}_{\#}(X)$ and $h^{1}_{cw}(X)$ with the local Hardy space considered in \cite{DafniMomYue16}, that we denote by $ h^1_g(X) $ in this work. Mac\'ias and Segovia in \cite{MaSe79} and \cite{MaSe79b} showed  the existence of a quasi-metric $\rho$ equivalent to $d$ (i.e. $c_{1},c_{2}>0$ such that $c_{1}\rho(x,y)\leq d(x,y)\leq c_{2}\rho(x,y)$ for all $x,y \in X$) satisfying the following property:  there exist $ \alpha\in (0,1) $ and a constant $ C_d>0 $ such that for all $ x\in X $ and $ r>0 $
\begin{equation}\label{eq:holqm}
\abs{\rho(y,x)-\rho(z,x)} \leq C_d r^{1-\alpha}\rho(y,z)^\alpha
\end{equation}
whenever $ y,z\in B_{\rho}(x,r) $. The advantage is that $\mu(B_{d}(x,t)) \approx \mu(B_{\rho}(x,t)) $ and now balls are open. 
From now on, we consider $(X,d,\mu)$ a homogeneous type space where $d$ satisfies the condition \eqref{eq:holqm}.

We say that a function $ f\in L^1_{loc}(X) $ belongs to $ h^1_g(X) $ when
$
\norm{f}_{h^1_g}:=\norm{\M_\F f}_{L^1}<\infty,
$
where
\[
\M_\F f(x):= \sup_{\psi\in \F_x} \abs{ \int f\psi d\mu  }
\]
and $\F_x $ means the set of $\alpha $-H\"{o}lder continuous functions $\psi$   supported in a ball $ B(x,t) $, $ 0<t< 4A_0^2\,T$ satisfying 
\begin{equation}\label{eq:224}
\norm{\psi}_\infty \leq \frac{C_\F}{\abs{B(x,t)}} \qquad \text{and} \qquad \norm{\psi}_{\mathcal{L}^\alpha} \leq \frac{C_\F}{t^\alpha\abs{B(x,t)}}
\end{equation}
for some positive constant $ C_\F $. Here $ \alpha $ is the same constant appearing in \eqref{eq:holqm}. The space $h^{1}_{g}(X)$ is complete and continuously embedded in $L^{1}(X)$.  

In \cite{DafniMomYue16}, the authors proved an atomic decomposition, namely if $f\in h^{1}_{g}(X)$ then there exist a sequence of local $(1,\infty,T)$-atoms $\left\{a_{j}\right\}_{j}$ and a sequence of coefficients $\left\{\lambda_{j}\right\}_{j}$ in $\ell^{1}(\C)$ such that 
$$f=\sum_{j=1}^{\infty} \lambda_{j}a_{j}, \quad \text{with} \quad \sum_{j=1}^{\infty}|\lambda_{j}|\leq C\|f\|_{h^{1}_{g}}, $$    
where $C>0$ is independent of $f$. Conversely, if  $\left\{\lambda_{j}\right\}_{j}$ is a sequence in $\ell^{1}(\C)$ and  $\left\{a_{j}\right\}_{j}$ are local $(1,\infty,T)$-atoms (or approximate atoms) then $\sum_{j=1}^{\infty} \lambda_{j}a_{j}$ converges in $h^{1}_{g}(X)$ and 
$\|\sum_{j=1}^{\infty} \lambda_{j}a_{j}\|_{h^{1}_{g}}\leq \sum_{j=1}^{\infty} |\lambda_{j}| $.
 They also proved that $bmo(X)$ can be identified with dual of $h^{1}_{g}(X)$, i.e. each $\vp \in \,bmo(X)$ defines a bounded linear functional $\Lambda$ on $h^{1}_{g}(X)$ with
\begin{equation}\label{repre}
\Lambda(\vp)=\int f \vp d\mu, 
\end{equation}
for any $ f $ in a dense subset of $h^{1}_{g}(X)$ and $\|\Lambda\| \approx \|\vp\|_{bmo}$. Conversely, each $\Lambda \in (h^{1}_{g})^\ast(X) $ can be represented by a function $\varphi \in \, bmo(X)$, denoted by $\Lambda_{\vp}$, in the sense of \eqref{repre}. Clearly, each $(1,q,T)$-approximate atom can be paired with a function $\vp \in bmo(X)$ as follows (next $ B $ is the ball containing the support of the atom $a$)
\begin{align*}
\left| \int a \vp d\mu \right| &\leq \left| \int a \left(\vp-c_{B}\right) d\mu \right|+|c_{B}|\left| \int a d\mu \right| \\
& \leq  \|a\|_{L^{q}}\left(\int_{B} \left|\vp-c_{B}\right|^{q'} d\mu \right)^{1/q} +|c_{B}|\left| \int a d\mu \right| \\
& \leq  \left(\fint \left|\vp-c_{B}\right|^{q'} d\mu \right)^{1/q} + \frac{2|c_{B}|}{\log(2+T/r_B)} \\
& \leq 3  \|\vp\|_{bmo},
\end{align*}
where the constant $c_{B}$ and the last inequality follow from Lemma 6.1 in \cite{DafniMomYue16}. Summarizing $bmo(X)={(h^1_g)^\ast(X)}$ by \cite[Corollary 7.8]{DafniMomYue16}.

\begin{proposition}\label{prop:2003}
	$ h^1_\#(X) = h^1_g(X)$ with equivalent norms.
\end{proposition}
\begin{proof}
	We start showing $  h^1_g(X) \subset h^1_\#(X)$ continuously. Given a locally integrable function $ f$ in $ h^1_g(X) $, follows by atomic decomposition theorem mentioned before (\cite[Theorem 7.6]{DafniMomYue16}) that there exist a sequence of local $ (1,\infty,T) $-atoms 
	$ \lla{a_j}_{j} $ 
	and a sequence of coefficients $ \lla{\lambda_j}_{j}\in \ell^1(\C)$  such that
	\begin{equation}\label{eq:2000}
			f=\sum_{j=1}^\infty \lambda_j a_j, \qquad \text{in} \,\,{h^1_g (X)} 
		\end{equation}
and consequently in  
	$ L^1(X) $ satisfying $ \sum_{j=1}^{\infty} \abs{\lambda_j} \leq C \norm{f}_{h^1_g} $,
	for some positive constant $ C $ independent of $ f $. 
	Let $ \vp\in bmo $ and denote by $ \Lambda_\vp $ the identification as element in $ (h^{1}_g)^\ast(X) $ associated to $ \vp $ (\cite[Corollary 7.8]{DafniMomYue16}). Note that any element $ f\in h^1_g(X) $ defines an element $ \Gamma_f: bmo(X)\to \C $ given by 
	$ \ip{\Gamma_f,\vp}:= \Lambda_\vp( f )$,
	and also
	\[
	\abs{ \ip{\Gamma_f,\vp} }\leq \norm{\Lambda_\vp}_{(h^1_g)^*} \norm{f}_{h^1_g}\approx \norm{\vp}_{bmo}\norm{f}_{h^1_g},
	\]
	for any $ \vp \in bmo(X) $.
	Thus we have
	\[
	\abs{ \ip{\Gamma_f -\sum_{j=1}^n \lambda_j a_j,\vp} } = \abs{ \Lambda_\vp\pare{f-\sum_{j=1}^n \lambda_ja_j} }  \lesssim \norm{\vp}_{bmo} \Norm{ f - \sum_{j=1}^n \lambda_j a_j}_{h^1_g}
	\]
	for any $ \vp\in bmo(X) $
	that implies
	\[
	\Gamma_f = \sum_{j=1}^\infty \lambda_j a_j, \qquad \text{in }\quad  bmo^*(X),
	\]
	and consequently
	$
	\norm{ \Gamma_f }_{h^1_\#}\leq C\norm{f}_{h^1_g}
	$, as desired.\\

	For the other inclusion,
	let 
	$ F\in  h^1_\#(X) $. Then there exist $ \lla{\lambda_j}_j\in \ell^1(\C) $ and a sequence of $ (1,\infty, T)$-approximate atoms $ \lla{a_j}_{j} $ such that $ F=\sum_{j=1}^\infty \lambda_j a_j $ in $ bmo^*(X) $.
	Note that  $ \lla{\sum_{j=1}^n \lambda_j a_j}_n $ is a Cauchy sequence in $ h^1_g(X) $ from Proposition 7.5 in \cite{DafniMomYue16}.
	By completeness it has to converge to some $ f\in h^1_{g}(X)\hookrightarrow L^1(X) $ continuously with 
			$ \norm{f}_{h^1_g}\leq C\sum_{j=1}^\infty \abs{\lambda_j} $.
	By the arbitrariness of the decomposition of $ F $, we obtain
	\begin{equation}\label{eq:2002} 
			\norm{f}_{h^1_g}\leq C\norm{F}_{h^1_{cw}}.
		\end{equation}
	We claim that $ f $ is well defined. In fact, if $ F=\sum_{j=1}^\infty \beta_jb_j $ in $ bmo^*(X) $ and  $ \sum \beta_jb_j $ converges to some $ \tilde{f} $ in $ h^1_g(X)$-norm, maintaining the notation $ \Lambda_\vp\in h^1_g(X)^\ast $ for any $ \vp\in bmo(X) $ as before, follows by \cite[Corollary 7.8]{DafniMomYue16}
	that
	\begin{align*}
			\Lambda_{\vp} ( \tilde{f} )&= \Lambda_{\vp} ( \lim_{n} \sum_{j=1}^n \beta_j b_j )=\lim_n \sum_{j=1}^n \beta_j\Lambda_\vp (b_j)= \lim_n \sum_{j=1}^n \beta_j \int b_j \vp \, d\mu = \ip{ F,\vp } \\
			&= \lim_n \sum_{j=1}^n \lambda_j \int a_j \vp \, d\mu = \lim_n \sum_{j=1}^n \lambda_j\Lambda_\vp (a_j) =\Lambda_{\vp} ( \lim_{n} \sum_{j=1}^n \lambda_j a_j ) = \Lambda_{\vp} ( f )
		\end{align*}
	for any $ \vp \in bmo(X)$. Follows by identification and duality $bmo(X)=(h^1_g)^\ast(X)$ that $f=\tilde{f}$ almost everywhere. 
	As consequence, $ h^1_\#(X) \subset h^1_g(X)$ continuously.
\end{proof}

The next couple of results are self-improvements of Theorems 7.6 and 7.7 in \cite{DafniMomYue16}, that allow us to avoid the equivalence between norms used in Theorem \ref{teo:boundextenCZO} and implicitly in Theorem \ref{cor:last}. 
The first is a special Calder\'on-Zygmund type decomposition.

\begin{theorem}\label{teo:Dafni16-7.7}
	Given $ f\in L^1_{loc}(X)$,
	$ \alpha>0 $, $ C_0>4{A_0} $ we can write 
	\[
	f=g+b, \qquad b=\sum_{k=1}^\infty b_k
	\]
	for some functions $ g, b_k$ and a sequence of balls $ \lla{B_k}_{k=1}^\infty $ satisfying
	\begin{enumerate}[(i)]
		\item $ \norm{g}_\infty \leq c\alpha $ for some $ c\geq 1 $ depending on $ C_0 $, $ {A_0} $, $ \alpha $, $ C_d $ and $ A^\prime $;\\
		\item $ supp(b_k)\subset B^\ast_k:=C_0B_k $ and 
		\[
		\int b_k\, d\mu = 0, \quad \text{ when  } r(B_k^\ast)<\frac{T}{4(k^\prime)^2};
		\]
		\item
		\[
		(iii.1) \qquad \qquad \norm{b_k}_{L^1}\leq 2c \int_{B^\ast_k} \mathcal{M}_\mathcal{F} f \; d\mu,
		\]
		and
		\[
		(iii.2) \qquad \qquad \norm{b_k}_{L^2}^2\leq 4c^2 \int_{B^\ast_k} (\mathcal{M}_\mathcal{F} f)^2 \; d\mu;
		\]
		\item the balls $ B_k^* $ have bounded overlap and
		\[
		\bigcup B_k^\ast =\lla{ x\in X: M_\mathcal{F}f(x)>\alpha }.
		\]
	\end{enumerate}
\end{theorem}

The novelty here in comparison to Theorem 7.7 in \cite{DafniMomYue16} is the control ($ iii.2 $) that can be proved using the same steps as ($ iii.1 $). 

\begin{theorem}\label{teo:Dafni16-7.6}
	If $ f\in \left(h^1_g\cap L^2\right)(X) $, then there exist a sequence of local $ (1,\infty, T) $-atoms $ \lla{a_j}_{j} $ and a sequence of coefficients $ \lla{\lambda_j}_{j}\in {\ell^1(\C)} $ such that
	\begin{equation}\label{eq:2004}
	f=\sum \lambda_j a_j
	\end{equation}
	with convergence in $ h^1_g(X)$ and $ L^2(X)$. Moreover, $ \sum \abs{\lambda_{j}}\leq C\norm{f}_{h^1_g} $ for some positive constant $ C $ independent of $ f $.
\end{theorem}

\begin{proof}
Since the maximal Hardy-Littlewood operator $ \M $ is bounded in $L^{2}(X)$ and $ \M_\F f(x)\leq \M f(x) $ (see \cite[pp. 202]{DafniMomYue16}) then $ \M_\F f\in L^2(X) $. 
The proof follows the same steps as in \cite[Theorem 7.6]{DafniMomYue16} and the 
	convergence of \eqref{eq:2004} in $ L^2(X) $ follows from (iii.2) in Theorem \ref{teo:Dafni16-7.7}, since using the same notation from mentioned result we have  
	\begin{equation*}
	\|f-g_{j}\|_{L^{2}}=\|b_{j}\|_{L^{2}}\leq \sum_{k=1}^{\infty}\|b_{j}^{k}\|_{L^{2}}\leq {2}c \sum_{k} \int_{(B^{k}_{j})^*}[\mathcal{M}_{\mathcal{F}}(f)]^{2}d\mu \lesssim \int_{\left\{ \mathcal{M}_{\mathcal{F}}(f)>2^{j} \right\}} [\mathcal{M}_{\mathcal{F}}(f)]^{2}d\mu,
	\end{equation*}
	and
\begin{align*}
	\norm{g_j}_{L^2}^2 \leq \int_{U^j} \abs{g_j}^2  d\mu +\int_{F^j} \abs{g_j}^2  d\mu &\lesssim  2^{2j} \mu(\lla{\mathcal{M_\mathcal{F}}f(x)>2^j})+ \int_{\lla{\mathcal{M_F}f(x)\leq 2^j}} [\mathcal{M_F}f(x)]^2 d\mu(x)
\end{align*}
where the constants are independent of $f$ and $j$. Clearly, the terms in the right hand side in the last couple of inequalities go to zero when $j \rightarrow {\pm \infty}$ (for more details see \cite[pp. 210]{DafniMomYue16}).
\end{proof}

Now we are ready to establish the boundedness of Calder\'on-Zygmund operators in $ h^1_g(X)$.

\begin{theorem}\label{teo:boundextenCZOp1g}
	Let $ R $ be an inhomogeneous Calder\'on-Zygmund operator of order $ (\nu,s) $. If there exists $C>0$ such that for any ball $B:=B(x_{B},r_{B})\subset X$ with $r_{B}<T$ we have that $f:=R^{\ast}(1)$ satisfies 
	\begin{equation*}\label{eq:atomtomol3p1g}
			\bigg( \fint_B \abs{f-f_B}^2 d\mu \bigg)^{1/2} \leq C \frac{2}{\log(2+{T}/{r_B})},
	\end{equation*}
	then the operator $ R $ can be extended as a linear bounded operator on $ h^1_g(X) $.
\end{theorem}

\begin{proof}
	Since $ h^{1,2}_{fin}(X)\subset (h^1_g\cap L^2)(X) $ and $ h^{1,2}_{fin}(X)$ is a dense subset of $ (h^1_g(X),\norm{\cdot}_{h^1_g}) $, it will be sufficient to prove that 
	\begin{equation}\label{eq:2006}
		\norm{Rf}_{h^1_g}\leq C\norm{f}_{h^1_g}, \quad \forall \, f\in (h^1_g\cap L^2)(X). 
	\end{equation}
From Theorem \ref{teo:Dafni16-7.6}, given $ f\in (h^1_g\cap L^2)(X) $ consider the decomposition \eqref{eq:2004}. Thus, by the continuity of $ R $ on $ L^2 $, we obtain 
	\begin{equation}\label{eq:2005}
		Rf=\sum_k \lambda_k Ra_j 
	\end{equation}
	with convergence in $ L^2(X) $. We claim that the decomposition in \eqref{eq:2005} converges also in $ h^1_g(X)$-norm. In fact, by Propositions \ref{prop:atomtomol3p1} and \ref{prop:2003} there exists a constant $ C>0 $ such that 
	\begin{equation*}
		\norm{Ra_j}_{h^1_g}\leq C
	\end{equation*}
	for any local $ (1,2,T)$-atom $ a_j $ and then
	\begin{align*}
		\Norm{\sum_{k=1}^n \lambda_jRa_j}_{h^1_g}= \Norm{\mathcal{M_F}\pare{\sum_{k=1}^n \lambda_jRa_j}}_{L^1} \leq \sum_{k=1}^n \abs{\lambda_j} \norm{Ra_j}_{h^1_g} \leq C \sum_{j=1}^n \abs{\lambda_j}.
	\end{align*}
	This shows that the sequence of partial sums $ S_n:=\sum_{k=1}^n \lambda_jRa_j$ is a Cauchy sequence in $ h^1_g(X)$, then it converges to some $ F\in h^1_g(X)$ and
	\begin{equation*}
		\norm{F}_{h^1_g}\leq C \norm{f}_{h^1_g}.
	\end{equation*}
	On the other hand, since $ f\in h^1_{g}(X)\hookrightarrow L^1(X) $ 
we have that the sequence of partial sums $ S_n $ also converges in $ L^1$-norm to $ F $. Taking subsequences of $ S_n $ and from \eqref{eq:2005} we obtain that $ F=Rf $ almost everywhere. This shows \eqref{eq:2005} converges in $ h^1_g(X)$-norm and therefore we can establish \eqref{eq:2006}.
\end{proof}

Now we compare the spaces $h^{1}(X)$ and $h^{1}_{g}(X)$.

\begin{proposition}
Let $ x_1\in X $ and $ \theta\in (0,\infty) $. If $\psi$ is $ \alpha $-H\"{o}lder continuous function supported in a ball $ B(x_{1},r)$  satisfying \eqref{eq:224} then $ \psi\in \mathcal{G}(x_1,r,{\alpha},\theta)$, where $ \alpha $ is the same constant appearing in \eqref{eq:holqm}. Moreover, there exists $C_{\mathcal{F}}^\prime>0$ independent of $\psi$ such that
\begin{equation*}
	\norm{\psi}_{\mathcal{G}(x_1,r,\alpha,\theta)} \leq C_{\mathcal{F}}^\prime. 
\end{equation*} 
\end{proposition}

\begin{proof}
	
	Let $ x \in B(x_1,r) $. Since 
	\begin{equation}\label{eq:4000}
		V(x_1,x)+V_r(x_1)\leq 2V_r(x_1) , \quad \text{ and } \quad d(x_1,x)+r\leq 2r  
	\end{equation} 
	we obtain from the first inequality in \eqref{eq:224} that
	\begin{equation}\label{eq:2108}
		\abs{\psi(x)}\leq \frac{C_\mathcal{F}}{V_{r}(x_{1})}\leq  \frac{2^{\theta+1} C_\mathcal{F}}{ V(x_1,x) + V_r(x_1) }\pare{ \frac{r}{d(x_1,x)+r} }^\theta.
	\end{equation}
Note that \eqref{eq:2108} is trivially valid for $ x\in B(x_1,r)^\complement $, since $\psi(x)=0$.

	Now, let $ x,y\in X $ with 
	\begin{equation}\label{eq:2102}
		d(x,y)\leq (2A_0)^{-1}(r+d(x_1,x)) .	
	\end{equation}
	Firstly, suppose that $ x\in B(x_1,r) $.  
	Then by \eqref{eq:4000} and the second inequality in \eqref{eq:224}, we obtain
	\begin{align}
		\abs{ \psi(x)-\psi(y) }&\leq d(x,y)^\alpha  \frac{C_\mathcal{F}}{r^\alpha V_r(x_1)} \nn \\
		& \leq 2^{1+\theta+\alpha} C_\mathcal{F} \left[ \frac{d(x,y)}{r+d(x_1,x)} \right]^\alpha \frac{1}{ V_r(x_1)+V(x_1,x) }
		\left[\frac{r}{r+ d(x_1,x) }\right]^\theta. \label{eq:2106}
	\end{align}
	Now suppose that $ y\in B(x_1,r) $. From \eqref{eq:2102} we obtain $ d(x_1,x) \leq (2A_0+1)r$ and so by \eqref{upper} we have $ V(x_1,x)\leq A^\prime(2A_0+1)^\gamma V_r(x_1) $. Thereby from the second inequality in \eqref{eq:224},
	we have
	\begin{align}
		\abs{ \psi(x)-\psi(y) }&\leq d(x,y)^\alpha  \frac{C_\mathcal{F}}{r^\alpha V_r(x_1)} \nn \\
		& \leq (2A_0+2)^{\alpha+\theta}  \left[ \frac{d(x,y)}{r+d(x_1,x)} \right]^\alpha \left[\frac{r}{r+ d(x_1,x) }\right]^\theta C_\mathcal{F} \,  \frac{A^\prime(2A_0+1)^{\gamma}+1}{ V_r(x_1)+V(x_1,x) }        \nn \\
		& \leq C \left[ \frac{d(x,y)}{r+d(x_1,x)} \right]^\alpha \frac{1}{ V_r(x_1)+V(x_1,x) }  \left[\frac{r}{r+ d(x_1,x) }\right]^\theta, \nn 
	\end{align}
	with $C'_{\mathcal{F}}:=[A^\prime(2A_0+1)^{\gamma}+1] (2A_0+2)^{\alpha+\theta}C_\mathcal{F}$. Clearly, if $ x,y\in B(x_1,r)^\complement $ then the previous control trivially holds. 

Summarizing the inequalities, we conclude that $ \psi\in \mathcal{G}(x_1,r,\alpha,\theta)$ and moreover
\begin{equation*}
	\norm{\psi}_{\mathcal{G}(x_1,r,\alpha,\theta)} \leq C_{\mathcal{F}}^\prime 
\end{equation*}
uniformly in  $ x_1 $ and $r$.
\end{proof}
A direct consequence of the previous result is the following: if $ \alpha,\theta\in (0,\eta] $ and $ T:=(4A_0^2)^{-1} $, we obtain
\begin{equation}\label{eq:4010}
	\mathcal{M}_{\mathcal{F}}f(x)\leq C_\mathcal{F}^\prime\, f^\ast_0(x)
\end{equation}
for any $ x\in X $ and $ f\in (\mathcal{G}^\eta_0(\alpha,\theta))^\ast \cap L^1_{loc}(X)$. In this way, we state  the next result.
\begin{proposition}\label{propf}
	Let $ T=(4A_0^2)^{-1} $ and $ \alpha,\theta\in(0,\eta] $. Considering $ h^1(X) $ as a subspace of $  (\mathcal{G}^\eta_0(\alpha,\theta))^\ast $, we have $ h^1(X)\cap L^1_{loc}(X) = h^1_g(X) $, with equivalent norms. 
\end{proposition}
\begin{proof}
The continuous inclusion $ h^1(X)\cap L^1_{loc}(X) \subset h^1_g(X) $  follows directly from \eqref{eq:4010}.
		On the other hand, if $ f\in h^1_g(X) $, then $ f\in L^1(X) $ and by \cite[Theorem 7.1]{DafniMomYue16} there exist a sequence of local $ (1,\infty,T) $ atoms $ \lla{a_j}_j $ and a sequence of coefficients $ \lla{\lambda_j}_j \in \ell^1(\mathbb{C})$ such that $ f=\sum_j \lambda_j a_j $ in $ L^1(X) $. Then it follows by \eqref{lq} that the convergence of $ f=\sum_j \lambda_j a_j $ also holds in $ (\mathcal{G}^\eta_0(\alpha,\theta))^\ast $, and thereby from Proposition 4.3 in \cite{HeYangWen21} we obtain $ f\in h^1(X) $ with $ \norm{f}_{h^1}\leq C \norm{f}_{h^1_g} $, for a positive constant $ C $ independent of $ f $. Consequently, $ h^1_g(X)\subset h^1(X) $ continuously.
\end{proof}

\noindent {\textbf{Acknowledgments:} The authors would like to thank Marius Mitrea and Emmanuel Russ for suggesting this problem and the referee for their careful reading and useful suggestions and comments.}\\

%\noindent {\textbf{Data Availability.} This work has no associated data.\\

%\noindent {\textbf{Conflict of interest.} On behalf of all authors, the corresponding author states that there is no conflict of interest.\\

\end{document}